\documentclass{arxiv}
\usepackage{amsmath}
\usepackage{amssymb}
\usepackage{amsfonts}
\usepackage{amstext}
\usepackage{amsopn}
\usepackage{amsxtra}
\usepackage[latin1]{inputenc}
\usepackage{dsfont}
\usepackage{mathrsfs}
\usepackage{pst-all}
\usepackage{graphicx}
\newcommand\R{{\ensuremath {\mathbb R} }}
\newcommand\C{{\ensuremath {\mathbb C} }}

\newcommand\1{{\ensuremath {\mathds 1} }}
\renewcommand\phi{\varphi}

\newcommand{\gH}{\mathfrak{H}}

\newcommand{\bH}{\mathbb{H}}
\newcommand{\bS}{\text{\textnormal{\textbf{S}}}}
\newcommand{\bN}{\text{\textnormal{\textbf{N}}}}
\newcommand{\bF}{\text{\textnormal{\textbf{F}}}}

\newcommand{\gS}{\mathfrak{S}}

\renewcommand{\to}{\rightarrow}
\newcommand{\cP}{\mathcal{P}}
\newcommand{\cN}{\mathcal{N}}

\newcommand{\cF}{\mathcal F}

\newcommand{\cE}{\mathcal{E}}

\newcommand{\cC}{\mathcal{C}}

\newcommand{\cK}{\mathcal{K}}

\newcommand{\lmax}{{\ell_{\rm max}}}

\newcommand\ii{{\ensuremath {\infty}}}
\newcommand\pscal[1]{{\ensuremath{\left\langle #1 \right\rangle}}}
\newcommand{\norm}[1]{ \left| \! \left| #1 \right| \! \right| }

\DeclareMathOperator{\tr}{{\rm Tr}}

\renewcommand{\leq}{\leqslant}
\renewcommand{\geq}{\geqslant}

\begin{document}
\date{\today}
\title{\Large A Numerical Perspective\\ \bigskip  on Hartree-Fock-Bogoliubov Theory}
\authormark{Mathieu LEWIN \& Séverine PAUL}
\runningtitle{A Numerical Perspective on Hartree-Fock-Bogoliubov Theory}

\author{Mathieu LEWIN}
\address{CNRS and Laboratoire de Mathématiques (CNRS UMR 8088)\\ Universit{\'e} de Cergy-Pontoise, 95 000 Cergy-Pontoise - France.\\ 
Email: \email{mathieu.lewin@math.cnrs.fr}}

\author{Séverine PAUL}
\address{Laboratoire de Mathématiques (CNRS UMR 8088)\\ Universit{\'e} de Cergy-Pontoise, 95 000 Cergy-Pontoise - France.\\ 
Email: \email{severine.paul@u-cergy.fr}}

\maketitle

\bigskip

\begin{abstract}
The method of choice for describing attractive quantum systems is Hartree-Fock-Bogoliubov (HFB) theory. This is a nonlinear model which allows for the description of \emph{pairing effects}, the main explanation for the superconductivity of certain materials at very low temperature.

This paper is the first study of Hartree-Fock-Bogoliubov theory from the point of view of numerical analysis. We start by discussing its proper discretization and then analyze the convergence of the simple fixed point (Roothaan) algorithm. Following works by Cancès, Le Bris and Levitt for electrons in atoms and molecules, we show that this algorithm either converges to a solution of the equation, or oscillates between two states, none of them being a solution to the HFB equations. We also adapt the Optimal Damping Algorithm of Cancès and Le Bris to the HFB setting and we analyze it.

The last part of the paper is devoted to numerical experiments. We consider a purely gravitational system and numerically discover that pairing always occurs. We then examine a simplified model for nucleons, with an effective interaction similar to what is often used in nuclear physics. In both cases we discuss the importance of using a damping algorithm.

\medskip 

\noindent{\scriptsize\copyright~2012 by the authors. This paper may be reproduced, in its entirety, for non-commercial purposes.}
\end{abstract}

\bigskip

\tableofcontents

\bigskip

\section{Introduction}

Hartree-Fock-Bogoliubov (HFB) theory is the method of choice for describing some special features of attractive fermionic quantum systems~\cite{RinSch-80}. It is a generalization of the famous Hartree-Fock (HF) method~\cite{LieSim-77} used in quantum chemistry. It also generalizes the Bardeen-Cooper-Schrieffer (BCS) theory of superconductivity~\cite{BarCooSch-57} which was invented in 1957 to explain the complete loss of resistivity of certain materials at very low temperature. In 1958 Bogoliubov realized in~\cite{Bogoliubov-58} that the BCS theory was actually very similar to his previous works~\cite{Bogoliubov-47,Bogoliubov-47b,Bogoliubov-47c} on the superfluidity of certain bosonic systems. He adapted it to fermions, leading to a model that is now called Hartree-Fock-Bogoliubov, and which is the main interest of this article. 

In Hartree-Fock-Bogoliubov theory the state of the system is completely determined by two operators~\cite{BacLieSol-94}. The \emph{one-particle density matrix} $\gamma$ is the same as in Hartree-Fock theory~\cite{LieSei-09}, whereas the \emph{pairing density matrix} $\alpha$ describes the Cooper pairing effect. This effect can only hold in attractive systems. When the interaction potential is positive, the energy is decreased by replacing $\alpha$ by $0$.

One of the most interesting questions in HFB theory is precisely the existence of pairing, that is, the non-vanishing of the matrix $\alpha$ for a minimizer. This question has been settled in the simpler translation-invariant BCS theory~\cite{BilFan-68,Vansevenant-85,Yang-91,LeoYan-00,FraHaiNabSei-07,HaiHamSeiSol-08,HaiSei-08} and in some cases for the translation-invariant Hubbard model~\cite{BacLieSol-94}. But it remains completely open for general attractive systems with a few particles, like those encountered in nuclear physics. In~\cite{LenLew-10}, the first author of this paper has shown with Lenzmann the existence of HFB minimizers for a purely Newtonian system of $N$ fermions, but it is not yet known if $\alpha\neq0$. One of the purpose of this work is to answer this question numerically.

The HFB energy is a nonlinear function of $\gamma$ and $\alpha$. Because of nonlinearity, minimizing this functional on a computer is not an easy task. The simpler Hartree-Fock model in which $\alpha=0$ is now well understood from the point of view of numerical analysis~\cite{LeBris-05,CanDefKutLeBMad-03}, even though most authors have concentrated their attention to the special case of electrons in an atom or a molecule. Cancès and Le Bris have studied in~\cite{CanBri-00a,CanBri-00} the simpler fixed point algorithm called the Roothaan algorithm~\cite{Roothaan-51} and they have shown that this algorithm either converges or oscillates between two points, none of them being the solution of the HF equation. This result was recently improved by Levitt~\cite{Levitt-12}. Cancès and Le Bris have also proposed a new algorithm called \emph{Optimal Damping}, which is now used in several chemistry programs. It is based on the fact that one can freely minimize the energy over mixed HF states instead of pure HF states, by Lieb's variational principle~\cite{Lieb-81}. 

Because the HFB model is an extension of HF theory, it is natural to believe that these ideas can be applied to the HFB case as well. In particular, under appropriate assumptions on the interaction potential, Bach, Fröhlich and Jonsson have recently shown in~\cite{BacFroJon-09} an HFB equivalent of Lieb's variational principle. Up to some difficulties that will be explained later, we will show in this paper that the previously mentioned results can indeed be transposed to the Hartree-Fock-Bogoliubov model.

The paper is organized as follows. In the next section we quickly recall the basic formulation of Hartree-Fock-Bogoliubov theory. Then, in Section~\ref{sec:discretized}, we derive the discretized HFB equations and we prove that, in the limit of a large Galerkin basis set, the discretized solution converges to the true solution. We also discuss at length the possible symmetries of the system and we formulate the theory when these symmetries are taken into account. 

In Section~\ref{sec:algos} we study the HFB Roothaan algorithm and we prove that it either converges to a solution of the HFB equation, or oscillate between two points, none of them being a solution of the equations. We then introduce an equivalent of the Optimal Damping Algorithm of Cancès and Le Bris, which is based on an optimization in the set of mixed HFB states.

Section~\ref{sec:num} is devoted to the presentation of some numerical results. We first consider the purely gravitational model studied by Lenzmann and Lewin~\cite{LenLew-10} and we numerically discover that there is always pairing. Then, we introduce a simplified model for nucleons, for which we as well present some preliminary numerical results. We particularly discuss the importance of using a damped algorithm instead of a simple fixed point method, a fact which has already been noticed in nuclear physics~\cite{DecGog-80}. Our approach could help in improving the existing numerical techniques.

\bigskip

\noindent\textbf{Acknowledgement.} The authors would like to thank Laurent Bruneau and Julien Sabin for useful discussions. They acknowledge financial support from the French Ministry of Research (ANR-10-BLAN-0101) and from the  European Research Council under the European Community's Seventh Framework Programme (FP7/2007-2013 Grant Agreement MNIQS 258023).

\section{A quick review of Hartree-Fock-Bogoliubov theory}
\subsection{Hartree-Fock-Bogoliubov states and their energy}

We consider a system composed of $N$ identical fermions, described by the many-body Hamiltonian
\begin{equation}
H(N)=\sum_{j=1}^N T_j + \sum_{1\leq k<\ell\leq N} W_{k\ell},
\label{eq:N-body-Hamiltonian}
\end{equation}
acting on the fermionic $N$-body space $\gH^N=\bigwedge_1^N\gH$, where $\gH$ is the space for one particle. Here, $T:\gH\to\gH$ is a one-body operator and $W:\gH^2\to\gH^2$ accounts for the interactions between the particles. We use the notation $T_j$ for the operator $T$ which acts on the $j$th component of the tensor product $\gH^N=\bigwedge_1^N\gH$, that is $T_j=1\otimes\cdots \otimes T\otimes\cdots\otimes1$, and a similar convention for $W_{k\ell}$.

Most of what follows is valid in an abstract setting.
However, for the sake of simplicity, in the whole paper we will restrict ourselves to the special case of nonrelativistic fermions with $q$ internal degrees of freedom, moving in $\R^3$ ($q=2$ for spin-$1/2$ particles like electrons). We also assume that no external force is applied, and that their interaction is translation-invariant. Then,
in units where $m=1/2$ and $\hbar=1$, we have
$$\gH=L^2(\R^3,\C^q),\qquad T=-\Delta,\qquad W_{k\ell}=W(x_k-x_\ell).$$
The $N$-body space $\gH^N=\bigwedge_1^NL^2(\R^3,\C^q)$ consists of wave functions $\Psi(x_1,\sigma_1,...,x_N,\sigma_N)$ which are antisymmetric with respect to exchanges of the variables $(x_i,\sigma_i)$. In principle $W(x_k-x_\ell)$ is also a function of the two internal variables $\sigma_k,\sigma_\ell\in\{1,...,q\}$ of the particles $k$ and $\ell$. Again for simplicity, we will assume that $W$ only depends on the space variable $x_k-x_\ell$. Finally, we make the assumption that $W$ is smooth and decays fast enough at infinity to ensure that $H$ is bounded from below. To make this more explicit, we assume in the whole paper that
\begin{equation}
 W=W_1+W_2\in L^p(\R^3)+L^q(\R^3)\quad\text{for some $2\leq p\leq q<\ii$.}
\label{eq:assumption_W_general}
\end{equation}
Sometimes we will make more precise assumptions on $W$.

We are interested in the case where $W$ is attractive ($W\leq0$), or at least partially attractive ($W\leq0$ on a set of measure non zero).
By translation invariance, the Hamiltonian $H(N)$ has no ground state (that is, the bottom of its spectrum cannot be an eigenvalue). But it may have one once the center of mass is removed, if $W$ is sufficiently negative.

In a nonlinear model approximating the many-body problem above, there could be a ground state, even if the system is translation-invariant. Of course, translation invariance is not lost and there are then infinitely many ground states, obtained by translating the system arbitrarily. In Hartree-Fock theory~\cite{LieSim-77,BacLieSol-94}, such breaking of symmetry is known to occur for instance when $W(x)=-1/|x|$ is a purely gravitational interaction and $T=-\Delta$ (nonrelativistic), or $T=\sqrt{1-\Delta}-1$ (pseudo-relativistic), see~\cite{LenLew-10,Lewin-11} and Theorem~\ref{thm:existence} below.

For attractive systems, it is often convenient to allow for another symmetry breaking, namely that of \emph{particle number}. This means that the fixed particle number $N$ is replaced by an operator $\cN$ whose eigenvalues are $0,1,2,...$. Only the \emph{average particle number} is well defined for a quantum state. The classical way to define $\cN$ is to introduce the fermionic Fock space
$$\cF=\C\oplus\bigoplus_{n\geq1}\gH^n,$$
which gathers all the possible $n$-particle subspaces in a direct sum. A (pure) quantum state in $\cF$ is a vector $\Psi=\psi_0\oplus\psi_1\oplus\cdots$ which is normalized in the sense that
$$\norm{\Psi}_\cF^2=|\psi_0|^2+\sum_{n\geq1}\norm{\psi_n}_{\gH^n}^2=1.$$
The \emph{average particle number} is the diagonal operator
$$\cN:=0\oplus\bigoplus_{n\geq1}n,$$
such that the average number of particles in a state $\Psi$ is given by the formula
$$\pscal{\Psi,\cN\Psi}=\sum_{n\geq1} n\norm{\psi_n}^2_{\gH^n}.$$
Instead of imposing that $\Psi\in\gH^N$ which is equivalent to $\Psi$ being an eigenvector of $\cN$, $\cN\Psi=N\Psi$, we will only fix the average particle number of $\Psi$:
$$\pscal{\Psi,\cN\Psi}=N.$$
Allowing to have $\psi_n\neq0$ for $n\neq N$ is useful to describe some physical properties of attractive systems. In most practical cases it is expected that the variance $\sum_{n\geq0}(n-N)^2\norm{\psi_n}^2_{\gH^n}$ will be quite small, i.e. that $\Psi$ will live in a neighborhood of $\gH^N$. 

Similarly to the particle number operator $\cN$, the many-body Hamiltonian $H(N)$ is now replaced by a many-body Hamiltonian $\bH$ on Fock space
\begin{equation}
\bH:=0\oplus\bigoplus_{n\geq1}H(n)
\label{eq:def_bH} 
\end{equation}
which is nothing else but the diagonal operator which coincides with $H(n)$ on each $n$-particle subspace. We will not discuss here the problem of defining $\bH$ as a self-adjoint operator on $\cF$.

\medskip

The Hartree-Fock-Bogoliubov (HFB) model generalizes the well-known Hartree-Fock (HF) method and it allows for breaking of particle number in a very simple fashion. The method consists in restricting the many-body Hamiltonian $\bH$ on $\cF$ to a special class of states called \emph{Hartree-Fock-Bogoliubov states} (or \emph{quasi-free states}), which are completely characterized by their one-particle density matices~\cite{BacLieSol-94}. 

Let us recall that a state in Fock space has two one-particle density matrices, instead of one in usual HF theory. These are two operators $\gamma:\gH\to\gH$ and $\alpha:\gH\to\gH$, which are defined by means of creation and annihilation operators by the relations~\cite{BacLieSol-94}
$$\pscal{\Psi,a^\dagger(f)a(g)\Psi}_\cF=\pscal{g,\gamma f}_\gH,\qquad \pscal{\Psi,a(f)a(g)\Psi}_\cF=\pscal{g,\alpha \overline{f}}_\gH.$$
When $\Psi$ lives in a particular $N$-particle subspace $\gH^N$, then $a(f)a(g)\Psi\in\gH^{N-2}$ hence $\pscal{\Psi,a(f)a(g)\Psi}_\cF=0$ for all $f,g\in\gH$ and the matrix $\alpha$ vanishes. However for a general state $\Psi\in\cF$, one can have $\alpha\neq0$.

The two operators $\gamma$ and $\alpha$ satisfy several constraints. First, we have $\gamma^*=\gamma$, $0\leq\gamma\leq 1$ (in the sense of operators) and $\tr\gamma=\pscal{\Psi,\cN\Psi}=N$, for the one-particle matrix $\gamma$. On the other hand, the so-called \emph{pairing matrix} $\alpha$ satisfies $\alpha^T=-\alpha$. Its kernel $\alpha(x,\sigma,x',\sigma')=-\alpha(x',\sigma',x,\sigma)$ can thus be seen as an antisymmetric two-body wavefunction in $\gH^2$. It is interpreted as describing pairs of virtual particles, called \emph{Cooper pairs}. 

It is well known that a quantum state $\Psi\in\gH^N$ such that $(\gamma_\Psi)^2=\gamma_\Psi$ is necessarily a Slater determinant (that is, a Hartree-Fock state). The same is true for states in Fock space. Consider a pair $(\gamma,\alpha)$ which is such that
\begin{equation}
\Gamma^2=\Gamma,\quad \text{with}\quad \Gamma:=\begin{pmatrix}
\gamma&\alpha\\ \alpha^*&1-\overline{\gamma}
\end{pmatrix}
\label{eq:def_big_Gamma}
\end{equation}
on $\gH\oplus\gH$. Hence we have for instance $\alpha\alpha^*=\gamma-\gamma^2$. Then there exists a unique state $\Psi$ in $\cF$ which has $\gamma$ and $\alpha$ as density matrices. This state has the property that any observable can be computed using only $\gamma$ and $\alpha$, by Wick's formula (see Thm 2.3 in~\cite{BacLieSol-94}). The quantum states obtained by considering projections $\Gamma$ are called Hartree-Fock-Bogoliubov states and they generalize usual Hartree-Fock states. When $\alpha\equiv0$, then $\gamma=\sum_{j=1}^N|\phi_j\rangle\langle\phi_j|$ is a rank-$N$ projection and the corresponding state is the usual Slater determinant 
$$\Psi=\phi_1\wedge\cdots\wedge\phi_N=\frac{1}{\sqrt{N!}}\det\big(\phi_i(x_j,\sigma_j)\big).$$
When $\alpha\neq0$, $\Psi$ can be obtained by applying a Bogoliubov rotation to the vacuum but we will not explain this further. For the present work, we will only need the formula of the total energy, in terms of $\gamma$ and $\alpha$: 
\begin{align}
\pscal{\Psi,\bH\Psi}_\cF&=\sum_{n\geq0}\pscal{\psi_n,H(n)\psi_n}_{\gH^n}\nonumber\\
&=\tr(-\Delta)\gamma+\frac12\int_{\R^3}\int_{\R^3}W(x-y)\Big(\rho_\gamma(x)\rho_\gamma(y)-|\gamma(x,y)|^2+|\alpha(x,y)|^2\Big)dx\,dy\nonumber\\
&:=\cE(\gamma,\alpha)
\label{eq:def_energy}
\end{align}
where $\rho_\gamma(x)=\tr_{\C^q}(\gamma(x,x))$ is the density of particles in the system. The terms in the double integral are respectively called the \emph{direct}, \emph{exchange} and \emph{pairing} terms. Taking $\alpha\equiv0$ one recovers the usual Hartree-Fock energy which has been studied by many authors~\cite{LieSim-77,Lions-87,Bach-92,BacLieSol-94}.
Our main goal in this paper is to investigate the minimization of the more complicated nonlinear functional $\cE(\gamma,\alpha)$, when $\gamma$ and $\alpha$ are submitted to the above constraints, and its numerical implementation. We will show below that the energy $\cE$ is well defined in an appropriate function space, under our assumption~\eqref{eq:assumption_W_general} on $W$.

Note that the variance of the particle number for a HFB state $\Psi$ in Fock space can be expressed only in terms of $\alpha$ by
$$\pscal{\Psi,\big(\cN-\pscal{\Psi,\cN\Psi}_{\cF}\big)^2\Psi}_\cF=\sum_{n\geq0}(n-N)^2\norm{\psi_n}^2_{\gH^n}=2\,\tr_{\gH}(\alpha\alpha^*),$$
see Lemma 2.7 in~\cite{BacLieSol-94}. The spreading of the HFB state over the different spaces $\gH^n$ is therefore determined by the Hilbert-Schmidt norm of the pairing matrix $\alpha$. We recover the fact that an HFB state has a given particle number if and only if its pairing matrix $\alpha$ vanishes.

\subsection{Pure vs mixed states}
In our previous description, we have only considered pure states in Fock space, that is states given by a normalized vector $\Psi\in\cF$. For practical purposes, it is very convenient to extend the model to mixed states, which are nothing else but convex combinations of pure states, given by a (many-body) density matrix in $\cF$
$$D=\sum_{j}\lambda_j|\Psi_j\rangle\langle\Psi_j|\quad \text{with}\quad \lambda_j\geq0,\quad \sum_j\lambda_j=1,\quad \pscal{\Psi_i,\Psi_j}_\cF=\delta_{ij}.$$
The average particle number and energy are then given by the formulas
$$\tr_\cF\big(\cN D)=\sum_{j}\lambda_j\pscal{\Psi_j,\cN\Psi_j},\qquad \tr_\cF\big(\bH D)=\sum_{j}\lambda_j\pscal{\Psi_j,\bH\Psi_j}.$$
Resorting to mixed states is mandatory at positive temperature, when the equilibrium state of the system will actually always be a mixed state. But it is also very useful at zero temperature, even if the true ground state is a pure state. We will recall later the practical advantages of using mixed HFB states.

Similarly to what we have explained in the previous section, there is a class of \emph{mixed} Hartree-Fock-Bogoliubov states, which are completely characterized by their one-particle density matrices $\gamma$ and $\alpha$. The latter now satisfy the constraint
\begin{equation}
\begin{pmatrix}
0&0\\ 0&0\end{pmatrix}
\leq 
\Gamma:=\begin{pmatrix}
\gamma&\alpha\\ \alpha^*&1-\overline{\gamma}\end{pmatrix}
\leq
\begin{pmatrix}
1&0\\ 0&1\end{pmatrix}
\label{eq:relaxed-constraint}
\end{equation}
on $\gH\oplus\gH$, which is nothing else but the relaxation of the constraint $\Gamma^2=\Gamma$ of pure states.
The set of one-particle density matrices of mixed HFB states is therefore a convex set, whose extremal points are the density matrices of pure HFB states. One should remember that the set of mixed HFB states is \emph{not} the convex hull of HFB pure states, however. The relation between the density matrices $(\gamma,\alpha)$ and the corresponding HFB states in $\cF$ is highly nonlinear.

The energy of a mixed HFB state described by the density matrices $(\gamma,\alpha)$ is given by the same formula~\eqref{eq:def_energy} as for pure states. Hence, minimizing this energy under the relaxed constraint~\eqref{eq:relaxed-constraint} is equivalent to minimizing the full quantum energy over all mixed HFB states. The natural question arises whether a minimizer, when it exists, is automatically a pure state. The answer to this question is positive in many situations, as we will see below.

\bigskip

Before turning to the comparison between the minimization among pure and mixed states, we first introduce the variational sets on which the energy is well defined. The sets of all pure and mixed HFB states with finite kinetic energy are respectively given by 
\begin{equation}
\cP:=\left\{(\gamma,\alpha)\in\gS_1(\gH)\times\gS_2(\gH)\ :\ \alpha^T=-\alpha,\ \Gamma=\Gamma^*=\Gamma^2,\ \tr(-\Delta)\gamma<\ii\right\}
\label{eq:def_set-pure} 
\end{equation}
and
\begin{equation}
\cK:=\left\{(\gamma,\alpha)\in\gS_1(\gH)\times\gS_2(\gH)\ :\ \alpha^T=-\alpha,\ 0\leq \Gamma=\Gamma^*\leq1_{\gH\oplus\gH},\ \tr(-\Delta)\gamma<\ii\right\},
\label{eq:def_set-mixed} 
\end{equation}
(the matrix $\Gamma$ is the one appearing in~\eqref{eq:def_big_Gamma}). 
Here $\gS_1(\gH)$ and $\gS_2(\gH)$ denote the spaces of trace-class and Hilbert-Schmidt operators~\cite{Simon-77}. The expression $\tr(-\Delta)\gamma$ is to be understood in the sense of quadratic forms, that is
$$\tr(-\Delta)\gamma=\sum_{k=1}^3\tr(p_k\gamma p_k)\in[0,+\ii],\qquad \text{with}\quad p_k=-i\partial_{x_k}.$$

In practice we want to fix the total average number of particles. For this reason we also define the constrained sets 
\begin{equation}
\cP(N):=\left\{(\gamma,\alpha)\in\cP\ :\ \tr\gamma=N\right\}
\label{eq:def_set-pure_N} 
\end{equation}
and
\begin{equation}
\cK(N):=\left\{(\gamma,\alpha)\in\cK\ :\ \tr\gamma=N\right\},
\label{eq:def_set-mixed_N} 
\end{equation}
of pure and mixed states with average particle number $N$. In practice $N$ is an integer but it is convenient to allow any non-negative real number.

The following lemma says that the energy is a well-defined functional on the largest of the above sets $\cK$, and that it is bounded from below on $\cK(N)$ for any $N\geq0$.

\begin{lemma}[The HFB energy is bounded-below on $\cK(N)$]\label{lem:bd-below}
When $W=W_1+W_2\in L^p(\R^3)+L^q(\R^3)$ with $2\leq p\leq q<\ii$, then $\cE(\gamma,\alpha)$ is well defined for any $(\gamma,\alpha)\in\cK$. It also satisfies a bound of the form
\begin{equation}
\forall(\gamma,\alpha)\in\cK,\qquad \cE(\gamma,\alpha)\geq \frac12\tr(-\Delta)\gamma-C(N)
\label{eq:bound-below-energy}
\end{equation}
for some constant $C(N)$ depending only on $N=\tr(\gamma)$.
\end{lemma}

\begin{proof}
The assumption that $W=W_1+W_2\in L^p(\R^3)+L^q(\R^3)$ with $2\leq p\leq q<\ii$ implies that $W$ is relatively form-bounded with respect to the Laplacian, with relative bound as small as we want~\cite{Davies}. This means $|W|\leq \epsilon(-\Delta)+C_\epsilon$ in the sense of quadratic forms, for all $\epsilon>0$ and for some constant $C_\epsilon$. This can now be used to verify that the energy is well defined under the assumption that $\tr(-\Delta)\gamma<\ii$. First, we have for the direct term
\begin{equation*}
\int_{\R^3}\int_{\R^3}|W(x-y)|\,\rho_\gamma(x)\rho_\gamma(y)\,dx\,dy \leq \epsilon N \int_{\R^3}|\nabla\sqrt{\rho_\gamma}|^2+C_\epsilon\,N^2 
\leq \epsilon N \tr(-\Delta)\gamma+C_\epsilon\,N^2,
\end{equation*}
where in the last line we have used the Hoffmann-Ostenhof inequality~\cite{Hof-77}, 
\begin{equation}
\int_{\R^3}|\nabla\sqrt{\rho_\gamma}|^2\leq\tr(-\Delta)\gamma.
\label{eq:Hoffmann-Ostenhof} 
\end{equation}
The exchange term is bounded similarly by applying the inequality $|W|\leq \epsilon(-\Delta)+C_\epsilon$ in $x$ with $y$ fixed:
\begin{align*}
\int_{\R^3}\int_{\R^3}|W(x-y)|\,|\gamma(x,y)|^2\,dx\,dy&\leq  \epsilon\int_{\R^3}\int_{\R^3}|\nabla_x\gamma(x,y)|^2\,dx\,dy+C_\epsilon \int_{\R^3}\int_{\R^3}|\gamma(x,y)|^2\,dx\,dy\\
&=\epsilon\tr(-\Delta)\gamma^2+C_\epsilon\tr\gamma^2\leq \epsilon\tr(-\Delta)\gamma+C_\epsilon N,
\end{align*}
since $\gamma^2\leq\gamma$. Similarly we have, since $\alpha\alpha^*\leq \gamma-\gamma^2\leq\gamma$,
\begin{equation*}
\int_{\R^3}\int_{\R^3}|W(x-y)|\,|\alpha(x,y)|^2\,dx\,dy\leq \tr\big(\epsilon(-\Delta)+C_\epsilon\big)\alpha\alpha^*\leq \epsilon\tr(-\Delta)\gamma+C_\epsilon N.
\end{equation*}
All this shows that all the terms in the energy are well defined when $(\gamma,\alpha)\in\cK(N)$. Also, we have 
\begin{equation}
\cE(\gamma,\alpha)\geq \big(1-\epsilon-\epsilon N/2\big)\tr(-\Delta)\gamma-C_\epsilon N-C_\epsilon N^2/2.
\label{eq:lower_estimate} 
\end{equation}
Taking $\epsilon=1/(2+N)$ finishes the proof.
\end{proof}

Lemma~\ref{lem:bd-below} allows us to define the minimization problems for pure and mixed states as follows:
\begin{equation}
I(N):=\inf_{(\gamma,\alpha)\in\cK(N)}\cE(\gamma,\alpha),
\label{eq:def_I_N}
\end{equation}
\begin{equation}
J(N):=\inf_{(\gamma,\alpha)\in\cP(N)}\cE(\gamma,\alpha).
\label{eq:def_J_N}
\end{equation}
Of course we have $J(N)\geq I(N)$ since $\cP(N)\subset\cK(N)$. In many cases, we have that $I(N)=J(N)$ and that any minimizer, when it exists, is automatically a pure HFB state. We give two results in the literature going in this direction. The first deals with \emph{purely repulsive interactions} and it is Lieb's famous variational principle~\cite{Lieb-81} (see also Thm. 2.11 in~\cite{BacLieSol-94}).

\begin{theorem}[Lieb's HF Variational Principle~\cite{Lieb-81}]\label{thm:Lieb}
Assume that 
$$W\geq0$$ 
and let $N$ be an integer. Then for any $(\gamma,\alpha)\in\cK(N)$, there exists $(\gamma',0)\in\cP(N)$ such that 
$$\cE(\gamma,\alpha)\geq \cE(\gamma',0).$$
In particular, we have $I(N)=J(N)$.

If $W>0$ a.e., then any minimizer for $I(N)$, when it exists, is necessarily of the form $(\gamma',0)$ with $(\gamma')^2=\gamma'$.
\end{theorem}

We see that for repulsive interactions, $W>0$, there is never pairing ($\alpha\equiv0$) and the ground state is always a pure HF state, that is, a Slater determinant. The fact that, in HF theory, one can minimize over mixed states and get the same ground state energy is very important from a numerical point of view. This was used by Cancès and Le Bris~\cite{CanBri-00a,CanBri-00} to derive well-behaved numerical strategies, to which we will come back later in Section~\ref{sec:ODA}.

For \emph{purely attractive interactions}, the following recent result of Bach, Fröhlich and Jonsson~\cite{BacFroJon-09} is relevant.

\begin{theorem}[HFB Constrained Variational Principle~\cite{BacFroJon-09}]\label{thm:variational_constrained}
Assume that the number of spin states is $q=2$ (spin-$1/2$ fermions), and that $W$ can be decomposed in the form
\begin{equation}
W(x-y)=-\int_\Omega d\mu(\omega) \1_{A_\omega}(x)\,\1_{A_\omega}(y)
\label{eq:form_W_BFJ}
\end{equation}
on a given measure space $(\Omega,\mu)$, with $A_\omega$ a measurable family of bounded domains in $\R^3$.
Let $N$ be any positive real number.
Then for any $(\gamma,\alpha)\in\cK(N)$, we have
\begin{equation}
\cE(\gamma,\alpha)\geq \cE(\gamma',\alpha'),
\end{equation}
with
\begin{equation}
\gamma'=g\otimes \begin{pmatrix}
1 &0\\ 0&1
\end{pmatrix},\qquad \alpha'=\sqrt{g(1-g)}\otimes\begin{pmatrix}
0 &1\\ -1&0
\end{pmatrix}
\label{eq:form_gamma_alpha} 
\end{equation}
(the matrices act on the spin variables), and 
$$g=g^T=\overline{g}=\frac{\gamma_{\uparrow\uparrow}+\gamma_{\downarrow\downarrow}+\overline{\gamma_{\uparrow\uparrow}}+\overline{\gamma_{\downarrow\downarrow}}}{4}.$$
This HFB state is pure: $(\gamma',\alpha')\in\cP(N)$. In particular, we have $I(N)=J(N)$.

Furthermore, if $W<0$ a.e., then any minimizer for $I(N)$, when it exists, is necessarily of the previous form.
\end{theorem}

\begin{remark}\label{rmk:N_odd}
Note that $N$ does not have to be an even integer in this result. Since $\tr_{L^2(\R^3)}(g)=N/2$, the operator $g$ must have one eigenvalue different from 1 when $N$ is odd, and it follows that $\alpha\neq0$ in this special case.
\end{remark}

In Theorem~\ref{thm:variational_constrained}, we have decomposed the operator $\gamma$ acting on $L^2(\R^3\times\{\uparrow,\downarrow\},\C)$ according to the spin variables as follows:
$$\gamma=\begin{pmatrix}
\gamma_{\uparrow\uparrow}&\gamma_{\downarrow\uparrow}\\
\gamma_{\uparrow\downarrow}&\gamma_{\downarrow\downarrow}
\end{pmatrix}.$$
Theorem~\ref{thm:variational_constrained} says that when $W$ satisfies~\eqref{eq:form_W_BFJ}, one can minimize over states which are pure, real, and have a simple spin symmetry. The antisymmetry of $\alpha$ is only contained in the spin variables, hence the Cooper pairs are automatically in a singlet state.
Of course, one can express the total energy only in terms of the real operator $g$, as follows
\begin{multline*}
\cE(\gamma',\alpha')=2\;\tr_{L^2(\R^3)}(-\Delta)g\\
+\int_{\R^3}\int_{\R^3}W(x-y)\Big(2\rho_g(x)\rho_g(y)-|g(x,y)|^2+|\sqrt{g(1-g)}(x,y)|^2\Big). 
\end{multline*}
In practice it will be more convenient to keep a pairing term $a(x,y)$ not \emph{a priori} related to $g$ and to optimize over both $g$ and $a$, that is, to consider mixed states. When $W$ satisfies the assumptions of the theorem, any ground state will automatically lead to $a=\pm\sqrt{g(1-g)}$.

Let us conclude our comments on Theorem~\ref{thm:variational_constrained}, by noticing that several simple attractive potentials $W$ can be written in the form~\eqref{eq:form_W_BFJ}. For instance the Fefferman-de~la~Llave formula~\cite{FefLla-86}
$$\frac1{|x-y|}=\frac1\pi\int_0^\ii\frac{dr}{r^5}\int_{\R^3}dz\; \1_{B(z,r)}(x)\,\1_{B(z,r)}(y)$$
shows that a simple Newtonian interaction $W(x-y)=-|x-y|^{-1}$ is covered (here $\1_{B(z,r)}$ is the characteristic function of the ball centered at $z$, of radius $r$). Hainzl and Seiringer showed in~\cite{HaiSei-02} that any smooth enough radial function $W$ can be written in the form
$$W(x-y)=\int_0^\ii\,dr\, g(r)\int_{\R^3}dz\; \1_{B(z,r)}(x)\,\1_{B(z,r)}(y)$$
for some explicit function $g$, whose sign can easily be studied.

\subsection{Existence results and properties of minimizers}

Before turning to the discretization and the numerical study of the HFB minimization problem, we make some comments on the existence of minimizers. In addition to the infinite dimension and the nonlinearity of the model, an important difficulty is the invariance under translations. For instance, there are always minimizing sequences which do not converge strongly (assuming there exists a minimizer, one can simply translate it far away). 

Consider the electrons in an atom or in a molecule, with fixed classical nuclei (Born-Oppenheimer approximation). From the point of view of the electrons, the problem is no more translation-invariant, once the nuclei have been given a fixed position. Since the Coulomb interaction between the electrons is repulsive, $W(x-y)=|x-y|^{-1}$, Theorem~\ref{thm:Lieb} tells us that there is never pairing, $\alpha\equiv0$. In this case there are several existence results, starting with the fundamental works of Lieb and Simon~\cite{LieSim-77} and continuing with works by Lions~\cite{Lions-87}, Bach~\cite{Bach-92}, Solovej~\cite{Solovej-91,Solovej-03}. 

For interactions $W$ which have no particular sign, the pure HF problem was studied by Friesecke in~\cite{Friesecke-03} and by the first author of this paper in~\cite{Lewin-11}. There is an HVZ-type theorem for HF wavefunctions which states that some binding conditions imply the existence of minimizers (Theorem~22 in~\cite{Lewin-11}).

After the fundamental paper of Bach, Lieb and Solovej~\cite{BacLieSol-94} with its study of the Hubbard model, to our knowledge the existence of ground states for the HFB model with pairing was only studied recently by Lenzmann and the first author of this paper in~\cite{LenLew-10}. Some caricatures of HFB in nuclear physics had been previously considered by Gogny and Lions in~\cite{GogLio-86}.

\medskip

We give here an existence result for the variational problem $I(N)$. In some cases we have $I(N)=J(N)$ (see the previous section) and for this reason, we only concentrate on $I(N)$. The next theorem can be proved by following the method of~\cite{LenLew-10}, which dealt with the more complicated case of a pseudo-relativistic kinetic operator $\sqrt{1-\Delta}-1$. 

\begin{theorem}[Existence of minimizers and compactness of minimizing sequences]\label{thm:existence}
We assume as before that $W=W_1+W_2\in L^p(\R^3)+L^q(\R^3)$ with $2\leq p\leq q<\ii$. Let $\lambda>0$. Then the following assertions are equivalent:
\begin{enumerate}
\item All the minimizing sequences $(\gamma_n,\alpha_n)\subset\cK(\lambda)$ for $I(\lambda)$ are precompact up to translations, that is there exists a sequence $(x_k)\subset\R^3$ and $(\gamma,\alpha)\in\cK(\lambda)$ such that, for a subsequence, 
\begin{multline*}
\lim_{k\to\ii}\norm{(1-\Delta)^{1/2}(\tau_{x_k}\gamma_{n_k}\tau_{-x_k}-\gamma)(1-\Delta)^{1/2}}_{\gS_1}\\
=\lim_{k\to\ii}\norm{(1-\Delta)^{1/2}(\tau_{x_k}\alpha_{n_k}\tau_{-x_k}-\alpha)}_{\gS_2}=0. 
\end{multline*}
In particular $(\gamma,\alpha)$ is a minimizer for $I(\lambda)$.

\smallskip

\item The \emph{binding inequalities}
\begin{equation}
I(\lambda)<I(\lambda-\mu)+I(\mu) \qquad\text{for all $0<\mu<\lambda$}
\label{eq:binding}
\end{equation}
are satisfied.
\end{enumerate}

\medskip

Furthermore, if $W$ is Newtonian at infinity, that is
\begin{equation}
W(x)\leq -\frac{a}{|x|}\quad \text{for $a>0$ and $|x|\geq R$},
\label{assumption_LenLew_1}
\end{equation}
then the previous two equivalent conditions are verified.
\end{theorem}

The assumption that the interaction is Newtonian at infinity is a big simplification, as it means that two subsystems receiding from each other always attract at large distances. One can expect that minimizers exist even if the potential is not attractive at infinity, as soon as it has a sufficiently large negative component. A typical effective potential $W(x)$ used in nuclear physics is nonnegative for small $|x|$, and has a negative well at intermediate distances~\cite{RinSch-80}. At infinity it typically decays like $+\kappa|x|^{-1}$ for two protons, and exponentially fast when one of the two particles is a neutron. Even in HF theory, we are not aware of any existence result dealing with such potentials, however. We will numerically investigate a model of this form in Section~\ref{sec:Coulomb-perturbed}.

The form of the nonlinear equation solved by minimizers is well-known in the physics literature, and it was re-explained in~\cite{BacLieSol-94}. The following result summarizes some known properties.

\begin{theorem}[HFB equation and properties of minimizers~\cite{BacLieSol-94,LenLew-10}]
A HFB minimizer on $\cK(N)$ solves the nonlinear equation
\begin{equation}
\Gamma=\1_{(-\ii,0)}\big(F_\Gamma-\mu\cN\big)+\delta
\label{eq:HFB_equation}
\end{equation}
where $0\leq\delta\leq \1_{\{0\}}(F_\Gamma-\mu\cN)$ has the same form as $\Gamma$, and where 
\begin{equation}
\cN:=\begin{pmatrix}
1&0\\ 0 & -1
\end{pmatrix},\quad 
F_\Gamma=\begin{pmatrix}
h_\gamma&\pi\\ \pi^*&-\overline{h_\gamma}
\end{pmatrix}
\end{equation}
with $h_\gamma=-\Delta+\rho_\gamma\ast W-W(x-y)\gamma(x,y)$ and $\pi(x,y)=\alpha(x,y)W(x-y)$. 

If $W(x-y)=-\kappa|x-y|^{-1}$ (Newtonian interaction) and $N$ is an integer, then all the minimizers are of the special form~\eqref{eq:form_gamma_alpha}. In this case, we have either $\alpha\equiv0$ and $\gamma$ is a projector of rank $N$, or $\alpha\neq0$ and $\gamma$ has an infinite rank.
\end{theorem}

The nonlinear equation~\eqref{eq:HFB_equation} is in principle similar to the usual equation obtained in HF theory, 
\begin{equation}
\gamma=\1_{(-\ii,\mu)}(h_\gamma)+\delta
\label{eq:HF_equation} 
\end{equation}
Indeed,~\eqref{eq:HFB_equation} reduces to~\eqref{eq:HF_equation} when $\alpha\equiv0$. Let us however emphasize that the mean-field operator $F_\Gamma$ has a spectrum which is symmetric with respect to 0. Hence $F_\Gamma$ is usually not even semi-bounded, on the contrary to $h_\gamma$ which is always bounded from below. Furthermore, the operator $\cN$ does \emph{not} commute with $F_\Gamma$ (except when $\alpha\equiv0$) and the equation cannot be written in a simple form as in HF theory. This will cause several difficulties to which we will come back at length later.

The fact that $\gamma$ has an infinite rank when there is pairing, $\alpha\neq0$, is a dramatic change of behavior compared to the simple HF case. However, no information on the decay of the eigenvalues of $\gamma$ seems to be known.

An important open question is to show that minimizers actually exhibit non-vanishing pairing $\alpha \neq 0$, at least for a sufficiently strong attractive potential $W$. On heuristic grounds, one expects such a phenomenon of {\em ``Cooper pair formation''} to be energetically favorable due to the attractive interaction among particles. However, it seems to be a formidable task to find mathematical proof for this claim. The existence of pairing is known in some particular situations (when $N$ is odd and $W$ is Newtonian, see Remark~\ref{rmk:N_odd}, for the Hubbard model~\cite{BacLieSol-94}, or in translation-invariant BCS theory~\cite{BilFan-68,Vansevenant-85,Yang-91,LeoYan-00,FraHaiNabSei-07,HaiHamSeiSol-08,HaiSei-08}), but for the model presented here, we are not aware of any result of this sort. One of the purposes of this paper is to investigate this question numerically.

\section{Discretized Hartree-Fock-Bogoliubov theory}\label{sec:discretized}

In this section, we write and study the Hartree-Fock-Bogoliubov model in a discrete basis.

\subsection{Convergence analysis}
Here we show that the ground state HFB energy in a finite basis converges to the true HFB ground state energy when the size of the basis grows. We consider a sequence of finite-dimensional spaces $V_h\subset H^1(\R^3,\C^q)$ for $h\to0$. We assume that any function $f\in H^1(\R^3,\C^q)$ can be approximated by functions from $V_h$:
\begin{equation}
\forall f\in H^1(\R^3,\C^q),\quad \exists f_h\in V_h\quad\text{ such that }\quad \norm{f-f_h}_{H^1}\underset{h\to0}{\longrightarrow}0.
\label{assumption_V_h}
\end{equation}
We typically think of a sequence $V_h$ given by the Finite Elements Method.
Let $\pi_h$ denote the orthogonal projection on $V_h$ in $L^2(\R^3,\C^q)$. We define the set of density matrices living on $V_h$ (with average particle number $N$) as follows
\begin{equation}
\cK_h(N)=\big\{ (\gamma,\alpha)\in\cK(N)\ :\ \pi_h\gamma\pi_h=\gamma,\ \pi_h\alpha\,\overline{\pi_h}=\alpha\big\}. 
\end{equation}
The corresponding minimization problem is now
\begin{equation}
I_h(N)=\inf_{(\gamma,\alpha)\in\cK_h(N)}\cE(\gamma,\alpha).
\end{equation}
Since $\cK_h(N)\subset\cK(N)$ by definition, it is obvious that $I_h(N)\geq I(N)$. The following result is a consequence of Theorem~\ref{thm:existence}.

\begin{theorem}[Convergence of the approximate HFB problem]
When $W=W_1+W_2\in L^p(\R^3)+L^q(\R^3)$ with $2\leq p\leq q<\ii$ and under Assumption~\eqref{assumption_V_h} on the sequence $(V_h)$, we have
\begin{equation}
 \lim_{h\to0}I_h(N)=I(N).
\end{equation}
If the binding inequality~\eqref{eq:binding} is satisfied, then any sequence of minimizers $(\gamma_h,\alpha_h)\in\cK_h(N)$ for $I_h(N)$ converges, up to a subsequence and up to a translation, to a minimizer $(\gamma,\alpha)\in\cK(N)$ of $I(N)$, in the sense that
\begin{multline}
\lim_{h_k\to0}\norm{(1-\Delta)^{1/2}(\tau_{x_k}\gamma_{h_k}\tau_{-x_k}-\gamma)(1-\Delta)^{1/2}}_{\gS_1}\\
=\lim_{h_k\to0}\norm{(1-\Delta)^{1/2}(\tau_{x_k}\alpha_{h_k}\tau_{-x_k}-\alpha)}_{\gS_2}=0.
\label{eq:CV_discretized} 
\end{multline}
\end{theorem}

\begin{proof}
We only have to show that $I_h(N)\to I(N)$ as $h\to0$. Then, any sequence of exact minimizers $(\gamma_h,\alpha_h)$ for $I_h(N)$ is also a minimizing sequence for $I(N)$. Applying Theorem~\ref{thm:existence} concludes the proof.

We know that finite-rank operators are dense in $\cK(N)$. Let $(\gamma,\alpha)\in\cK(N)$ be any such finite rank operator. Let $(f_i)_{i=1}^K$ be an orthonormal basis of the range of $\gamma$. By Löwdin's theorem (Lemma 13 in~\cite{Lewin-11}), we know that the two-body wavefunction $\alpha$ can be expanded in the same basis $(f_1,...,f_K)$. Now, for every $i=1,..,K$, we apply~\eqref{assumption_V_h} and take a sequence $f_i^h\in V_h$ be such that $f_i^h\to f_i$ in $H^1(\R^3)$. The system $(f_i^h)_{i=1}^k$ is not necessarily orthonormal but we have $\pscal{f_i^h,f_j^h}\to\pscal{f_i,f_j}=\delta_{ij}$. Applying the Gram-Schmidt procedure, we can therefore replace the $(f_i^h)_{i=1}^K$ by an orthonormal set $(g_i^h)_{i=1}^K\subset V_h$ having the same properties. An equivalent procedure is to take
$g_i^h=\sum_{j=1}^K(S^{-1/2}_h)_{ji}f_j^h$ where $S_h$ is the Gram matrix $(\pscal{f_i^h,f_j^h})$. Let now $U_h$ be any unitary operator on $L^2(\R^3)$ which is such that $U_hf_i=g_i^h$ for all $i=1,...,K$. We then take $\gamma_h:=U_h\gamma U_h^*$ and $\alpha_h:=U_h\alpha U_h^T$. In words, we just replace the $f_i$ by $g_i^h$ in the decomposition of $\gamma$ and $\alpha$. To see that $(\gamma_h,\alpha_h)\in\cK(N)$, we just notice that
$$\begin{pmatrix}
\gamma_h & \alpha_h\\
\alpha_h^* & 1-\overline{\gamma_h}
  \end{pmatrix}
=
\begin{pmatrix}
U_h & 0\\
0 & \overline{U_h}
  \end{pmatrix}
\begin{pmatrix}
\gamma & \alpha\\
\alpha^* & 1-\overline{\gamma}
  \end{pmatrix}
\begin{pmatrix}
U_h^* & 0\\
0 & \overline{U_h}^*
  \end{pmatrix}.$$
Note also that $\tr(\gamma_h)=\tr(\gamma)=N$ since $U_h$ is unitary. Now $\gamma_h$ and $\alpha_h$ belong to $\cK_h(N)$ by definition, hence we have that 
$\cE(\gamma_h,\alpha_h)\geq I_h(N)$. On the other hand, by the convergence of $f_i^h$ (hence of $g_i^h$) towards $f_i$ in $H^1(\R^3)$, we easily see that 
$$\lim_{h\to0}\cE(\gamma_h,\alpha_h)=\cE(\gamma,\alpha)$$
by continuity of $\cE$.
Therefore we have proved that 
$$\limsup_{h\to0}I_h(N)\leq \cE(\gamma,\alpha).$$
This is valid for all finite-rank $(\gamma,\alpha)\in\cK(N)$, hence we deduce that 
$$\limsup_{h\to0}I_h(N)\leq \inf_{(\gamma,\alpha)\in\cK(N)}\cE(\gamma,\alpha)=I(N).$$
On the other hand the inequality $I_h(N)\geq I(N)$ is always satisfied, hence we have proved the claimed convergence $I_h(N)\to I(N)$.
\end{proof}

\subsection{Discretization}

In this section we compute the HFB energy $\cE(\gamma,\alpha)$ of a discretized state $(\gamma,\alpha)\in\cK_h(N)$ and we write the corresponding self-consistent equation. We fix once and for all the approximation space $V_h$ and we consider a basis set $(\chi_i)_{i=1}^{N_b}$ of $V_h$, which is not necessarily orthonormal. We will assume that $V_h$ is stable under complex conjugation, which means that $f\in V_h\Rightarrow \overline{f}\in V_h$ (this amounts to replacing $V_h$ by ${\rm Span}(V_h,\overline{V_h})$). Then we can choose a basis $(\chi_i)_{i=1}^{N_b}$ which is real, that is $\overline{\chi_i}=\chi_i$ for all $i=1,...,N_b$. This will dramatically simplify the calculation below.

Since $\pi_h\gamma\pi_h=\gamma$ and $\pi_h\alpha\overline{\pi_h}=\alpha$, we can write the kernels of $\gamma$ and $\alpha$ as follows
\begin{equation}
\gamma(x,y)_{\sigma,\sigma'}=\sum_{i,j=1}^{N_b}G_{ij}\,\chi_i(x)_{\sigma}\overline{\chi_j(y)_{\sigma'}},\qquad \alpha(x,y)_{\sigma,\sigma'}=\sum_{i,j=1}^{N_b}A_{ij}\,\chi_i(x)_{\sigma}\chi_j(y)_{\sigma'}.
\label{eq:def_G_A}
\end{equation}
The complex conjugation on $\overline{\chi_j}$ is superfluous but we keep it to emphasize the difference between $\gamma$ and $\alpha$.
The matrices $G$ and $A$ (defined by the previous relation) satisfy the constraints $G^*=G$ and $A^T=-A$. Note that since $A$ is antisymmetric, we can also write
\begin{align*}
\alpha(x,y)_{\sigma,\sigma'}&=\sum_{1\leq i<j\leq N_b}A_{ij}\,\big(\chi_i(x)_{\sigma}\chi_j(y)_{\sigma'}-\chi_j(x)_{\sigma}\chi_i(y)_{\sigma'}\big)\\
&=\sqrt{2}\sum_{1\leq i<j\leq N_b}A_{ij}\,\chi_i\wedge\chi_j(x,\sigma,y,\sigma') 
\end{align*}
where $(\chi_i\wedge\chi_j)(x,\sigma;y,\sigma'):=\big(\chi_i(x)_\sigma\chi_j(y)_{\sigma'}-\chi_i(y)_{\sigma'}\chi_j(x)_{\sigma}\big)/\sqrt{2}$ is a two-body Slater determinant. Let us also remark that~\eqref{eq:def_G_A} can be written in the operator form
$$\gamma=\sum_{i,j=1}^{N_b}G_{ij}\;|\chi_i\rangle\langle\chi_j|,\qquad \alpha=\sum_{i,j=1}^{N_b}A_{ij}\;|\chi_i\rangle\langle\overline{\chi_j}|.$$
Again the complex conjugation on $\overline{\chi_j}$ is superfluous.

Let us define the so-called \emph{overlap matrix} $S=S^*=\overline{S}$ by
\begin{equation}
S_{ij}=\pscal{\chi_i,\chi_j}_{\gH}=\sum_{\sigma=1}^q\int_{\R^3}\!\!\overline{\chi_i(x)_\sigma}\chi_j(x)_\sigma\,dx=\sum_{\sigma=1}^q\int_{\R^3}\!\!\chi_i(x)_\sigma\chi_j(x)_\sigma\,dx.
\end{equation}
It is tedious but straightforward to verify that the constraint 
$$\begin{pmatrix}
0&0\\ 0&0\end{pmatrix}
\leq 
\Gamma:=\begin{pmatrix}
\gamma&\alpha\\ \alpha^*&1-\overline{\gamma}\end{pmatrix}
\leq
\begin{pmatrix}
1&0\\ 0&1\end{pmatrix}$$
can be written for the matrices $G$ and $A$ in the form
\begin{equation}
 \begin{pmatrix}
0&0\\ 0&0\end{pmatrix}
\leq 
\begin{pmatrix}
SGS&SAS\\ SA^*S&\ \ S-S\overline{G}S
\end{pmatrix}
\leq
\begin{pmatrix}
S&0\\ 0&S\end{pmatrix}
\label{eq:discretized_constraint}
\end{equation}
or, equivalently,
\begin{equation}
0\leq 
\Upsilon\bS\Upsilon
\leq
\Upsilon
\label{eq:discretized_constraint2}
\end{equation}
where $\Upsilon$ and $\bS$ are defined by
\begin{equation}
\Upsilon:=\begin{pmatrix}
G&A\\ A^* & \ \ S^{-1}-\overline{G}
\end{pmatrix}\quad\text{ and }\quad
\bS:=\begin{pmatrix}
S & 0\\
0 & S
\end{pmatrix}.
\label{eq:def_Gamma_bigS}
\end{equation}
Another way to write the constraint is $0\leq\bS^{1/2}\Upsilon\bS^{1/2}\leq 1$.
Let us notice that we have used everywhere the fact that $S=S^*=\overline{S}=S^T$. The formulas are much more complicated when $S$ is not real. 

The energy can be expressed in terms of the matrices $G$ and $A$, as well. A calculation shows that
\begin{equation}
\cE(\gamma,\alpha)=\tr(hG)+\frac12\tr(G\,J(G))-\frac12\tr(G\,K(G))+\frac12\tr(A^*\,X(A)).
\label{eq:discretized_energy}
\end{equation}
The trace here is the usual one for $N_b\times N_b$ matrices. As we think that there is no possible confusion with $\cE(\gamma,\alpha)$, we will also denote by $\cE(G,A)$ this discretized energy functional.
In formula~\eqref{eq:discretized_energy}, 
$$h_{ij}=\pscal{\chi_i,(-\Delta)\chi_j}=\sum_{\sigma=1}^q\int_{\R^3}\nabla\chi_i(x)_\sigma\cdot \nabla\chi_j(x)_\sigma\,dx,$$
\begin{equation}
J(G)_{ij}=\sum_{k,\ell=1}^{N_b}(ij|\ell k)\, G_{k\ell},\qquad K(G)_{ij}=\sum_{k,\ell=1}^{N_b}(ik|\ell j)\, G_{k\ell},\qquad X(A)_{ij}=\sum_{k,\ell=1}^{N_b}(ik|j\ell)\, A_{k\ell},
\label{eq:def_direct_echange_appar}
\end{equation}
and
\begin{equation}
(ij|k\ell):=\int_{\R^3}\int_{\R^3}W(x-y)\, \chi_i(x)^*\;\chi_j(x)\; \chi_k(y)^*\;\chi_\ell(y)\,dx\,dy.
\label{eq:def_echange}
\end{equation}
We use here the notation $a^*b=\sum_{\sigma=1}^q\overline{a}_\sigma\,b_\sigma$ (but the complex conjugation is superfluous for our real basis, hence $K=X$). Similarly, we have
\begin{equation}
\tr(\gamma)=\tr(SG).
\label{eq:trace}
\end{equation}
We define the discretized number operator as
\begin{equation}
\bN=\begin{pmatrix}
S & 0\\ 0 & -S
\end{pmatrix}
\label{eq:def_N}
\end{equation}
The constraint $\tr(\gamma)=N$ can be written equivalently as
$$\tr(\bN\Upsilon)=2N-N_b$$

We deduce from this calculation that the variational problem $I_h(N)$ can be written in finite dimension as
\begin{equation}
I_h(N)=\min\Big\{\cE(G,A)\ :\ 0\leq\Upsilon\bS\Upsilon\leq\Upsilon,\ \tr(\bN\Upsilon)=2N-N_b\Big\}
\label{eq:value_I_h}
\end{equation}
where we recall that $\Upsilon$ and $\bS$ have been defined in~\eqref{eq:def_Gamma_bigS}. Here the infimum is always attained because the problem is finite dimensional.

In this form, the discretized problem is very similar to the usual discretized Hartree-Fock problem~\cite{CanDefKutLeBMad-03,LeBris-05}, in dimension $2N_b$ instead of $N_b$. There is a big difference, however. In HF theory the constraint involves the matrix $\bS$ instead of $\bN$. This difference will cause several difficulties. To understand the problem, let us introduce a new variable $\Upsilon'=\bS^{1/2}\Upsilon\bS^{1/2}$ (which is the same as orthonormalize the basis $(\chi_i)$). Then the constraint $0\leq\Upsilon\bS\Upsilon\leq\Upsilon$ is transformed into $0\leq\Upsilon'\leq1$. However, the constraint on the number of particles becomes $\tr(\bS^{-1/2}\bN\bS^{-1/2}\Upsilon')=2N-N_b$. In usual Hartree-Fock theory, the matrix $\bS^{-1/2}\bN\bS^{-1/2}$ is replaced by the identity. The fact that this matrix then commutes with the Fock Hamiltonian (defined below) simplifies dramatically the self-consistent equations. Here we get
$$\bS^{-1/2}\bN\bS^{-1/2}=\begin{pmatrix}
1&0\\ 0&-1
\end{pmatrix}$$
which commutes with the Fock Hamiltonian if and only if $A\equiv0$. 

The self-consistent equation is obtained like in~\cite{BacLieSol-94}. The result is as follows.

\begin{lemma}[Discretized HFB equation]
Let $\Upsilon$ be a minimizer for the variational problem $I_h(N)$. Then there exists $\mu\in\R$ such that $\Upsilon$ solves the linear problem
\begin{equation}
\min\Big\{\tr(\bF_\Upsilon-\mu\bN)\tilde{\Upsilon}\ :\ 0\leq\tilde{\Upsilon}\bS\tilde{\Upsilon}\leq\tilde{\Upsilon}\Big\}
\end{equation}
where
\begin{equation}
\bF_\Upsilon:=\begin{pmatrix}
h_G&X(A)\\ X(A)^* & -\overline{h_G}
\end{pmatrix},\qquad h_G:=h+J(G)-K(G).
\end{equation}
The solution $\Upsilon$ can be written in the form
\begin{equation}
\Upsilon=\bS^{-1/2}\bigg(\1_{(-\ii,0)}\Big(\bS^{-1/2}(\bF_\Upsilon-\mu\bN)\bS^{-1/2}\Big)+\delta\bigg)\bS^{-1/2}
\label{eq:discretized_eq}
\end{equation}
where $0\leq\delta\leq 1$ lives only in the kernel of $\bS^{-1/2}(\bF_\Upsilon-\mu\bN)\bS^{-1/2}$.
\end{lemma}

The solution $\Upsilon$ of the self-consistent equation~\eqref{eq:discretized_eq} may be equivalently written by considering the generalized eigenvalue problem
\begin{equation}
\big(\bF_\Upsilon-\mu\bN\big)f_i=\epsilon_i\,\bS\,f_i,\qquad \pscal{f_i,\bS f_j}=\delta_{ij}.
\label{eq:generalized_eig_pb}
\end{equation}
Then we have simply (assuming $\epsilon_i\neq0$ for all $i=1,...,2N_b$)
$$\Upsilon=\sum_{\epsilon_i<0}f_i\,f_i^*.$$
Again, this is similar to the Hartree-Fock solution~\cite{CanDefKutLeBMad-03,LeBris-05} except that $\mu$ is unknown and $\bN$ does not always commute with $\bF_\Upsilon$.

Remark that although the basis functions $\chi_i$ are real, the density matrix $\Upsilon$ is not necessarily real. In the next section, we will restrict ourselves to real-valued density matrices and impose some spin symmetry.

\subsection{Using symmetries}
The HFB energy $\cE(\Gamma)$ has some natural symmetry invariances which we describe in detail in this section. Recall that since $\cE$ is a \emph{nonlinear} functional, it cannot be guaranteed that the HFB minimizers will all have the same symmetries as the HFB energy. The set of all minimizers will be invariant under the action of the symmetry group, but each minimizer alone does not have to be invariant.

We have already allowed for the breaking of particle-number symmetry and we hope to find an HFB ground state. It will then automatically break the translational invariance of the system. There are three other symmetries (spin, complex conjugation and rotations) which are of interest to us. We have the choice of imposing these symmetries by adding appropriate constraints, or not. Because this reduces the computational cost, it will be convenient to impose them.

\subsubsection{Time-reversal symmetry}\label{sec:time-reversal}
Let us now assume that $q=2$, which means that our fermions are spin-$1/2$ particles. Since the Laplacian and the interaction function $W$ do not act on the spin variable, the HFB energy has some \emph{spin symmetry}, which can be written for $q=2$ as
$$\forall k=1,2,3,\qquad \cE(\Sigma_k\Gamma\Sigma_k^*)=\cE(\Gamma)$$
where 
$$\Sigma_k:=\begin{pmatrix}
i\sigma_k & 0\\ 0 & \overline{i\sigma_k}\end{pmatrix},$$
with $\sigma_1,\sigma_2,\sigma_3$ being the usual Pauli matrices. Note that $\Sigma_k$ has the form of a Bogoliubov transformation, hence $\Sigma_k\Gamma\Sigma_k^*$ is also an HFB state. The number operator is also invariant which means that 
$$\Sigma_k\cN\Sigma^*_k=\cN$$
for all $k=1,2,3$.
Thus, the contraint $\tr(\gamma)=N$ is conserved and we have $\Sigma_k\Gamma\Sigma_k^*\in\cK(N)$ when $\Gamma\in\cK(N)$.

Another important symmetry is that of \emph{complex conjugation} which means this time that
$$\cE(\overline{\Gamma})=\cE(\Gamma)$$
and which is based on the fact that the Laplacian and $W$ are real operators.
Again we have $\tr(\overline\gamma)=\tr(\gamma)$ hence $\cK(N)$ is invariant under complex conjugation.

As was explained in~\cite{HaiLenLewSch-10} (see Remark 5 page 1032), the density matrices $(\gamma,\alpha)$ can be written in the special form  
\begin{equation}
\gamma'=g\otimes \begin{pmatrix}
1 &0\\ 0&1
\end{pmatrix},\qquad \alpha'=a\otimes\begin{pmatrix}
0 &1\\ -1&0
\end{pmatrix},\qquad \text{with $g=g^T=\overline{g}$ and $a=a^T=\overline{a}$}
\label{eq:form_gamma_alpha2} 
\end{equation}
(the $2\times2$ matrix refers to the spin variables), if and only if
\begin{equation}
\begin{cases}
\Sigma_k\Gamma\Sigma_k^*=\Gamma, & \text{for $k=1,2$, and}\\
\overline{\Gamma}=\Gamma.
\end{cases}
\end{equation}
In other words, $\Gamma$ is invariant under the action of the group generated by $\Sigma_1$, $\Sigma_2$ and $\cC$. This invariance is sometimes called the \emph{time-reversal symmetry}. As remarked in~\cite{HaiLenLewSch-10}, imposing $\Sigma_k\Gamma\Sigma_k^*=\Gamma$ for all $k=1,2,3$ implies $\alpha\equiv0$ which is not interesting for us.

When $W$ can be written in the form~\eqref{eq:form_W_BFJ}, the Theorem~\ref{thm:variational_constrained} of Bach, Fröhlich and Jonsson~\cite{BacFroJon-09}, tells us that there is no breaking of the time-reversal symmetry. That is, we can always minimize over such special states. For other interactions $W$ this is not necessarily true but it is often convenient to impose this symmetry anyhow. 

Because it holds
$$F_{\Sigma_k\Gamma\Sigma_k^*}=\Sigma_kF_\Gamma\Sigma_k^*,\quad F_{\overline{\Gamma}}=\overline{F_\Gamma},$$
it can then be verified that minimizers under the additional symmetry constraint, satisfy the same self-consistent equation as when no constraint is imposed.

When we discretize the problem by imposing time-reversal symmetry, we use two real symmetric matrices $G$ and $A$, related through the constraint that
\begin{equation}
0\leq\Upsilon\bS\Upsilon\leq\Upsilon,\quad\text{with}\quad  
\Upsilon:=\begin{pmatrix}
G&A\\ A&\ S^{-1}-G\end{pmatrix}\quad\text{and}\quad
\bS=\begin{pmatrix}
S&0\\ 0&S\end{pmatrix}.
\label{eq:relaxed-constraint-symmetry}
\end{equation}
The energy becomes
\begin{equation}
\cE(G,A)=2\tr(hG)+2\tr(G\,J(G))-\tr(G\,K(G))+\tr(A\,K(A)).
\label{eq:discretized_energy_nospin}
\end{equation}
and the associated particle number constraint is $\tr(SG)=N/2$. In practice we always assume that $N$ is even for simplicity. The basis $(\chi_i)$ is now composed of (real-valued) functions in $H^{1}(\R^3,\R)$, instead of functions in $H^1(\R^3,\C^q)$ as before, and the formulas for $S$, $h$, $J$, $K$ and $X$ are the same as before.

\subsubsection{Rotational symmetry}\label{sec:spherical-symm}

The group $SO(3)$ of rotations in $\R^3$ also acts on HFB states and it leaves the energy invariant when the interaction $W$ is a radial function.
In this section we always assume that the spin variable has already been removed according to the previous section and we denote by $g$ and $a$ the corresponding (real symmetric) density matrices. If the spin were still present, rotations would act on it as well.

To any rotation $R\in SO(3)$ we can associate a unitary operator on $L^2(\R^3)$, denoted also by $R$, defined by $(R\phi)(x)=\phi\big(R^{-1}x\big)$. Now we say that an HFB state $\Gamma$ with density matrices $(\gamma,\alpha)$ is invariant under rotations when it satisfies
$$\mathscr{R}\,\Gamma\mathscr{R}^*=\Gamma,\qquad\text{where}\quad \mathscr{R}=\begin{pmatrix}
R & 0\\ 0 & R
\end{pmatrix}.
$$
Note that $\mathscr{R}$ is a Bogoliubov rotation since $\overline{R}=R$. The density matrices of an invariant state satisfy 
$$g(Rx,Ry)=g(x,y),\qquad a(Rx,Ry)=a(x,y)$$
for all $x,y\in\R^3$ and any rotation $R\in SO(3)$. 

As the angular momentum $L=x\times (-i\nabla)$ generates the group of rotations, a (smooth enough) HFB state is invariant under rotations if and only if
$$\mathscr{L}\Gamma=\Gamma\mathscr{L},\qquad\text{where}\quad \mathscr{L}=\begin{pmatrix}
L & 0\\ 0 & L
\end{pmatrix}.
$$
The density matrices $g$ and $a$ can then be written in the special form
\begin{equation}
g(x,y)=\frac{1}{4\pi}\sum_{\ell\geq0}g^\ell(|x|,|y|)\, (2\ell+1)\,P_\ell\big(\omega_x\cdot\omega_y\big),\ \  a(x,y)=\frac{1}{4\pi}\sum_{\ell\geq0}a^\ell(|x|,|y|)\, (2\ell+1)\,P_\ell\big(\omega_x\cdot\omega_y\big)
\label{eq:form_angular_momentum}
\end{equation}
where $P_\ell$ is the Legendre polynomial of degree $\ell$, which is such that $P_\ell(1)=1$. The constraint on $g$ and $a$ are transfered in each angular momentum sector (labelled by $\ell\geq0$), leading to 
$$\begin{pmatrix}
0&0\\ 0&0\end{pmatrix}
\leq 
\begin{pmatrix}
g^\ell&a^\ell\\ a^\ell&\ 1-g^\ell\end{pmatrix}
\leq
\begin{pmatrix}
1&0\\ 0&1\end{pmatrix}$$
on $L^2([0,\ii),r^2\,dr)\oplus L^2([0,\ii),r^2\,dr)$.
However, there is no such constraint between different $\ell$'s. The total average particle number is given by
$$\tr(g)=\sum_{\ell\geq0}(2\ell+1)\tr(g^\ell)=N/2.$$
This is the only constraint which mixes the different angular momentum density matrices.

Now we can discretize the radial HFB problem. We choose a finite-dimensional subspace $V_{\text{rad}}$ in $L^2([0,\ii),r^2dr)$ with basis $(\chi_1,...,\chi_{N_b})$, which we use to expand the density matrices $g^\ell$ and $a^\ell$. Then we fix a maximal angular momentum $\lmax$ and we truncate the series in~\eqref{eq:form_angular_momentum}. This is the same as taking as discretization space 
$$V=\Big\{f(|x|)\, Y_\ell^m(\omega_x)\ :\ f\in V_{\rm rad},\ 0\leq\ell\leq\lmax,\ -\ell\leq m\leq\ell\Big\}\subset L^2(\R^3,\R)$$
where $Y_\ell^m$ is the spherical harmonics of total angular momentum $\ell$ and azimuthal angular momentum $m$. We then assume that $g$ and $a$ are radial and live in this space. The matrices $G^\ell$ and $A^\ell$ of $g^\ell$ and $a^\ell$ in the basis $(\chi_i)$ are defined similarly as before by
\begin{equation}
g^\ell(r,r')=\sum_{i,j=1}^{N_b}G^\ell_{ij}\,\chi_i(r)\chi_j(r'),\qquad a^\ell(r,r')=\sum_{i,j=1}^{N_b}A^\ell_{ij}\,\chi_i(r)\chi_j(r').
\label{eq:def_G_A_ell}
\end{equation}
The constraints on the matrices $G^\ell$ and $A^\ell$ are
\begin{equation}
0\leq\Upsilon^\ell\bS\Upsilon^\ell\leq\Upsilon^\ell,\quad\text{with}\quad  
\Upsilon^\ell:=\begin{pmatrix}
G^\ell&A^\ell\\ A^\ell&\ S^{-1}-G^\ell\end{pmatrix}\quad\text{and}\quad
\bS=\begin{pmatrix}
S&0\\ 0&S\end{pmatrix}
\label{eq:discretized_constraint_ell}
\end{equation}
with
$$S_{ij}=\int_0^\ii \chi_i(r)\,\chi_j(r)\,r^2\,dr$$
and
\begin{equation}
\sum_{\ell=0}^\lmax(2\ell+1)\tr(SG^\ell)=N/2.
\label{eq:contrainte_N_ell} 
\end{equation}

The total HFB energy is now
\begin{multline}
\cE(G^0,...,G^\lmax,A^0,...,A^\lmax)=2\sum_{\ell=0}^\lmax(2\ell+1)\tr(h^\ell G^\ell)\\
+\sum_{\ell,\ell'=0}^\lmax (2\ell+1)(2\ell'+1)\Big(2\tr(G^\ell\,J(G^{\ell'}))-\tr(G^\ell\,K^{\ell\ell'}(G^{\ell'}))+\tr(A^\ell\,K^{\ell\ell'}(A^{\ell'}))\Big),
\label{eq:discretized_energy_nospin_radial}
\end{multline}
where
$$h^\ell_{ij}:=\int_0^\ii \chi_i'(r)\,\chi_j'(r)\,r^2\,dr+\ell(\ell+1)\int_0^\ii \chi_i(r)\,\chi_j(r)\,dr,$$
$$J(G^{\ell'})_{ij}:=\sum_{m,n=0}^{N_b}(ij|nm)_{0,0}\, G^{\ell'}_{mn},\qquad K^{\ell\ell'}(G^{\ell'})_{ij}:=\sum_{m,n=0}^{N_b}(im|jn)_{\ell,\ell'}\, G^{\ell'}_{mn},$$
$$(ij|mn)_{\ell,\ell'}:= \int_0^\ii r^2\,dr\int_0^\ii s^2\,ds\;\chi_i(r)\,\chi_j(r)\,\chi_m(s)\,\chi_n(s)\, w_{\ell,\ell'}(r,s)$$
and
$$w_{\ell,\ell'}(r,s):=\frac12\int_{-1}^1W\left(\sqrt{r^2+s^2-2rst}\right)\, P_\ell(t)\,P_{\ell'}(t)\,dt.$$
Any minimizer $(G^0,...,G^\lmax,A^0,...,A^\lmax)$ of $\cE$ under the constraints~\eqref{eq:discretized_constraint_ell} and~\eqref{eq:contrainte_N_ell} will be of the form
\begin{equation}
\Upsilon^\ell=\sum_{\epsilon_i^\ell<0}f_i^\ell\big(f_i^\ell\big)^T,\qquad 0\leq\ell\leq\lmax,
\label{eq:generalized_form_Gamma_ell} 
\end{equation}
where the vectors $f^\ell_i$'s solve the generalized eigenvalue problem
\begin{equation}
\Big(\bF^\ell-\mu(2\ell+1)\bN\Big)f_i^\ell=\epsilon_i^\ell\,\bS\,f_i^\ell,\qquad \pscal{f_i,\bS f_j}=\delta_{ij}
\label{eq:generalized_eig_pb_ell}
\end{equation}
with 
$$\bF^\ell:=\begin{pmatrix}
h^\ell&0\\ 0 & -h^\ell
\end{pmatrix}+\sum_{\ell'=0}^\lmax\begin{pmatrix}
2J\big(G^{\ell'}\big)-K^{\ell\ell'}\big(G^{\ell'}\big) & K^{\ell\ell'}\big(A^{\ell'}\big)\\
K^{\ell\ell'}\big(A^{\ell'}\big) & -2J\big(G^{\ell'}\big)+K^{\ell\ell'}\big(G^{\ell'}\big)
\end{pmatrix}.
$$
The Euler-Lagrange multipler $\mu$ appearing in~\eqref{eq:generalized_eig_pb_ell} is common to all the different angular momentum sectors and it is chosen to ensure the validity of the constraint~\eqref{eq:contrainte_N_ell}.

\section{Algorithmic strategies and convergence analysis}\label{sec:algos}

In this section we study the convergence of two algorithms which can be used in practice to solve the HFB minimization problem~\eqref{eq:def_J_N}. In order to simplify our presentation, we restrict ourselves to the finite-dimensional case, that is, to the discretized problem~\eqref{eq:value_I_h}. We also assume that the discretization basis $(\chi_j)$ is orthonormal, such that $\bS=I_{2N_b}$, the $(2N_b)\times (2N_b)$ identity matrix.  
Finally, we only consider states which are invariant under time-reversal symmetry like in Section~\ref{sec:time-reversal}. This means that the HFB state is described by real and symmetric matrices $G$ and $A$ such that
\begin{equation}
\begin{pmatrix}
0&0\\ 0&0\end{pmatrix}
\leq
\Upsilon:=\begin{pmatrix}
G&A\\ A&\ 1-G\end{pmatrix}
\leq
\begin{pmatrix}
1&0\\ 0&1\end{pmatrix}
\label{eq:constraint-discretized} 
\end{equation}
The energy is given by~\eqref{eq:discretized_energy_nospin},
\begin{equation}
\cE(\Upsilon)=2\tr(hG)+2\tr(G\,J(G))-\tr(G\,K(G))+\tr(A\,K(A)).
\label{eq:discretized_energy_nospin_bis}
\end{equation}
The extension to more general situations is straightforward.

The energy $\cE$ is continuous (it is indeed real-analytic) with respect to $\Upsilon$. Also the set
$\cK$ of density matrices $\Upsilon$ of the form \eqref{eq:constraint-discretized} is compact in finite dimension. Hence minimizers always exist and, as we have seen, they solve the nonlinear equation
\begin{equation}
\Upsilon=\1_{(-\ii,0)}\Big(\bF_\Upsilon-\mu\bN\Big)+\delta,
\label{eq:discretized_eq_bis}
\end{equation}
where $\mu$ is a Lagrange multiplier chosen to ensure the constraint that $\tr(G)=N/2$.  Of course we must have $N/2\leq N_b$, the dimension of the (no-spin) discretization space $V_h$, otherwise the minimization set is always empty.

\subsection{Roothaan Algorithm}
The most natural technique used in practice to solve the equation~\eqref{eq:discretized_eq_bis} is a simple fixed point method~\cite{QueFlo-78,DecGog-80}. This is usually refered to as the \emph{Roothaan algorithm} in the chemistry literature~\cite{Roothaan-51}. The iteration scheme is the following
\begin{equation}
\Upsilon_{n+1}=\1_{(-\ii,0)}\Big(\bF_{\Upsilon_n}-\mu_{n+1}\bN\Big)+\delta_{n+1}.
\label{eq:def_Roothaan}
\end{equation}
At each step, one has to determine $\mu_{n+1}$ such as to satisfy the constraint $\tr(G_{n+1})=N/2$. If the operator $\bF_{\Upsilon_n}-\mu_{n+1}\bN$ has a trivial kernel, then $\delta_{n+1}\equiv0$. This is the usual situation encountered in practice. In the iteration~\eqref{eq:def_Roothaan}, the state is assumed to be pure at each step ($\Upsilon$ is an orthogonal projection). Recall that by Theorem~\ref{thm:variational_constrained} we know that minimizers of $\cE$ under a particle number constraint are always pure, under suitable assumptions on the interaction potential $W$. The algorithm is stopped when the commutator
$$\norm{\big[\Upsilon_{n},F_{\Upsilon_n}\big]}$$
and/or when the variation of the HFB state
$$\norm{\Upsilon_{n+1}-\Upsilon_{n}}$$
are smaller than a prescribed $\varepsilon$.

Our purpose in this section is to study the behavior of the Roothaan algorithm~\eqref{eq:def_Roothaan}. In the Hartree-Fock case, it was shown in a fundamental work of Cancès and Le Bris~\cite{CanBri-00,CanBri-00a}, that the algorithm converges or oscillate between two points, none of them being a solution to the equation~\eqref{eq:discretized_eq_bis}. This result was recently improved by Levitt in~\cite{Levitt-12}. We will explain that the results of Cancès-Le Bris and Levitt can be generalized to the HFB model. Actually, in a discretization space of dimension $N_b$, HFB is equivalent to a Hartree-Fock-like minimization problem in dimension $2N_b$, with additional constraints. The adaptation of the previously cited works in the HF case reduces to handling these contraints.

In order to avoid the convergence problems of the Roothaan algorithm, Cancès and Le Bris have proposed the \emph{Optimal Damping Algorithm} (ODA). We will study the equivalent of this algorithm in HFB theory in Section~\ref{sec:ODA}.

\medskip

To start with, we show that the Roothaan algorithm is well defined, in the sense that for any HFB state $\Upsilon_n$, there exists $(\Upsilon_{n+1},\mu_{n+1},\delta_{n+1})$ solving~\eqref{eq:def_Roothaan}. To this end, we follow~\cite{CanBri-00,CanBri-00a} and introduce the auxillary functional
\begin{equation}
\tilde\cE(\Upsilon,\Upsilon'):=\tr(hG)+\tr(hG')+2\tr(G\,J(G'))-\tr(G\,K(G'))+\tr(A\,K(A'))
\end{equation}
as well as the variational problem
\begin{equation}
I_{\Upsilon}(\lambda):=\min_{\Upsilon'}\left\{\tilde\cE(\Upsilon,\Upsilon')\ :\ \tr(G')=\lambda\right\}
\label{eq:def_min_one-variable}
\end{equation}
which consists in minimizing over $\Upsilon'$ with $\Upsilon$ fixed. The matrix $\Upsilon'$ must be an admissible HFB state which, in our context, means that
\begin{equation}
\begin{pmatrix}
0&0\\ 0&0\end{pmatrix}
\leq
\Upsilon':=\begin{pmatrix}
G'&A'\\ A'&\ 1-G'\end{pmatrix}
\leq
\begin{pmatrix}
1&0\\ 0&1\end{pmatrix},\qquad (G')^T=\overline{G'}=G',\quad (A')^T=\overline{A'}=A'.
\end{equation}
Recall that we have chosen an orthonormal basis and that the spin has been eliminated. It is clear that~\eqref{eq:def_min_one-variable} admits at least one solution $\Upsilon'$, as soon as $0\leq\lambda\leq N_b$, where we recall that $N_b$ is the dimension of the discretization space $V_h$. The following says that these solutions are exactly those solving the equation of the Roothaan method.

\begin{lemma}[The Roothaan algorithm is well defined]\label{lem:Roothaan}
The function $\lambda\in[0,N_b]\mapsto I_\Upsilon(\lambda)$ is convex, hence left and right differentiable. For any $\lambda\in[0,N_b]$, the minimizers $\Upsilon'$ of $I_\Upsilon(\lambda)$ are exactly the states of the form
\begin{equation}
\Upsilon'=\1_{(-\ii,0)}\big(\bF_{\Upsilon}-\mu'\bN\big)+\delta' 
\label{eq:equation_Upsilon}
\end{equation}
where  $\mu'\in[\partial_-I_\Upsilon(\lambda),\partial_+I_\Upsilon(\lambda)]$ and $0\leq\delta'\leq \1_{\{0\}}(\bF_{\Upsilon}-\mu'\bN)$. If $0\notin\sigma( \bF_\Upsilon-\mu'\bN)$, then $\delta'\equiv0$ and $\Upsilon'$ is unique, for any such $\mu'\in[\partial_-I_\Upsilon(\lambda),\partial_+I_\Upsilon(\lambda)]$.
\end{lemma}

\begin{proof}
To see that $I_\Upsilon(\lambda)$ is convex, let $0\leq\lambda_1<\lambda_2\leq N_b$ and let $\Upsilon'_i$ be a minimizer for $I_\Upsilon(\lambda_i)$ with $i=1,2$. Then $t\Upsilon_1'+(1-t)\Upsilon_2'$ is a test state for the problem $I_\Upsilon(t\lambda_1+(1-t)\lambda_2)$. Therefore it holds
\begin{align*}
 I_\Upsilon(t\lambda_1+(1-t)\lambda_2)\leq \tilde\cE(\Upsilon,t\Upsilon_1'+(1-t)\Upsilon_2')&=t\tilde\cE(\Upsilon,\Upsilon_1')+(1-t)\tilde\cE(\Upsilon,\Upsilon_2')\\
&=tI_\Upsilon(\lambda_1)+(1-t)I_\Upsilon(\lambda_2). 
\end{align*}
Then, by convexity we get that $I_\Upsilon(\lambda')\geq I_\Upsilon(\lambda)+\mu(\lambda'-\lambda)$ for any $\lambda'\in[0,N_b]$ and any $\mu\in[\partial_-I_\Upsilon(\lambda),\partial_+I_\Upsilon(\lambda)]$. Thus
\begin{align*}
I_\Upsilon(\lambda)-\mu\lambda&=\min\{I_\Upsilon(\lambda')-\mu\lambda'\ :\ 0\leq\lambda'\leq N_b\}\\
& =\min_{\Upsilon'}\big\{\tilde\cE(\Upsilon,\Upsilon')-\mu\tr(G')\big\}\\
&=\tr(hG)+\frac12\min_{\Upsilon'}\tr(\bF_\Upsilon-\mu\bN)\Upsilon'.
\end{align*}
In the previous two mins, $\Upsilon'$ is varied over all possible HFB states, without any particle number constraint. It is well known that the minimizers of the problem on the right side are exactly the solutions of the equation~\eqref{eq:equation_Upsilon}.
\end{proof}

Lemma~\ref{lem:Roothaan} tells us that for any given $\Upsilon_n$, there always exists at least one solution $(\Upsilon_{n+1},\mu_{n+1},\delta_{n+1})$ of the equation~\eqref{eq:def_Roothaan}. It is obtained by solving the minimization problem $I_{\Upsilon_n}(N/2)$, and one has to take
$\mu_{n+1}\in[\partial_-I_{\Upsilon_n}(N/2),\partial_+I_{\Upsilon_n}(N/2)]$. 
We can always take by convention
$$\mu_{n+1}:=\frac{\partial_-I_{\Upsilon_n}(N/2)+\partial_+I_{\Upsilon_n}(N/2)}{2}.$$
However, $\Upsilon_{n+1}$ is not uniquely defined yet because of the possibility of having $\delta_{n+1}\neq0$. As we have seen it is unique when 
\begin{equation}
0\notin \sigma\big(\bF_{\Upsilon_n}-\mu_{n+1}\bN\big).
\label{eq:zero_not_in_spec} 
\end{equation}
In this section we always assume that it is possible to find $(\Upsilon_{n+1},\mu_{n+1},\delta_{n+1})$ exactly. In practice, we will only know $(\Upsilon_{n+1},\mu_{n+1},\delta_{n+1})$ approximately. Later in Section~\ref{sec:constraint} we explain how to do this numerically. We will also see that the condition~\eqref{eq:zero_not_in_spec} is ``very often'' satisfied. This vague statement is made precise in Lemma~\ref{lem:generic_nu} below.
Following Cancès and Le Bris, we now introduce the concept of uniform well-posedness.

\begin{definition}[Uniform well-posedness]
We say that for a given initial HFB state $\Upsilon_0$, the sequence $(\Upsilon_n)$ generated by the Roothaan algorithm is \emph{uniformly well posed} when there exists $\eta>0$ such that 
\begin{equation}
\left|\bF_{\Upsilon_n}-\mu_{n+1}\bN\right|\geq \eta 
\label{eq:uniform-well-posed}
\end{equation}
for all $n\geq0$, where $\mu_{n+1}=\big(\partial_-I_{\Upsilon_n}(N/2)+\partial_+I_{\Upsilon_n}(N/2)\big)/{2}$.
\end{definition}

Note that the condition $|\bF_{\Upsilon_n}-\mu_{n+1}\bN|\geq \eta$ is equivalent to $(-\eta,\eta)\cap\sigma(\bF_{\Upsilon_n}-\mu_{n+1}\bN)=\emptyset$.
Later in Section~\ref{sec:constraint} we will make several comments concerning Assumption~\eqref{eq:uniform-well-posed}.

We have seen that the sequence generated by the Roothaan algorithm can be obtained by solving the minimization problem
$$I_{\Upsilon_n}=\min_{\Upsilon'}\tilde\cE(\Upsilon_n,\Upsilon').$$
Since $\tilde\cE(\Upsilon,\Upsilon')=\tilde\cE(\Upsilon',\Upsilon)$, we conclude that the Roothaan algorithm is the same as minimizing $\tilde\cE$ with respect to its first and second variables one after another, inductively. This fact allows to prove the following result, which is the HFB equivalent of Theorem 7 in~\cite{CanBri-00a} and Theorem 5.1 in~\cite{Levitt-12} in the HF case.

\begin{theorem}[Convergence of the Roothaan algorithm]\label{thm:Roothaan}
Assume that $0<N/2<N_b$.
Let $\Upsilon_0$ be an initial HFB state such that the sequence $(\Upsilon_n)$ generated by the Roothaan algorithm is \emph{uniformly well posed}. Then 
\begin{itemize}
\item The sequence $\tilde\cE(\Upsilon_{2n},\Upsilon_{2n+1})$ decreases towards a critical value of $\tilde\cE$;

\medskip

\item The sequence $(\Upsilon_{2n},\Upsilon_{2n+1})$ converges towards a critical point $(\Upsilon,\Upsilon')$ of $\tilde\cE$;

\medskip

\item If $\Upsilon=\Upsilon'$, then this state is a solution of the original HFB equation~\eqref{eq:discretized_eq_bis}, but if $\Upsilon\neq\Upsilon'$, then none of these two states is a solution to~\eqref{eq:discretized_eq_bis}.
\end{itemize}
\end{theorem}

Theorem~\ref{thm:Roothaan} says that (provided it is uniformly well posed) the sequence $\Upsilon_n$ will either converge to a solution of the self-consistent Equation~\eqref{eq:discretized_eq_bis}, or oscillate between two points $\Upsilon$ and $\Upsilon'$, none of them being a solution to the desired equation.

\begin{proof}
We split the proof into several steps. 

\bigskip

\noindent\textbf{Step 1: \textit{$\mu_n$ is uniformly bounded.}} It will be very useful to know that the sequence $\mu_n$ is uniformly bounded. The following says that, as soon as we fix $\tr(G)=N/2$ with $0<N/2<N_b$, the chemical potential $\mu$ cannot be too negative and too positive.

\begin{lemma}[Bounds on the multiplier $\mu$]\label{lem:bound_mu}
Let $\Upsilon'$ be any fixed HFB state and  
$$\Upsilon_\mu:=\1_{(-\ii,0)}\big(\bF_{\Upsilon'}-\mu\bN\big)=\begin{pmatrix}
G_\mu&A_\mu\\ A_\mu & \ 1-G_\mu
\end{pmatrix}
$$
the corresponding HFB ground state at chemical potential $\mu$. 
There exists a constant $C$ which is independent of $\Upsilon'$ and $\mu$, such that 
\begin{equation}
\forall \mu\leq -C,\quad \tr G_\mu\leq \frac{C}{|\mu|}\qquad\text{and}\qquad 
\forall \mu\geq C,\quad \tr G_\mu \geq N_b-\frac{C}{\mu}.
\end{equation}
\end{lemma}

The lemma says that the average number of particles in $\1_{(-\ii,0)}\big(\bF_\Upsilon-\mu\bN\big)$ tends to $N_b$ when $\mu\to\ii$ whereas it tends to 0 when $\mu\to-\ii$, this \emph{uniformly with respect to the state $\Upsilon'$} used to build the mean-field operator $\bF_{\Upsilon'}$.

\begin{proof}
We first remark that there exists a constant $C$ such that 
\begin{equation}
\norm{F_{\Upsilon'}}\leq C
\label{eq:bound_F_unif}
\end{equation}
for any HFB state $\Upsilon'$. This follows from the fact that $F_{\Upsilon'}$ is continuous with respect to $\Upsilon'$ and that the latter lives in a compact set since we always have $0\leq\Upsilon\leq 1$. The chosen norm for $\norm{F_{\Upsilon'}}$ does not matter since we are in finite dimension.
Now, for $\mu$ large enough we can use regular perturbation theory
and obtain that
\begin{multline*}
\norm{\1_{(-\ii,0)}\big(F_{\Upsilon'}-\mu\bN\big)-\1_{(-\ii,0)}\big(-\mu\bN\big)}\\
= \norm{\1_{(-\ii,0)}\left(\frac{F_{\Upsilon'}}{|\mu|}-\frac{\mu}{|\mu|}\,\bN\right)-\1_{(-\ii,0)}\left(-\frac{\mu}{|\mu|}\,\bN\right)}\leq \frac{C}{|\mu|}.
\end{multline*}
Note that 
$$\1_{(-\ii,0)}\left(-\frac{\mu}{|\mu|}\,\bN\right)=\begin{cases}
\begin{pmatrix}
1&0\\ 0&0
\end{pmatrix}&\text{for $\mu>0$,}\\[0.4cm]
\begin{pmatrix}
0&0\\ 0&1
\end{pmatrix}&\text{for $\mu<0$.}
\end{cases}$$
Taking the trace against $\bN$ gives the result.
\end{proof}

From Lemma~\ref{lem:bound_mu} we deduce that our sequence $\mu_n$ is bounded. Indeed, since we have by construction $\tr(G_{n+1})=N/2$ with $0<N/2<N_b$, we must have 
\begin{equation}
-\max\left(C\,,\,\frac{2C}N\right)\leq \mu_{n+1}\leq \max\left(C\,,\,\frac{C}{N_b-N/2}\right)
\label{eq:bound_mu_n} 
\end{equation}
otherwise $\tr(G_{n+1})$ would be too small or too large.

\bigskip

\noindent\textbf{Step 2: \textit{convergence of $\tilde\cE(\Upsilon_n,\Upsilon_{n+1})$.}}
We now follow~\cite{CanBri-00,CanBri-00a,Cances-00b}.
At each step, we know from Lemma~\ref{lem:Roothaan} that $\Upsilon_{n+1}$ is a solution of the minimization problem 
$\min_{\Upsilon'}\tilde\cE(\Upsilon_n,\Upsilon')$. In particular, we deduce that 
\begin{equation}
\tilde\cE(\Upsilon_{n},\Upsilon_{n+1})\leq \tilde\cE(\Upsilon_{n},\Upsilon_{n-1})=\tilde\cE(\Upsilon_{n-1},\Upsilon_{n}).
\label{eq:energy-decreasing} 
\end{equation}
Thus the sequence of real numbers $\tilde\cE(\Upsilon_{n},\Upsilon_{n+1})$ is non-increasing. It is also bounded from below, hence it converges to a limit $\ell$. We now use the uniform well-posedness to prove an inequality which is more precise than~\eqref{eq:energy-decreasing}. We remark that
\begin{align}
\tilde\cE(\Upsilon_{n},\Upsilon_{n-1})-\tilde\cE(\Upsilon_{n},\Upsilon_{n+1})&=\frac12\tr\bF_{\Upsilon_n}\big(\Upsilon_{n-1}-\Upsilon_{n+1}\big)\nonumber\\
&=\frac12\tr\big(\bF_{\Upsilon_n}-\mu_{n+1}\bN\big)\big(\Upsilon_{n-1}-\Upsilon_{n+1}\big)\nonumber\\
&\geq\frac12\tr\big|\bF_{\Upsilon_n}-\mu_{n+1}\bN\big|\big(\Upsilon_{n-1}-\Upsilon_{n+1}\big)^2\nonumber\\
&\geq\frac\eta2\tr\big(\Upsilon_{n-1}-\Upsilon_{n+1}\big)^2=\frac\eta2 \norm{\Upsilon_{n-1}-\Upsilon_{n+1}}^2.\label{eq:estim_diff_carres}
\end{align}
In the above calculation we have used that $\tr\bN\,\Upsilon_{n+1}=\tr\bN\,\Upsilon_{n-1}=N-N_b$. We have also used that $\Upsilon_{n+1}$ is the negative spectral projector of $\bF_{\Upsilon_n}-\mu_{n+1}\bN$, such that we can write
$$\big(\bF_{\Upsilon_n}-\mu_{n+1}\bN\big)=\big|\bF_{\Upsilon_n}-\mu_{n+1}\bN\big|\big(\Upsilon_{n+1}^\perp-\Upsilon_{n+1}\big).$$
Finally, we have used that
$0\leq\gamma\leq1$
is equivalent to $(\gamma-P)^2\leq P^\perp(\gamma-P)P^\perp-P(\gamma-P)P$,
for any orthogonal projector $P$, thus
$$\Upsilon_{n+1}^\perp\big(\Upsilon_{n-1}-\Upsilon_{n+1}\big)\Upsilon_{n+1}^\perp-\Upsilon_{n+1}\big(\Upsilon_{n-1}-\Upsilon_{n+1}\big)\Upsilon_{n+1}\geq \big(\Upsilon_{n-1}-\Upsilon_{n+1}\big)^2.$$
Since $\tilde\cE(\Upsilon_n,\Upsilon_{n+1})$ converges to a limit $\ell$, we deduce that
$$\sum_{n\geq1}\norm{\Upsilon_{n+1}-\Upsilon_{n-1}}^2<\ii.$$
In particular, $\norm{\Upsilon_{n+1}-\Upsilon_{n-1}}\to0$ which is called \emph{numerical convergence} in~\cite{CanBri-00,CanBri-00a}.

\bigskip

\noindent\textbf{Step 3: \textit{convergence of $\Upsilon_{2n}$ and $\Upsilon_{2n+1}$.}}
In order to upgrade the numerical convergence to true convergence, we use a recent method of Levitt~\cite{Levitt-12}. 
Namely we show that
\begin{equation}
\left(\tilde\cE(\Upsilon_{n},\Upsilon_{n-1})-\ell\right)^\theta-\left(\tilde\cE(\Upsilon_{n},\Upsilon_{n+1})-\ell\right)^\theta\geq\frac{\eta'}2\norm{\Upsilon_{n-1}-\Upsilon_{n+1}}
\label{eq:estim_diff_theta}
\end{equation}
for a well chosen $0<\theta\leq1/2$. Summing over $n$ and using the convergence of $\tilde\cE(\Upsilon_{n},\Upsilon_{n-1})$, hence of $(\tilde\cE(\Upsilon_{n},\Upsilon_{n-1})-\ell)^\theta$, then gives the convergence of $\Upsilon_{2n}$ and $\Upsilon_{2n+1}$.

For the proof of~\eqref{eq:estim_diff_theta}, we argue as follow. 
Consider a (real, no-spin) pure HFB state $\Upsilon$. It is possible to parametrize the manifold of pure HFB states around $\Upsilon$ by 
using Bogoliubov transformations as follows:
$$H\mapsto e^{H}\Upsilon e^{-H}$$
where $H$ is assumed to be of the form
$$\begin{pmatrix}
h & p\\ -p & -h
\end{pmatrix},\qquad h^T=-\overline{h}=-h,\quad p^T=\overline{p}=p.$$
These constraints ensure that $iH$ is a self-adjoint Hamiltonian such that $e^{H}=e^{-i(iH)}$ is a Bogoliubov rotation. They also ensure that $e^{H}\Upsilon e^{-H}$ stays real. That $H\mapsto e^{H}\Upsilon e^{-H}$ is a local chart of the manifold of pure HFB states around $\Upsilon$ follows from the arguments in~\cite{BacLieSol-94} as well as simple considerations in linear algebra. 

Let us now consider the energy $\tilde\cE$ in a neighborhood of any fixed $(\Upsilon,\Upsilon')$. The map
\begin{equation*}
f:(H,H')\mapsto \tilde\cE\big(e^{H}\Upsilon e^{-H},e^{H'}\Upsilon' e^{-H'}\big)
-\frac{\mu'}2\,\tr\bN e^{H}\Upsilon e^{-H} - \frac{\mu}2\,\tr\bN e^{H'}\Upsilon' e^{-H'} 
\end{equation*}
is real analytic in a neighborhood of $(0,0)$ for any fixed $\mu,\mu'\in\R$ and any fixed pure HFB states $(\Upsilon,\Upsilon')$. The {\L}ojasiewicz inequality (Theorem 2.1 in~\cite{Levitt-12}) then tells us that there exist $0<\theta\leq1/2$ and a constant $\kappa>0$ such that $\norm{H}+\norm{H'}\leq \kappa$ implies 
$$|f(H,H')-f(0)|^{1-\theta}\leq \kappa^{-1}\Big(|\nabla_{H}f(H,H')|+|\nabla_{H'}f(H,H')|\Big).$$
A simple computation shows that 
$$\nabla_{H}f(H,H')=\frac{1}2\big[\bF_{\Upsilon'}-\mu'\bN\,,\,e^{H}\Upsilon e^{-H}\big],\qquad \nabla_{H'}f(H,H')=\frac{1}2\big[\bF_{\Upsilon}-\mu\bN\,,\,e^{H'}\Upsilon' e^{-H'}\big].$$
If we rephrase all this in our setting, this means that for any fixed pure HFB states $(\Upsilon_1,\Upsilon_1')$ and any $\mu,\mu'\in\R$, there is a $\kappa>0$ such that for any $(\Upsilon_2,\Upsilon_2')$ another pure HFB state which is at most at a distance $\kappa$ from $(\Upsilon_1,\Upsilon_1')$, we have
\begin{multline}
\Big|\tilde\cE\big(\Upsilon_1,\Upsilon_1'\big) - \tilde\cE\big(\Upsilon_2,\Upsilon_2'\big) +\mu'\tr\bN (G_2-G_1) + \mu\tr\bN (G_2'-G_1')\Big|^{1-\theta}\\
\leq \kappa^{-1} \Big(\norm{\big[\bF_{\Upsilon_2'}-\mu'\bN\,,\,\Upsilon_2\big]}+\norm{\big[\bF_{\Upsilon_2}-\mu\bN\,,\,\Upsilon_2'\big]}\Big).
\end{multline}
The constants $\kappa$ and $\theta$ depend on $\mu$, $\mu'$ and of the reference point $(\Upsilon_1,\Upsilon_1')$. But they stay positive as soon as $\mu$, $\mu'$ and $(\Upsilon_1,\Upsilon_1')$ stay in a compact set. By a simple compactness argument, we therefore deduce that there exists a neighborhood of the compact set 
$$\left\{(\Upsilon,\Upsilon')\ :\ \tilde\cE(\Upsilon,\Upsilon') = \ell,\ \tr(G)=\tr(G')=N/2\right\}$$
such that for any $(\Upsilon,\Upsilon')$ in this neighborhood and $\mu,\mu'$ in a compact set in $\R$, we have
\begin{multline}
\Big|\tilde\cE\big(\Upsilon,\Upsilon'\big) - \ell +\mu'\big(N/2-\tr\bN G\big)+\mu\big(N/2-\tr\bN G'\big)\Big|^{1-\theta}\\
\leq \kappa^{-1} \Big(\norm{\big[\bF_{\Upsilon'}-\mu'\bN\,,\,\Upsilon\big]}+\norm{\big[\bF_{\Upsilon}-\mu\bN\,,\,\Upsilon'\big]}\Big)
\end{multline}
for some $0<\theta\leq1/2$ and some $\kappa>0$. We recall that $\ell$ is by definition the limit of $\tilde\cE\big(\Upsilon_n,\Upsilon_{n+1}\big)$.

Recall our inequality~\eqref{eq:bound_mu_n} which says that $\mu_n$ is uniformly bounded. Also, we know that $\tilde\cE(\Upsilon_n,\Upsilon_{n+1})$ converges to $\ell$ so, for $n$ large enough, $(\Upsilon_n,\Upsilon_{n+1})$ must be in the neighborhood of the level set $\ell$. Choosing $\mu=\mu_{n+1}$ and $\mu'=\mu_{n}$ and using that $G_n$ and $G_{n+1}$ have the correct trace, we get the estimate
\begin{align*}
\Big(\tilde\cE\big(\Upsilon_n,\Upsilon_{n-1}\big) - \ell \Big)^{1-\theta}
&\leq \kappa^{-1} \left(\norm{\big[\bF_{\Upsilon_{n}}-\mu_{n+1}\bN\,,\,\Upsilon_{n-1}\big]}+\norm{\big[\bF_{\Upsilon_{n-1}}-\mu_{n}\bN\,,\,\Upsilon_{n}\big]}\right)\\
&=\kappa^{-1} \norm{\big[\bF_{\Upsilon_{n}}-\mu_{n+1}\bN\,,\,\Upsilon_{n-1}-\Upsilon_{n+1}\big]}\\
&\leq C\norm{\Upsilon_{n-1}-\Upsilon_{n+1}}
\end{align*}
for $n$ large enough.
Here we have used that $\Upsilon_{n}$ commutes with $\bF_{\Upsilon_{n-1}}-\mu_{n}\bN$ and that $\Upsilon_{n+1}$ commutes with $\bF_{\Upsilon_{n}}-\mu_{n+1}\bN$ by construction, and that 
$\norm{\bF_{\Upsilon_n}}$ and $\mu_{n+1}$ are both uniformly bounded.
In order to conclude, we use the concavity of $x\mapsto x^\theta$ and~\eqref{eq:estim_diff_carres} like in~\cite{Levitt-12} to obtain 
\begin{align*}
&\Big(\tilde\cE\big(\Upsilon_n,\Upsilon_{n-1}\big) - \ell \Big)^{\theta}-\Big(\tilde\cE\big(\Upsilon_n,\Upsilon_{n+1}\big) - \ell \Big)^{\theta}\\
&\qquad\qquad\qquad\geq \frac{\theta}{\Big(\tilde\cE\big(\Upsilon_n,\Upsilon_{n-1}\big) - \ell \Big)^{1-\theta}}\Big(\tilde\cE\big(\Upsilon_n,\Upsilon_{n-1}\big)-\tilde\cE\big(\Upsilon_n,\Upsilon_{n+1}\big)\Big)\\
&\qquad\qquad\qquad\geq \frac{\eta\,\theta}{2\Big(\tilde\cE\big(\Upsilon_n,\Upsilon_{n-1}\big) - \ell \Big)^{1-\theta}}\norm{\Upsilon_{n+1}-\Upsilon_{n-1}}^2\\
&\qquad\qquad\qquad\geq \eta\theta/(2C)\,\norm{\Upsilon_{n+1}-\Upsilon_{n-1}}
\end{align*}
by~\eqref{eq:estim_diff_carres}. This concludes the proof of the inequality~\eqref{eq:estim_diff_theta}, hence the proof of the convergence of $(\Upsilon_{2n},\Upsilon_{2n+1})$, towards some pure HFB states $(\Upsilon,\Upsilon')$.

\bigskip

\noindent\textbf{Step 4: \textit{the limit $(\Upsilon,\Upsilon')$ of $(\Upsilon_{2n},\Upsilon_{2n+1})$ is a critical point of $\tilde\cE$.}}
Since we have $\Upsilon_{2n}\to\Upsilon$ and $\Upsilon_{2n+1}\to\Upsilon'$, we deduce that $\bF_{\Upsilon_{2n}}\to \bF_{\Upsilon}$ and $\bF_{\Upsilon_{2n+1}}\to \bF_{\Upsilon'}$, by continuity of the map $\Upsilon\mapsto\bF_\Upsilon$. Extracting a subsequence, we can assume that $\mu_{2n_k}\to\mu'$ and $\mu_{2n_k+1}\to\mu$. We have 
$$\Upsilon_{2n_k}=\1_{(-\ii,0)}(\bF_{\Upsilon_{2n_k-1}}-\mu_{2n_k}\bN),\qquad \Upsilon_{2n_k+1}=\1_{(-\ii,0)}(\bF_{\Upsilon_{2n_k}}-\mu_{2n_k+1}\bN)$$ 
and, by uniform well-posedness,
$$\big|\bF_{\Upsilon_{2n_k-1}}-\mu_{2n_k}\bN\big|\geq \eta,\qquad \big|\bF_{\Upsilon_{2n_k}}-\mu_{2n_k+1}\bN\big|\geq \eta.$$
Passing to the limit $k\to\ii$ we get
$$\Upsilon=\1_{(-\ii,0)}(\bF_{\Upsilon'}-\mu'\bN)\quad\text{and}\quad \Upsilon'=\1_{(-\ii,0)}(\bF_{\Upsilon}-\mu\bN).$$
This exactly means that $(\Upsilon,\Upsilon')$ is a critical point of $\tilde\cE$ on $\cP_h(N/2)\times\cP_h(N/2)$. Note that we have also
$\big|\bF_{\Upsilon'}-\mu'\bN\big|\geq \eta$ and $\big|\bF_{\Upsilon}-\mu\bN\big|\geq \eta$.

The remaining statements are verified exactly like in the HF case. This concludes the proof of Theorem~\ref{thm:Roothaan}.\hfill\qed
\end{proof}

\subsection{Optimal Damping Algorithm}\label{sec:ODA}

In the previous section we have studied the convergence properties of the Roothaan algorithm, which consists in solving the self-consistent equation by a fixed point method. We have seen that the algorithm can either converge or oscillate between two states, none of them being a solution to the problem.

Examples of such oscillations in quantum chemistry have been exhibited by Cancès and Le Bris~\cite{CanBri-00,CanBri-00a}. In this case the potential $W$ is repulsive and there is no pairing. In order to cure this problem of oscillations, Cancès and Le Bris proposed in~\cite{CanBri-00a} a \emph{relaxed algorithm} called the \emph{Optimal Damping Algorithm} (ODA). This method makes use of the important fact that one can minimize over mixed states and get the same ground state as when minimizing over pure states only (Theorem~\ref{thm:Lieb}).

The same oscillations can \emph{a priori} happen in HFB with an attractive potential $W$. They are frequently seen with the Roothaan algorithm and we will give several numerical examples later in Section~\ref{sec:num}. Even when the sequence $\Upsilon_n$ eventually converges towards a single state $\Upsilon$, these oscillations can slow down the convergence considerably. This phenomenon is well known in nuclear physics. Dechargé and Gogny already advocate in~\cite{DecGog-80} the use of a \emph{damping parameter} between two successive iterations, in order to \emph{``slow down the convergence on the density matrix. In this way the average field varies slowly and we can insure the convergence on the pairing tensor step by step''} (see~\cite{DecGog-80} page 1574). Even in the modern computations, this damping parameter is fixed all along the algorithm (Nathalie Pillet, private communication). 

We suggest to transpose the method of Cancès and Le Bris to the HFB setting by using an optimal damping parameter, chosen such as to minimize the energy. This means resorting to mixed states even if the final ground state is always a pure HFB state. This is theoretically justified when the assumptions of the Bach-Fröhlich-Jonsson Theorem~\ref{thm:variational_constrained} are fulfilled.

The ODA involves two density matrices $\Upsilon_n$ and $\tilde\Upsilon_n$. The HFB state $\Upsilon_n$ is always pure but $\tilde\Upsilon_n$ can (and will usually) be a mixed HFB state. The starting point $\Upsilon_0=\tilde\Upsilon_0$ being chosen, the sequence is then constructed by induction as follows:
\begin{enumerate}
\item One finds $(\Upsilon_{n+1},\mu_{n+1})$ solving
$$\Upsilon_{n+1}=\1_{(-\ii,0)}\big(\bF_{\tilde\Upsilon_n}-\mu_{n+1}\bN\big)\quad\text{and}\quad \tr(G_{n+1})=N/2.$$
This is always possible, by Lemma~\ref{lem:Roothaan} and we can take as before
$$\mu_{n+1}:=\frac{\partial_-I_{\tilde\Upsilon_n}(N/2)+\partial_+I_{\tilde\Upsilon_n}(N/2)}{2},$$
in case $0$ is in the spectrum of $\bF_{\Upsilon_n}-\mu_{n+1}\bN$.
\item One lets 
$$\tilde\Upsilon_{n+1}=t_{n+1}\tilde\Upsilon_{n}+(1-t_{n+1})\Upsilon_{n+1}$$
where the damping parameter $t_{n+1}\in[0,1]$ is chosen such as to minimize the (quadratic) function
$$t\mapsto \cE\big(t\tilde\Upsilon_{n}+(1-t)\Upsilon_{n+1}\big).$$
\item The algorithm is stopped when $\norm{[\Upsilon_{n},\bF_{\Upsilon_n}]}$ and/or $\norm{\Upsilon_{n+1}-\Upsilon_{n}}$ are smaller than a prescribed $\varepsilon$.
\end{enumerate}

\begin{figure}[h]
\centering
\includegraphics[width=10cm]{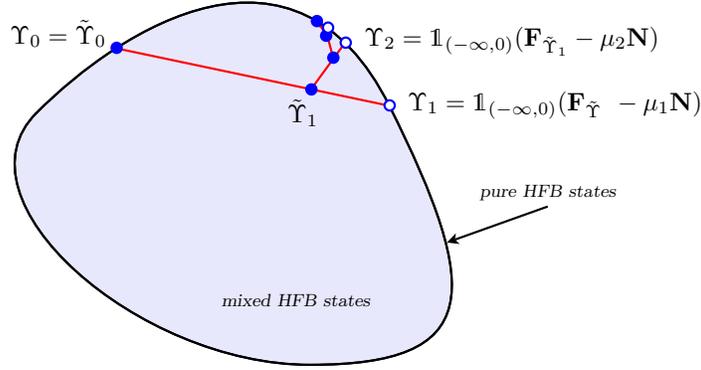}

\caption{The Optimal Damping Algorithm of Cancès \& Le Bris in the HFB case.\label{fig:ODA}}
\end{figure}

The general strategy of the ODA is displayed in Figure~\ref{fig:ODA}. By construction we see that $\cE(\tilde\Upsilon_n)$ is a non-increasing sequence. This guarantees the convergence of the ODA. The result is the following

\begin{theorem}[Convergence of the ODA]\label{thm:ODA}
Assume that $0<N/2<N_b$.
Let $\Upsilon_0=\tilde\Upsilon_0$ be an initial HFB state such that the sequence $(\Upsilon_n)$ generated by the ODA is \emph{uniformly well posed}, that is
\begin{equation}
\forall n,\qquad \big|\bF_{\tilde\Upsilon_n}-\mu_{n+1}\bN\big|\geq \eta>0.
\end{equation}
Then 
\begin{itemize}
\item The sequence $\cE(\tilde\Upsilon_{n})$ decreases towards a critical value of $\cE$;

\medskip

\item The sequence $\Upsilon_{n}$ \emph{numerically} converges towards a critical point $\Upsilon$ of $\cE$, in the sense that $\Upsilon_{n+1}-\Upsilon_{n}\to0$, $\Upsilon_{n+1}-\tilde\Upsilon_n\to0$ and that all the limit points $\Upsilon$ of subsequences of $(\Upsilon_n)$ solve $\Upsilon=\1_{(-\ii,0)}(\bF_\Upsilon-\mu\bN)$.
\end{itemize}
\end{theorem}

\begin{proof}
The proof is exactly the same as in the Hartree-Fock case~\cite{Cances-00b,CanBriMad-06} and we only sketch it. First we have by definition $\cE(\tilde\Upsilon_{n+1})\leq \cE(\tilde\Upsilon_n)$, so $\cE(\tilde\Upsilon_n)$ must converge to a limit $\ell$. Now we have 
\begin{equation*}
\cE(\tilde\Upsilon_{n+1})=\cE\big((1-t_{n+1})\tilde\Upsilon_{n}+t_{n+1}\Upsilon_{n+1}\big)=\cE(\tilde\Upsilon_{n})-t_{n+1}a_{n+1}+t_{n+1}^2b_{n+1}
\end{equation*}
where 
$$t_{n+1}={\rm argmin}_{t\in[0,1]}\big(-ta_{n+1}+t^2b_{n+1}\big)$$
and with
$$a_{n+1}:=\tr\bF_{\tilde\Upsilon_n}\big(\tilde\Upsilon_n-\Upsilon_{n+1}\big)=\tr|\bF_{\tilde\Upsilon_n}-\mu\bN|\big(\tilde\Upsilon_n-\Upsilon_{n+1}\big)^2\geq\eta \norm{\tilde\Upsilon_n-\Upsilon_{n+1}}^2,$$
\begin{multline*}
b_{n+1}=2\tr (\tilde G_{n+1}-G_n)J(\tilde G_{n+1}-G_n)-\tr (\tilde G_{n+1}-G_n)K(\tilde G_{n+1}-G_n)\\
+\tr (\tilde A_{n+1}-A_n)K(\tilde A_{n+1}-A_n).
\end{multline*}
In finite dimension we have $|b_{n+1}|\leq C\norm{\tilde\Upsilon_{n+1}-\Upsilon_n}^2\leq (C/\eta)a_{n+1}$. This can be used to prove that
$$-t_{n+1}a_{n+1}+t_{n+1}^2b_{n+1}\leq -\epsilon\, a_{n+1}\leq -\epsilon\,\eta\norm{\tilde\Upsilon_n-\Upsilon_{n+1}}^2$$
for some $\epsilon>0$ independent of $n$. This now proves that 
$$\sum_{n}\norm{\tilde\Upsilon_n-\Upsilon_{n+1}}^2<\ii,$$
hence that $\tilde\Upsilon_n-\Upsilon_{n+1}\to0$. In order to conclude the proof, we notice that 
$$\Upsilon_{n+1}-\Upsilon_{n}=\Upsilon_{n+1}-\tilde\Upsilon_{n}+(1-t_n)\big(\tilde\Upsilon_{n-1}-\Upsilon_{n}\big)$$
which finally implies
$$\sum_{n}\norm{\Upsilon_n-\Upsilon_{n+1}}^2<\ii\quad\text{and}\quad\sum_{n}\norm{\tilde\Upsilon_n-\tilde\Upsilon_{n+1}}^2<\ii.$$
Since $\Upsilon_{n+1}=\1_{(-\ii,0)}(\bF_{\tilde\Upsilon_n}-\mu_{n+1}\bN)$ by definition, the proof that any limit $\Upsilon$ of a subsequence of $(\Upsilon_n)$ satisfies the self-consistent equation is elementary.
\end{proof}

\subsection{Handling constraints}\label{sec:constraint}

Both the Roothaan algorithm and the ODA are based on Lemma~\ref{lem:Roothaan} which says that for any given $\bF_\Upsilon$, there exist $\mu'$, $\delta'$ and $\Upsilon'$ such that 
\begin{equation}
\begin{cases}
\Upsilon'=\1_{(-\ii,0)}\big(\bF_\Upsilon-\mu'\bN\big)+\delta',\\
\tr\bN \Upsilon'=N-N_b.
\end{cases}
\label{eq:iteration}
\end{equation}
The purpose of this section is to explain how to solve this problem numerically. To simplify our notation, we consider in this section a generic matrix
\begin{equation}
\bF=\begin{pmatrix}
h&p\\ p&-h
\end{pmatrix},\qquad\text{with  $p=\overline{p}=p^T$ and $h=\overline{h}=h^T$} 
\label{eq:form_F_abstract}
\end{equation}
and we study the problem consisting in finding $\Upsilon$, $\mu$ and $\delta$ such that
\begin{equation}
\begin{cases}
\Upsilon=\1_{(-\ii,0)}\big(\bF-\mu\bN\big)+\delta,\\
\tr\bN \Upsilon=N-N_b.
\end{cases}
\label{eq:iteration_abstract}
\end{equation}

Assume first that $p\equiv0$ (Hartree-Fock case). Then we have
$$\bF=\begin{pmatrix}
h&0\\ 0&-h
\end{pmatrix}$$
which commutes with $\bN$. The solution of~\eqref{eq:iteration_abstract} is then given by the \emph{aufbau principle},
$$\Upsilon=\begin{pmatrix}
G&0\\ 0&1-G
\end{pmatrix},\qquad G=\1_{(-\ii,\mu)}(h)+\delta$$
where $\mu$ is the $(N/2)$th eigenvalue of $h$, counted with multiplicity and $\delta$ lives in the corresponding eigenspace. 
Equivalently,
$$G=\sum_{i=1}^K v_iv_i^T+\sum_{i=K+1}^{K'}n_i\, v_iv_i^T$$
where the $v_i$'s solve the eigenvalue equation
$$h\,v_i=\epsilon_i\,v_i,$$
$K=\tr\1_{(-\ii,\epsilon_{N/2})}(h)$ is the dimension of the direct sum of all the eigenspaces corresponding to the eigenvalues $<\epsilon_{N/2}$ and $K'=\tr\1_{(-\ii,\epsilon_{N/2}]}(h)$ is the dimension of the direct sum of all the eigenspaces corresponding to the eigenvalues $\leq\epsilon_{N/2}$. The $n_i$'s are chosen such that
$$0\leq n_i\leq 1,\qquad K+\sum_{i=K+1}^{K'}n_i=\frac{N}2.$$
Therefore, finding $\Upsilon$, $\mu$ and $\delta$ in the Hartree-Fock case only requires to diagonalize $h$ once.

In the Hartree-Fock-Bogoliubov case ($p\neq0$), the situation is more complicated since $\bN$ does \emph{not} commute with $\bF$. Let us consider the real function
\begin{equation}
\nu_\bF:\mu\mapsto \nu(\mu)=\frac{\tr\bN\,\1_{(-\ii,0)}\big(\bF-\mu\bN\big)+N_b}2. 
\label{eq:def_nu}
\end{equation}
We are interested in solving the equation
$$\nu_\bF(\mu)=N/2.$$
In the Hartree-Fock case, $\nu_\bF$ is a non-decreasing piecewise constant function. There is a solution $\mu$ to $\nu_\bF(\mu')=N/2$ when $N/2$ belong to the range of $\nu_\bF$. Otherwise, one has to partially fill a shell using the matrix $\delta$.

In the Hartree-Fock-Bogoliubov case, $\nu_\bF$ is also non-decreasing and in general it is much smoother when $p\neq0$. The following lemma summarizes some important properties of $\nu_\bF$ in both the HF and HFB cases.

\begin{lemma}[Elementary properties of $\nu$]\label{lem:prop_nu}
Let $\bF$ be as in~\eqref{eq:form_F_abstract}. Then the function $\nu_\bF$ defined in~\eqref{eq:def_nu} is increasing with respect to $\mu$. It can only have finitely many jumps.
It satisfies for some constant $C$ depending only on $N_b$
\begin{itemize}
\item $\nu(\mu)\leq C/\mu$ for $\mu\leq -C$;
\item $\nu_\bF(\mu)\geq N_b-C/\mu$ for $\mu\geq C$.
\end{itemize}
If $0\notin \sigma\big(\bF-\mu\bN\big)$, then
\begin{equation}
\frac{d\nu_\bF}{d\mu}(\mu)=2\sum_{\substack{\epsilon_i<0\\ \epsilon_j>0}}\frac{\left|\pscal{v_j,\bN v_i}\right|^2}{\epsilon_j-\epsilon_i}\geq0
\label{eq:derivee_nu}
\end{equation}
where $(\bF-\mu\bN)v_i=\epsilon_i\,v_i$.
\end{lemma}

\begin{proof}
The behavior of $\nu_\bF$ for $|\mu|\gg1$ was already studied in Lemma~\ref{lem:bound_mu}.

The matrix $\bF-\mu\bN$ is a linear function of $\mu\in\R$, hence by~\cite{Kato}, we know that its eigenvalues form a set of real analytic functions. They cannot be constant because the matrix $\bN$ does not vanish. The eigenvalues of $\bF-\mu\bN$ all behave like $\pm\mu$ for large $\mu$, by perturbation theory. We conclude that $0$ can be an eigenvalue of $\bF-\mu\bN$ for a finite number of $\mu$'s, say $\mu_1<\cdots<\mu_K$. On the other hand, the map $\mu\mapsto \1_{(-\ii,0)}\big(\bF-\mu\bN\big)$ is real-analytic outside of the $\mu_k$'s (see~\cite{Kato}). So $\nu_\bF$ is itself real-analytic outside of this set and it can have at most a finite number of jumps. 

Outside of the $\mu_k$'s, it is possible to compute the derivative of $\nu_\bF$ by usual perturbation methods~\cite{Kato}. The answer is~\eqref{eq:derivee_nu} and the fact that $d\nu_\bF/d\mu\geq0$ proves that $\nu_\bF$ is increasing with respect to $\mu$, in between these points. That the jumps are all positive can be easily seen by an approximation argument using Lemma~\ref{lem:generic_nu} below. We skip the details. 
\end{proof}

The shape of the function $\nu_\bF$ is very different in the HF and HFB cases. For a Hartree-Fock state, the function $\nu_\bF$ is piecewise constant and it has jumps at the eigenvalues $\epsilon_1<\cdots<\epsilon_{N_b}$ of $h_{G}$. The size of the jumps is equal to the multiplicity of the associated eigenvalue. An HFB state will most always have a very smooth $\nu_\bF$. Of course, the smaller $p$ in the Hamiltonian $\bF$, the more $\nu_\bF$ looks like a step function.

In Figure~\ref{fig:fn_nu} below, we show the function $\nu_\bF$ for different values of the pairing term. More precisely, we have randomly chosen two symmetric real matrices $h$ and $p$ of size $N_b=5$, and we display the function $\nu_\bF$ when the pairing is replaced by $tp$ for $t=0$ (Hartree-Fock case), $t=0.1$ and $t=1$. Figure~\ref{fig:fn_eigenvals} is a plot of the eigenvalues of $\bF-\mu\bN$ for $t=0.1$, as functions of $\mu$. Note that there are some crossings of eigenvalues above and below the real line (recall that the spectrum is symmetric with respect to 0). But, around $0$ the crossings are avoided and there is a gap.

\begin{figure}[h!]
\centering
\includegraphics[width=9cm]{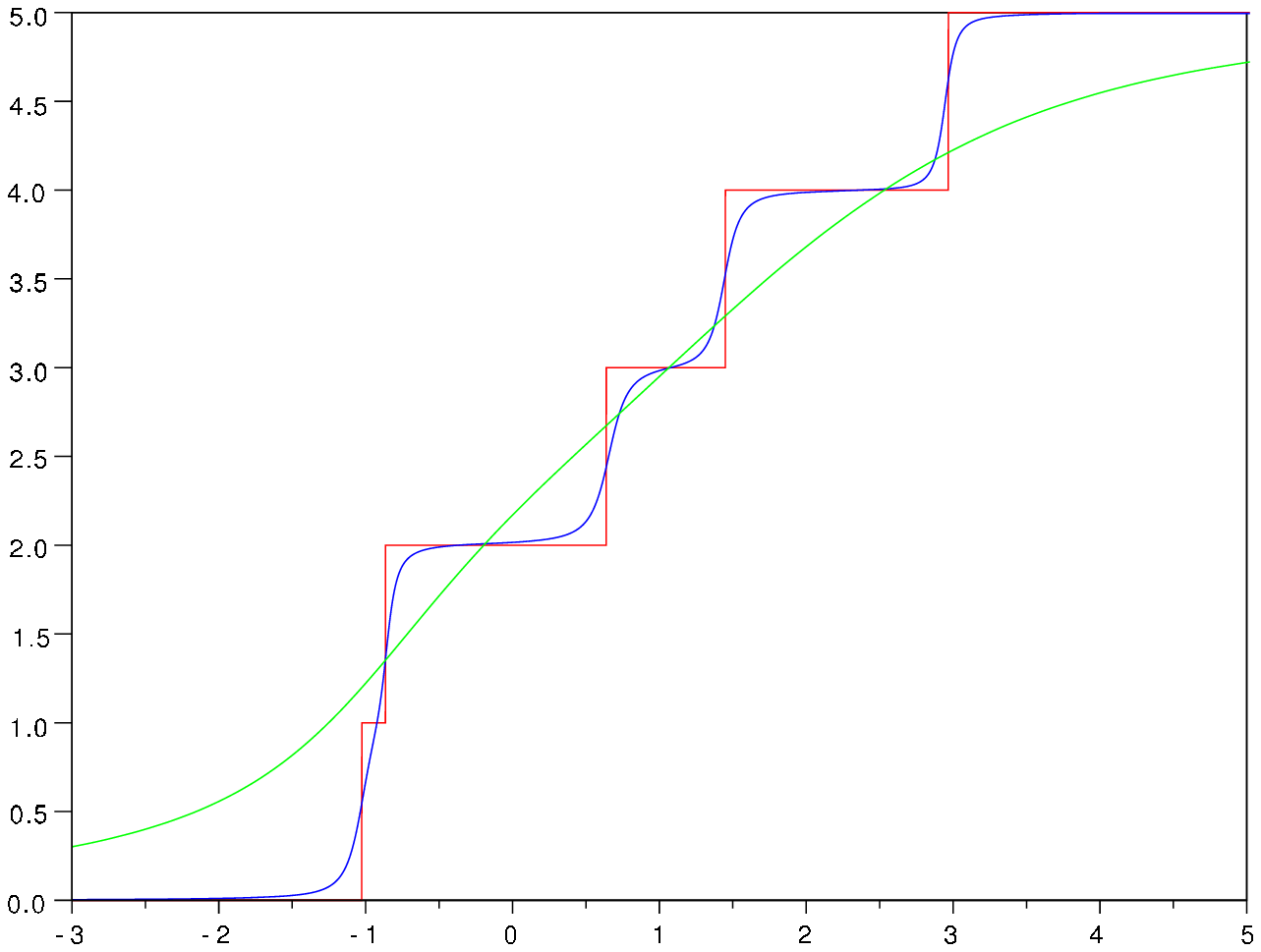}
\caption{The function $\nu_\bF(\mu)$ which gives the average number of particles in the state $\1_{(-\ii,0)}(\bF-\mu\bN)$, in terms of the chemical potential $\mu$. The pairing term in $\bF$ is equal to $tp$ with $t=0$ (Hartree-Fock case, red curve), $t=0.1$ (blue curve) and $t=1$ (green curve).\label{fig:fn_nu}}

\bigskip\bigskip\bigskip

\includegraphics[width=9cm]{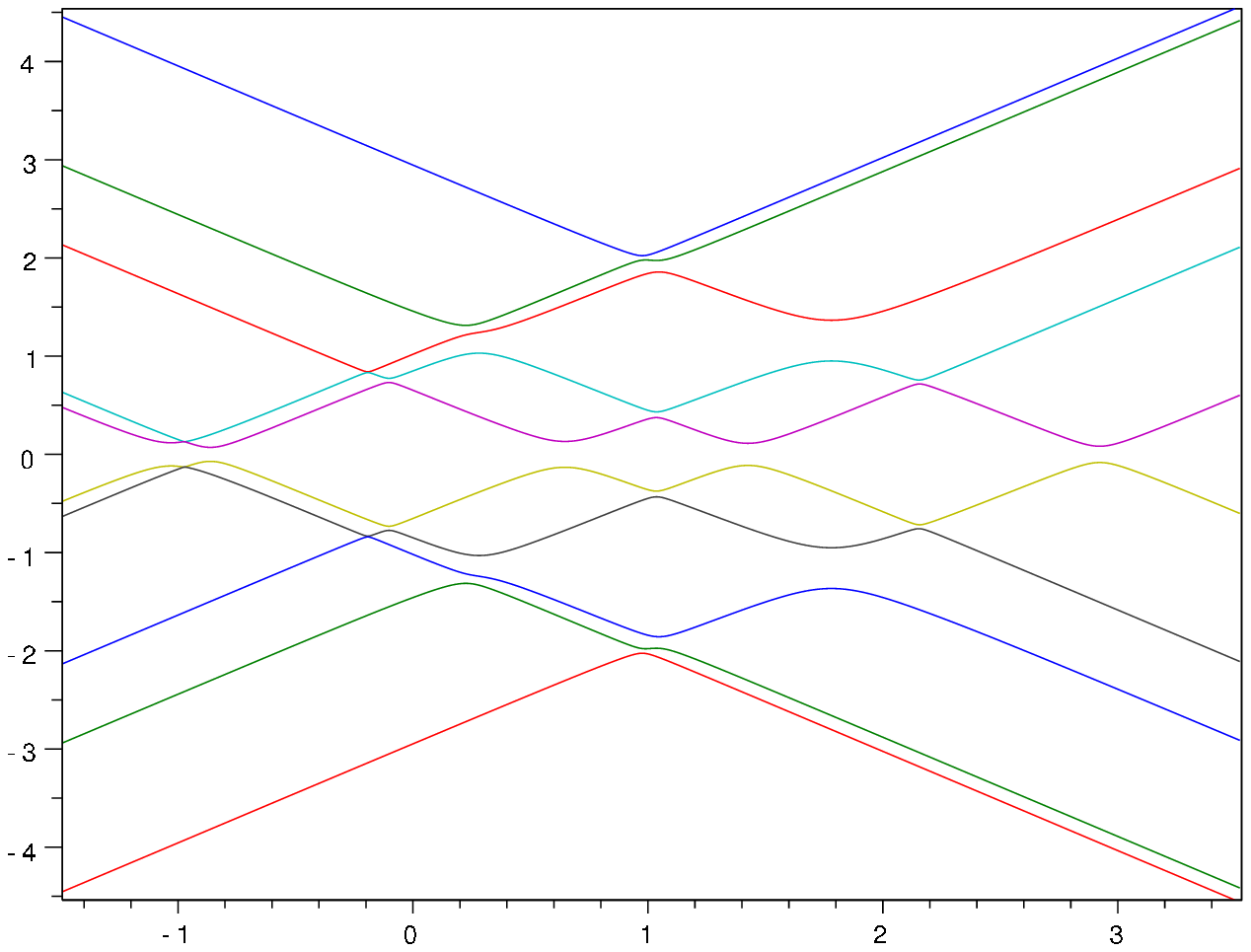}
\caption{The eigenvalues of $\bF-\mu\bN$ in terms of $\mu$ for $t=0.1$. \label{fig:fn_eigenvals}}
\end{figure}

If we repeat the numerical experiment with several random matrices $h$ and $p$, we \emph{never} see any jump for $\nu_\bF$. The purpose of the next result is to clarify this observation.

\begin{lemma}[Generic behavior of $\nu_\bF$]\label{lem:generic_nu}
The Fock matrix $\bF-\mu\bN$ is invertible if and only if $h\pm ip-\mu$ are invertible. More precisely, 
\begin{equation}
\min\sigma\big(|\bF-\mu\bN|\big)=\min\left(\norm{(h+ip-\mu)^{-1}}^{-1}\,,\,\norm{(h-ip-\mu)^{-1}}^{-1}\right).
\label{eq:estim_spec_F_mu}
\end{equation}
The set of real symmetric matrices $h$ and $p$ such that 
$$\sigma\big(\bF-\mu\bN\big)\cap\{0\}=\emptyset\quad \text{for all $\mu\in\R$}$$
is open and dense in $\{(h,p)\ :\ h=h^T=\overline{h},\ p=p^T=\overline{p}\}$. For $h$ and $p$ in this set, $\nu_\bF$ is real-analytic on $\R$.
\end{lemma}

It is obvious that there are matrices $h$ and $p$ for which $\bF-\mu\bN$ has $0$ as eigenvalue for some $\mu\in\R$. The simplest examples are HF Hamiltonians for which $p\equiv0$ and $\bF-\mu\bN$ is not invertible each time $\mu$ equals an eigenvalue of $h$. If $p$ does not vanish but commutes with $h$, then we have $|h+ip-\mu|^2=|h-ip-\mu|^2=(h-\mu)^2+p^2$ and we see that $0$ is never in the spectrum of $\bF-\mu\bN$ when the kernel of $p$ does not contain the eigenvectors of $h$. However there are counterexamples with $p$ invertible not commuting with $h$. For instance, $\bF+\bN$ is not invertible for 
$$h=\begin{pmatrix}
-1 & 0\\ 0 & 2
\end{pmatrix},\qquad 
p=\begin{pmatrix}
0 & 2\\ 2 & 0
\end{pmatrix}.$$
We now turn to the proof of Lemma~\ref{lem:generic_nu}.

\begin{proof}
The operator $\bF-\mu\bN$ is unitarily equivalent to
$$
\begin{pmatrix}
\frac{1}{\sqrt{2}} & \frac{i}{\sqrt{2}}\\ \frac{i}{\sqrt{2}} & \frac{1}{\sqrt{2}}
\end{pmatrix}
\big(\bF-\mu\bN\big) 
\begin{pmatrix}
\frac{1}{\sqrt{2}} & -\frac{i}{\sqrt{2}}\\ -\frac{i}{\sqrt{2}} & \frac{1}{\sqrt{2}}
\end{pmatrix}
=-i
\begin{pmatrix}
0 & h+ip-\mu\\
-(h-ip-\mu) & 0
\end{pmatrix}.
$$
From this we deduce that $\bF-\mu\bN$ is invertible if and only if $h+ip-\mu$ and $h-ip-\mu$ are invertible. 
Then we have
$$\big(\bF-\mu\bN\big)^{-1}
=i
\begin{pmatrix}
\frac{1}{\sqrt{2}} & -\frac{i}{\sqrt{2}}\\ -\frac{i}{\sqrt{2}} & \frac{1}{\sqrt{2}}
\end{pmatrix}
\begin{pmatrix}
0 & -(h-ip-\mu)^{-1}\\
(h+ip-\mu)^{-1} & 0
\end{pmatrix}
\begin{pmatrix}
\frac{1}{\sqrt{2}} & \frac{i}{\sqrt{2}}\\ \frac{i}{\sqrt{2}} & \frac{1}{\sqrt{2}}
\end{pmatrix}
$$
and
$$\norm{\big(\bF-\mu\bN\big)^{-1}}=\max\Big(\norm{(h+ip-\mu)^{-1}}\;,\; \norm{(h-ip-\mu)^{-1}}\Big).$$
The statement now follows from the fact that, on a dense open set, the spectra of $h\pm ip$ do not intersect the real axis.
\end{proof}

Lemma~\ref{lem:generic_nu} is interesting when we apply the Roothaan or the ODA, because it means that, as soon as $p\neq0$, most often we will have no choice for $\mu_{n+1}$ and we will take $\delta_{n+1}=0$. Saying differently, it is really reasonable to assume that the sequence generated by the Roothaan and the ODA are uniformly well posed (of course when the final state is believed to have a non vanishing pairing), as we did in Theorem~\ref{thm:Roothaan} and~\ref{thm:ODA}.

Even if it is in general smooth, the function $\nu$ can still vary quickly and this will be the case when the pairing term $p$ is small. The appropriate method to find the solution of $\nu_\bF(\mu)=N/2$ then depends on the properties of $\nu_\bF$. If the Hamiltonian $\bF$ has a large enough pairing matrix $p$, then $\nu_\bF$ is smooth and we can use a Newton-like method to solve the equation $\nu_\bF(\mu)=N/2$. A trial chemical potential $\mu^0$ being given, we compute the derivative $\partial\nu_\bF/\partial\mu(\mu^0)$ using Formula~\eqref{eq:derivee_nu} and then let
$$\mu^1:=\mu^0+\big(N/2-\nu_\bF(\mu^0)\big)\left(\frac{\partial\nu_\bF}{\partial\mu}(\mu^0)\right)^{-1}.$$
The method can be iterated until convergence of $\mu^n$ towards the desired $\mu$. The convergence is very fast, as soon as $\nu_\bF$ is smooth.

If the Hamiltonian $\bF$ has a small pairing matrix $p$, the function $\nu_\bF$ will be smooth but close to a step function. Its derivative varies very quickly and the previous Newton method is not appropriate. In this case we can use a simple bisection method. The bounds on $\nu_\bF(\mu)$ for large $|\mu|$ can be used to find a good starting interval $[\mu_l,\mu_r]$ such that $\nu_\bF(\mu_l)<N/2$ and $\nu_\bF(\mu_r)>N/2$.

We have to find a new $\mu_{n+1}$ at each step of the Roothaan or ODA. It is of course not efficient to find $\mu_{n+1}$ with a very high precision all along the algorithm. Dechargé and Gogny advice in Section II.E of ~\cite{DecGog-80} to apply the Newton scheme only once at each step. This means that 
$$\mu_{n+1}=\mu_n+\big(N/2-\nu_n(\mu_n)\big)\left(\frac{\partial\nu_n}{\partial\mu}(\mu_n)\right)^{-1}$$
where $\nu_n$ is the function $\nu$ corresponding to $\bF_\Upsilon=\bF_{\Upsilon_n}$.
This is then the same as doing perturbation theory on first order. We use a slightly different strategy which we explain in the next section.

\section{Numerical results}\label{sec:num}

In this section, we present some numerical results for two very simple interactions $W$. We start by considering in Section~\ref{sec:gravitation} a purely 3-dimensional gravitational model in which 
$$W(x)=-\frac{g}{|x|},\qquad g>0.$$
Then we consider in Section~\ref{sec:Coulomb-perturbed} a (repulsive) Coulomb potential which is perturbed at intermediate distances by an attractive effective potential, as is usually employed in nuclear physics:
$$W(x)=\frac{\kappa}{|x|}-a_1\,\text{exp}\left(-b_1{|x|^2}\right)+a_2\,\text{exp}\left(-b_2|x|^2\right),\qquad \kappa>0.$$
In the next section we quickly explain our numerical technique to treat these two simple systems.

\subsection{Method}

To simulate our physical systems, we have used the open source software Scilab~\cite{Scilab-software}. Our potential $W$ is always real, radial and spin-independent. To reduce the numerical cost we have therefore always imposed the spin, time-reversal and spherical symmetry. This means that we have to cope with $\ell_{\rm max}+1$ real and symmetric $N_b\times N_b$ matrices $G^\ell$ and $A^\ell$ (the one-particle density matrix and the pairing density matrix in the $\ell$th angular momentum sector). The total energy of the system is given by Equation~\eqref{eq:discretized_energy_nospin_radial} and we have to impose the constraints~\eqref{eq:discretized_constraint_ell} and~\eqref{eq:contrainte_N_ell} which we recall here for convenience:
\begin{equation}
0\leq\Upsilon^\ell\bS\Upsilon^\ell\leq\Upsilon^\ell,\quad\text{with}\quad  
\Upsilon^\ell:=\begin{pmatrix}
G^\ell&A^\ell\\ A^\ell&\ S^{-1}-G^\ell\end{pmatrix}\quad\text{and}\quad
\bS=\begin{pmatrix}
S&0\\ 0&S\end{pmatrix},
\label{eq:discretized_constraint_ell_bis}
\end{equation}
\begin{equation}
\sum_{\ell=0}^\lmax(2\ell+1)\tr(SG^\ell)=N/2.
\label{eq:contrainte_N_ell_bis} 
\end{equation}

We choose a simple basis set $(\chi_1,...,\chi_{N_b})$ of $L^2([0,\ii),r^2dr)$, made of ``hat functions'' associated with a chosen grid
$$0=r_0<r_1<\cdots <r_{N_b}<r_{N_b+1}:=r_{\rm max}.$$
We impose Dirichlet boundary conditions at $r_{\rm max}$. We have tested different types of grids and there was no important difference between them. The results presented here are all with regular grids. As we will explain later, for a given basis size $N_b$, the results usually depend a lot on the value of the radius $r_{\rm max}$ of the ball in which the system is placed.

Our main goal is to investigate the existence of pairing. We therefore always start by doing a precise Hartree-Fock calculation, for which we use the Optimal Damping Algorithm described in Section~\ref{sec:ODA}. We take as initial state a simple uniform state 
\begin{equation}
G_{\rm init}=\frac{N}{2\tr(S)}{\rm Id}_{N_b}
\label{eq:initial_state} 
\end{equation}
and we run HF until convergence. We have observed a global stability of the results with respect to initial states, hence the previous simple choice is appropriate (but more clever choices might decrease the total number of iterations).
Then, we use the converged HF state $G_{\rm opt}$ as initial datum for the HFB algorithm. Of course we have to perturb it a little bit since any HF solution is also an HFB solution. We proceed as follows. Assuming that the overlap matrix $S={\rm Id}_{N_b}$ and that $\ell_{\rm max}=0$ for simplicity, the optimal HF state $G_{\rm opt}$ can be written in the form
$$G_{\rm opt}=\sum_{k=1}^{N/2}v_kv_k^T,$$
where $v_k$ are the $N/2$ first eigenvectors of the mean-field matrix $h$,
$$hv_k=\epsilon_k\,v_k.$$
We then choose a number $n_v$ of valence orbitals and a mixing parameter $\theta$, and we perturb $G_{\rm opt}$ as follows
$$G_{\rm init}'=\sum_{k=1}^{N/2-n_v}v_kv_k^T+\theta\sum_{k=N/2-n_v+1}^{N/2}v_kv_k^T+(1-\theta)\sum_{k=N/2+1}^{N/2+n_v}v_kv_k^T,$$
$$A_{\rm init}'=\sqrt{\theta(1-\theta)}\sum_{k=N/2-n_v+1}^{N/2+n_v}v_kv_k^T.$$
In most cases, we have observed that $n_v=1$ and $\theta=0.95$ works perfectly well, that is, the algorithm escapes from the HF solution $G_{\rm opt}$ and converges towards an optimal HFB state. But other values of $n_v$ and $\theta$ seem to work fine also.

When the maximum angular momentum $\ell_{\rm max}$ is larger than 0, we often first run the algorithm with $\ell_{\rm max}=0$ for a few iterations before switching to the actual value of $\ell_{\rm max}$. We stop the algorithm when the commutators $[F_n^\ell,\Upsilon_{n}^\ell]$ are smaller than a prescribed error. 
We know from~\eqref{eq:generalized_form_Gamma_ell} and~\eqref{eq:generalized_eig_pb_ell} that these commutators must all vanish for an exact solution of the discretized HFB minimization problem.
In terms of the matrix $\bS$, the right quantity to look at is
$$\sum_{\ell=0}^{\ell_{\rm max}}\norm{\bS^{-\frac12}\big(\bF_n^\ell\Upsilon_n^\ell \bS - \bS\Upsilon_n^\ell \bF_n^\ell\big)\bS^{-\frac12}}$$
where $\norm{\cdot}$ is the usual operator norm for $(2N_b)\times(2N_b)$ matrices. There is a similar formula in the HF case~\cite{CanDefKutLeBMad-03}.

As we have explained, in the HFB case, ensuring the constraint~\eqref{eq:contrainte_N_ell_bis} is not as easy as in the HF case. In the beginning of the algorithm, our state $\Upsilon$ is rather close to an HF state by construction. Therefore, the function $\nu_\bF(\mu)$ defined in Section~\ref{sec:constraint} is close to a step function. We choose an error $\varepsilon$ and look for the next states $\Upsilon^\ell_{n+1}$ having a total number of particles $\sum_{\ell=0}^{\ell_{\rm max}}(2\ell+1)\tr(S G^\ell_{n+1})$ close to $N/2$, within the error $\varepsilon$, using a simple bisection method. We use the bisection for a fixed number of global iterations. Then, when the pairing term is large enough, we switch to a Newton method in order to find the state $\Upsilon_{n+1}$. We have observed that even if in the beginning several Newton iterations can be employed at each step, usually only one Newton iteration is necessary after a while. To guarantee a good value of the average number of particles in the end, we decrease the error $\varepsilon$ on $|\sum_{\ell=0}^{\ell_{\rm max}}(2\ell+1)\tr(S G^\ell_{n+1})-N/2|$ along the algorithm.

\subsection{Pure Newtonian interaction}\label{sec:gravitation}

\subsubsection{Model}

Here we consider a system of $N$ spin-$1/2$ neutral particles, only interacting through the Newtonian interaction
\begin{equation}
W(x)=-\frac{g}{|x|},\qquad g>0.
\label{eq:Newton}
\end{equation}
This potential is strongly attractive at short distances. Since $1/|x|$ does not decay too fast at infinity, it is also quite attractive at large distances. The kinetic energy does not scale the same as the potential energy. By a simple scaling argument, we can therefore always assume that 
$$g\equiv1.$$

This model can be used to describe neutron stars and white dwarfs when $N\gg1$. It has been particularly studied from a theoretical point of view in the pseudo-relativistic case where the kinetic energy is given by $T=\sqrt{c^4m^2-c^2\Delta}-mc^2$, see~\cite{LieThi-84,LieYau-87,LenLew-10}. In our simulations we restrict ourselves to the non-relativistic case of the Laplacian $T=-\Delta/(2m)$ (in units such that $m=1/2$). It would be interesting to take $N$ large but this is of course much too difficult from a numerical point of view.

As mentioned before, we always impose the spin and time-reversal symmetries, which is perfectly justified for the ground state since the interaction~\eqref{eq:Newton} satisfies the assumption of the Bach-Fröhlich-Jonsson Theorem~\ref{thm:variational_constrained}. We also impose spherical symmetry which, on the contrary, is not known to hold for the true ground state.

One advantage of the Newtonian interaction~\eqref{eq:Newton} is that the operators $J$ and $K^{\ell\ell'}$ can be explicitely computed in the basis of hat functions. We have shown in Section~\ref{sec:spherical-symm} that the energy can be expressed in terms of
$$(ij|mn)_{\ell,\ell'}= \int_0^\ii r^2\,dr\int_0^\ii s^2\,ds\;\chi_i(r)\,\chi_j(r)\,\chi_m(s)\,\chi_n(s)\, w_{\ell,\ell'}(r,s)$$
where
$$w_{\ell,\ell'}(r,s)=\frac12\int_{-1}^1W\left(\sqrt{r^2+s^2-2rst}\right)\, P_\ell(t)\,P_{\ell'}(t)\,dt=-\frac12\int_{-1}^1\frac{P_\ell(t)\,P_{\ell'}(t)}{\sqrt{r^2+s^2-2rst}}\,dt.$$
Using the well-known formula
$$\frac{1}{\sqrt{r^2+s^2-2rst}}=\sum_{n=0}^\ii \frac{\min(r,s)^n}{\max(r,s)^{n+1}}P_n(t)$$
we deduce that
$$w_{\ell,\ell'}=-\frac12\sum_{n=0}^\ii\left(\int_{-1}^1P_nP_\ell P_{\ell'}\right)\frac{\min(r,s)^n}{\max(r,s)^{n+1}}.$$
The integral over the Legendre polynomials is related to the usual Clebsch-Gordan coefficients as follows
$$\frac12\int_{-1}^1P_n(t)\,P_\ell(t)\, P_{\ell'}(t)\,dt=\begin{pmatrix}
\;\ell\;&\;\ell'\;&\;n\;\\ 0 & 0 & 0
\end{pmatrix}^2$$
and only a finite number of terms are non zero in the sum over $n$.
The final result can be expressed as
\begin{multline}
(ij|mn)_{\ell,\ell'}\\=-\sum_{n=0}^\ii\begin{pmatrix}
\;\ell\;&\;\ell'\;&\;n\;\\ 0 & 0 & 0
\end{pmatrix}^2\int_0^\ii r^2\,dr\int_0^\ii s^2\,ds\frac{\min(r,s)^n}{\max(r,s)^{n+1}}\;\chi_i(r)\,\chi_j(r)\,\chi_m(s)\,\chi_n(s).
\label{eq:formula_ijmn_Newton} 
\end{multline}
These integrals can be explicitely computed for hat functions and $0\leq \ell,\ell'\leq \ell_{\rm max}$ with $\ell_{\rm max}$ not too large. In our numerical experiments we have put the explicit formulas in Scilab for $\ell_{\rm max}=1$. The integrals were stored in memory during the whole calculation.

\subsubsection{Roothaan vs ODA}

In the HF case, we have observed that the Roothaan algorithm very often oscillates between two states, none of them being the solution of the problem (as described in Theorem~\ref{thm:Roothaan} and in~\cite{CanBri-00a}). The Roothaan algorithm seems more well behaved in the HFB case. With the model presented in this section, we never got real oscillations for HFB. Sometimes the convergence is improved by using the ODA, but in most cases the Roothaan algorithm always converges towards the same state as the ODA in the end. As we will see later, the situation is very different for the model studied in Section~\ref{sec:Coulomb-perturbed}, which is inspired of nuclear physics.

We start by comparing Roothaan and ODA in the HF case. There, oscillations seem to be related to the size of the gap between the largest filled eigenvalue and the smallest unfilled one. Indeed, oscillations in HF seem to only occur when there is pairing in HFB, an effect which is also well-known to be related to the size of the gap (see, e.g., Theorem 5 in~\cite{BacFroJon-09}). When there is no pairing, the HF Roothaan algorithm always behaves like the ODA. However, the situation is complex and there is no exact rule. Sometimes the Roothaan algorithm does \emph{not} oscillate even when the gap is rather small and there is pairing.

In Figure~\ref{fig:oscillations} we display the value of the energy obtained along the algorithm for the Roothaan and the ODA, for the following choice of parameters: $N=6$, $N_b=200$, $\ell_{\rm max}=0$ and $r_{\rm max}=30$. The ODA converges in about 17 iterations, whereas the Roothaan algorithm oscillates. We also show the value of the norms $\norm{G_{n}-G_{n-1}}$ and $\norm{G_{n}-G_{n-2}}$ along the Roothaan algorithm. The oscillation between two points is clearly demonstrated.

When we decrease the parameter $r_{\rm max}$ but keep $N_b=200$ constant, the gap is seen to increase slightly and the Roothaan algorithm behaves better. In Table~\ref{tab:gap}, we give the numerical value of the last filled eigenvalue and the corresponding gap. The Roothaan algorithm slowly converges for $r_{\rm max}=25$ and it coincides with the ODA when $r_{\rm max}=20$. The gap for $r_{\rm max}=20$ is $2.5$ times the one for  $r_{\rm max}=30$.
We will discuss the occurence of pairing in terms of the parameter $r_{\rm max}$ in the next section.
\begin{table}[h!]
\centering
\begin{tabular}{|c|c|c|c|}
\hline
$r_{\rm max}$ & $\epsilon_{N/2}$ & $\epsilon_{N/2+1}-\epsilon_{N/2}$ & behavior of HF Roothaan\\
\hline
20 & -0.532430& 0.159430 & fast convergence\\
\hline
25 & -0.536706 & 0.081016 & slow convergence\\
\hline
30 & -0.529200& 0.061928 & oscillations\\
 & -0.548554& 0.067422 & \\
\hline
\end{tabular}
\caption{Value of the last filled eigenvalue $\epsilon_{N/2}$ and the corresponding gap $\epsilon_{N/2+1}-\epsilon_{N/2}$ in HF, for $N=6$, $N_b=200$ and $\ell_{\rm max}=0$. For $r_{\rm max}=30$ the Roothaan algorithm oscillates and we display the last filled eigenvalue and the gap for the two states.\label{tab:gap}}
\end{table}

As we have mentioned the Roothaan algorithm is usually much more well behaved in the HFB case. However, sometimes the convergence can be improved dramatically by using the ODA. In Figure~\ref{fig:roothaan_slow_HFB} we display the energy along the iterations of the algorithm in both the Roothaan and ODA cases, for $N_b=500$, $N=16$, $\ell_{\rm max}=1$ and $r_{\rm max}=10$. In this case the Roothaan algorithm is very badly behaved. It passes very close to the HF ground state and it takes it a very long time to escape from it. On the other hand, the ODA does not suffer from this problem and it converges much more rapidly.

\begin{figure}[h!]
\centering

\bigskip

\includegraphics[width=7.1cm]{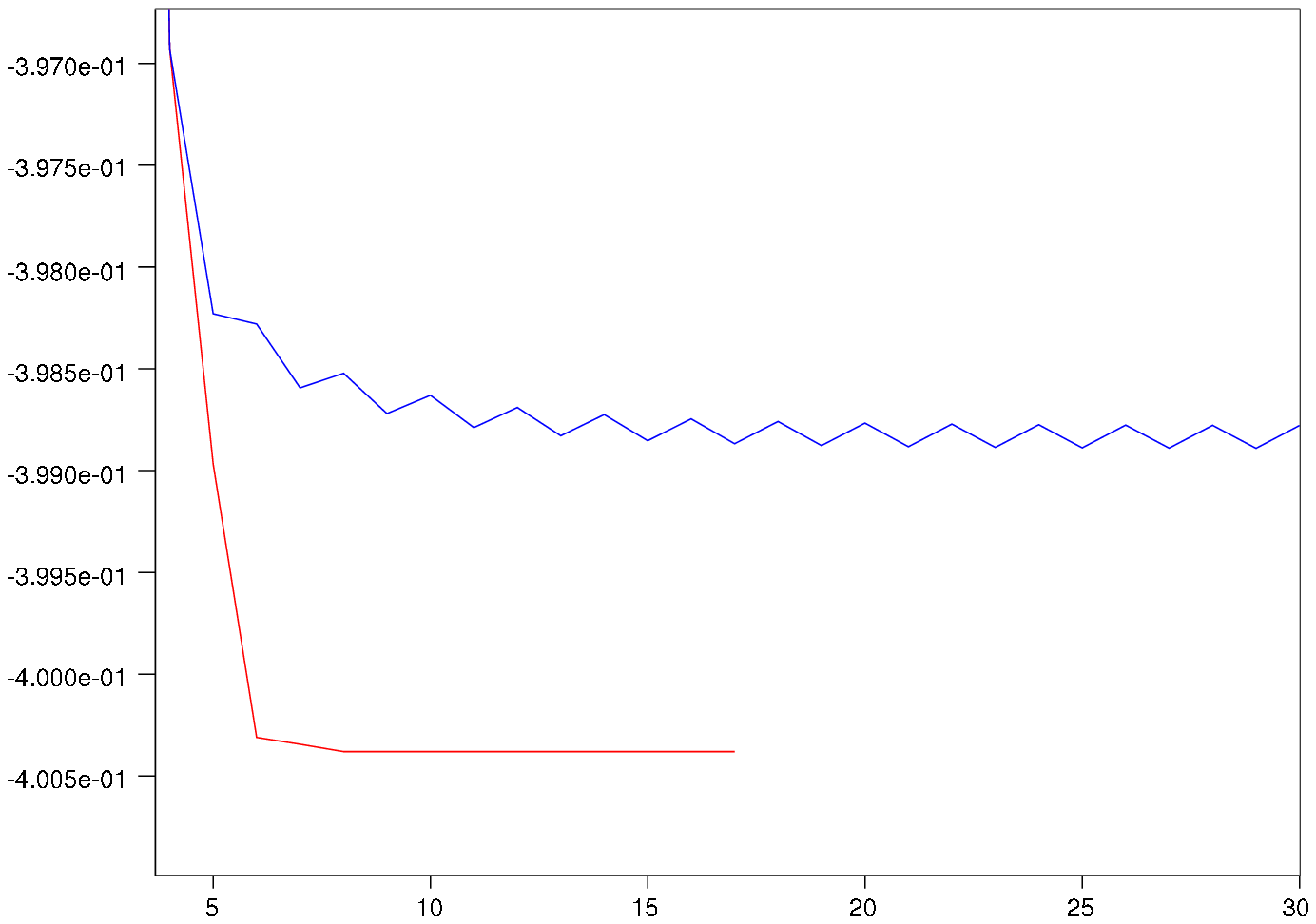} \includegraphics[width=6.6cm]{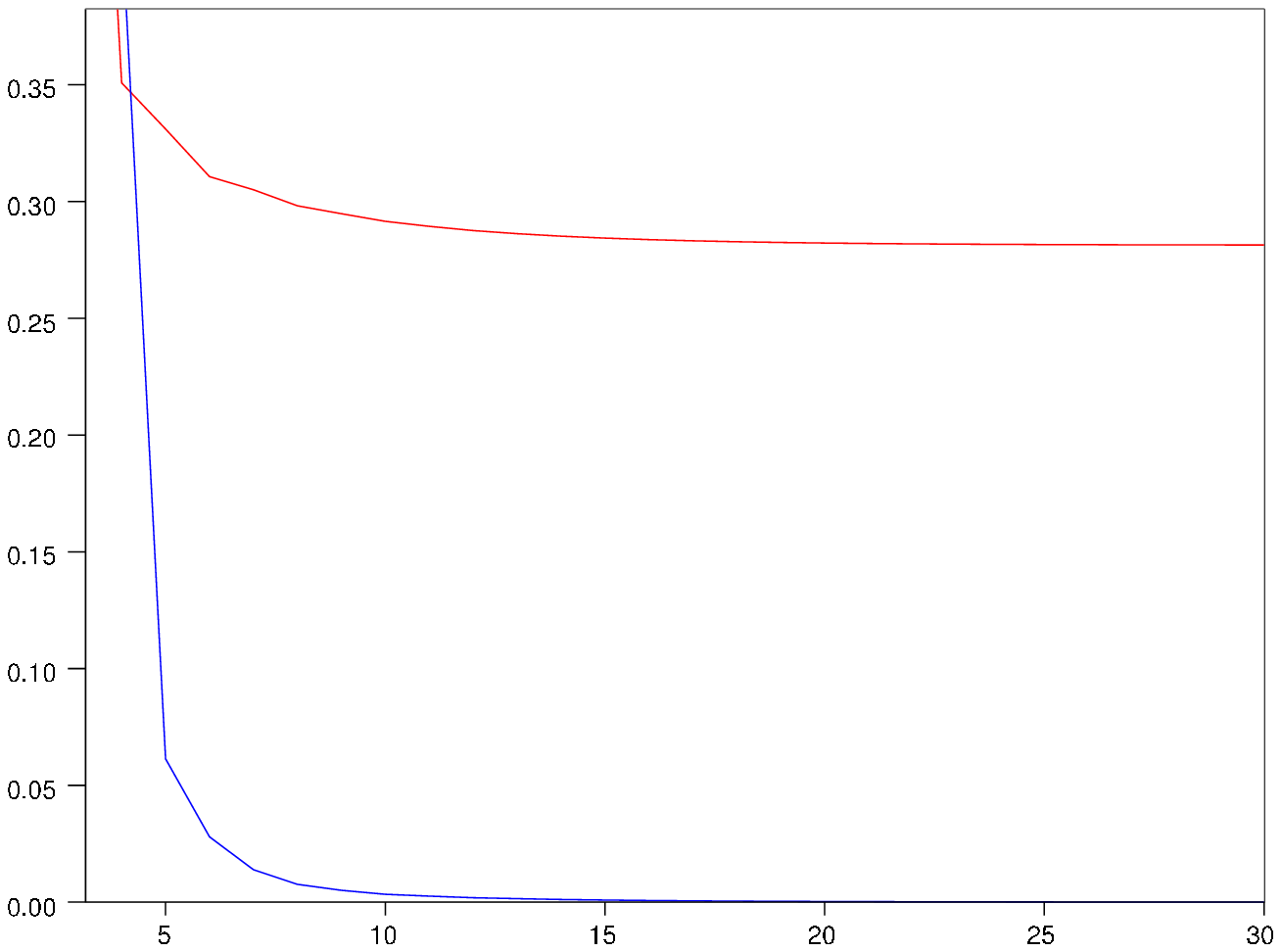}
\caption{\textit{Left:} HF energy along the iterations for the Roothaan Algorithm (blue) and the ODA (red). \textit{Right:} Values of $\norm{G_{n}-G_{n-1}}$ (red) and $\norm{G_{n}-G_{n-2}}$ (blue) along the Roothaan algorithm, showing the oscillations between two states. Here $N=6$, $N_b=200$, $\ell_{\rm max}=0$ and $r_{\rm max}=30$.\label{fig:oscillations}}

\bigskip\bigskip

\includegraphics[width=14cm]{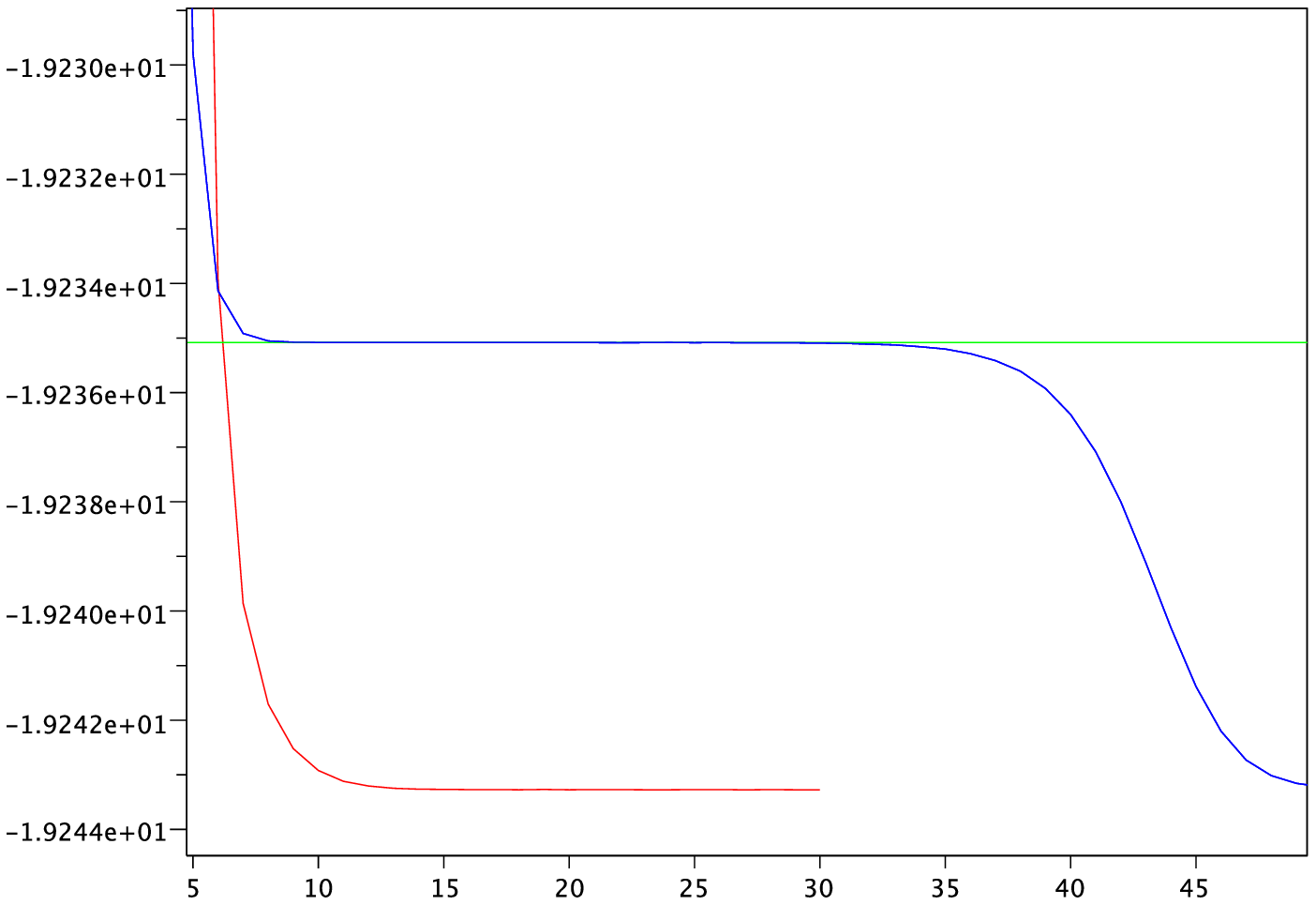}
\caption{HFB energy along the iterations of the Roothaan (blue) and the ODA (red) for $N_b=500$, $N=16$, $\ell_{\rm max}=1$ and $r_{\rm max}=10$. The optimal HF energy is also displayed (green).\label{fig:roothaan_slow_HFB}}
\end{figure}

\subsubsection{Numerical evidence of pairing}

Pairing effects in a finite discretization basis might depend on the properties of the basis. As we have explained, the occurence of pairing is related to the size of the gap in HF theory and this gap varies with the radius $r_{\rm max}$ in which the system is confined. If $r_{\rm max}$ is decreased the system is more condensed and the HF gap increases.

In Figure~\ref{fig:energies_rmax} we display the HF and HFB ground state energies computed for $N=16$ in a basis set of size $N_b=200$, in terms of $r_{\rm max}$. The HF and HFB curves are distinct for $r_{\rm max}$ large enough and they merge at $r_{\rm max}=6$ approximately. This observation is confirmed by the value of the norm of $A$ plotted on the right of the same figure. The minima of the HF and HFB ground state energies are attained at about $r_{\rm max}\simeq 10$ in the HF case and $r_{\rm max}\simeq 10.5$ in the HFB case, which is sufficiently far from the merging point. The minima of these curves correspond to the best possible approximation for a given basis size $N_b$ (here $N_b=200$) and a given type of grid (here regular). The difference between the corresponding HF and HFB energies is significant. The HF ground state energy at $r_{\rm max}=10$ is $-19.232176$ (in our units in which $m=1/2$ and $e=1$), whereas the HFB ground state energy at $r_{\rm max}=10.5$ is $-19.240176$. The norm of the pairing matrix $A$ is rather large at this point: 
$$\norm{A}=\sqrt{\tr(SA_0SA_0)+3\tr(SA_1SA_1)}\simeq0.462129.$$ 
This goes in favour of the conclusion that pairing really occurs for $N=16$ in this model. 
This intuition is confirmed by a more precise calculations with $N_b=500$ which we discuss below.

The observation of pairing requires to have an appropriate $r_{\rm max}$ but it does \emph{not} require to have a very large basis set. Even for $N_b=30$ and $r_{\rm max}=10.5$, we already find that the HFB energy is approximately $-19.078416$ whereas the HF energy is about $-19.072954$. The corresponding norm of the pairing matrix $A$ is $\norm{A}\simeq0.424124$.

Pairing is a subtle effect which decreases the energy by a small amount (much less than one percent here). Catching this effect requires to be very careful when choosing the radius $r_{\rm max}$. Taking $r_{\rm max}$ too small might lead to the conclusion that there is no pairing. In our simulations we have always observed the occurence of pairing, but provided we choose $r_{\rm max}$ appropriately. The values of $r_{\rm max}$ at which the HF and HFB energies attain their minimum were always found on the right of the merging point of the two curves. In Table~\ref{tab:results} below we give our results for $N_b=200$ and $N=6,10,16$ and $20$. The HFB ground state energy is always smaller than the HF energy.

In the paper~\cite{LenLew-10}, Lenzmann and Lewin have rigorously studied the gravitational model of this section. They showed the existence of a ground state in both the HF and HFB cases. But, so far, no proof that pairing occurs has been provided. The numerical results of this section tend to show that there is actually always pairing, at least for $N$ not too large.

\begin{figure}[ph!]
\centering

\bigskip

\includegraphics[width=12.5cm]{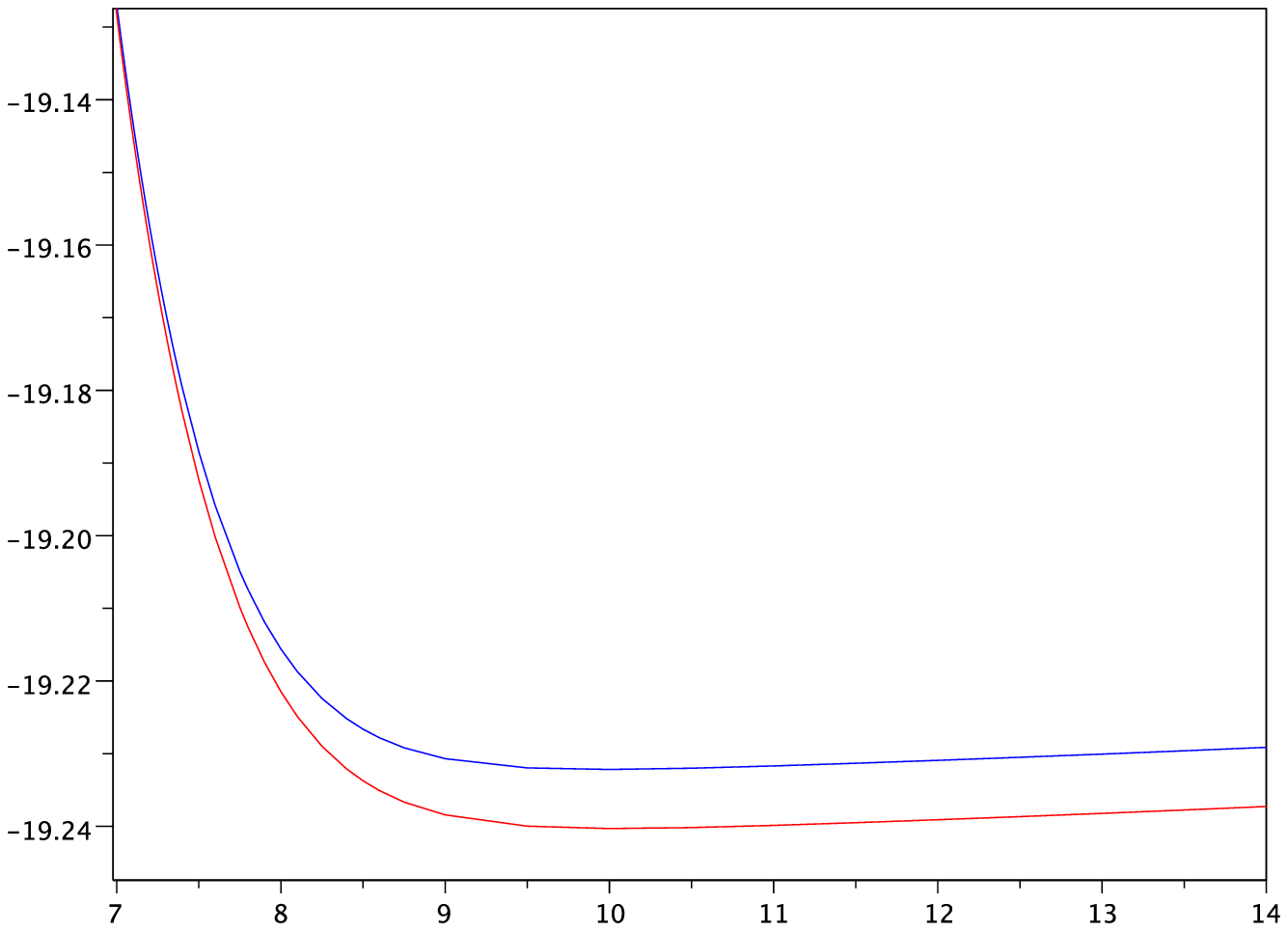}


\includegraphics[width=12.5cm]{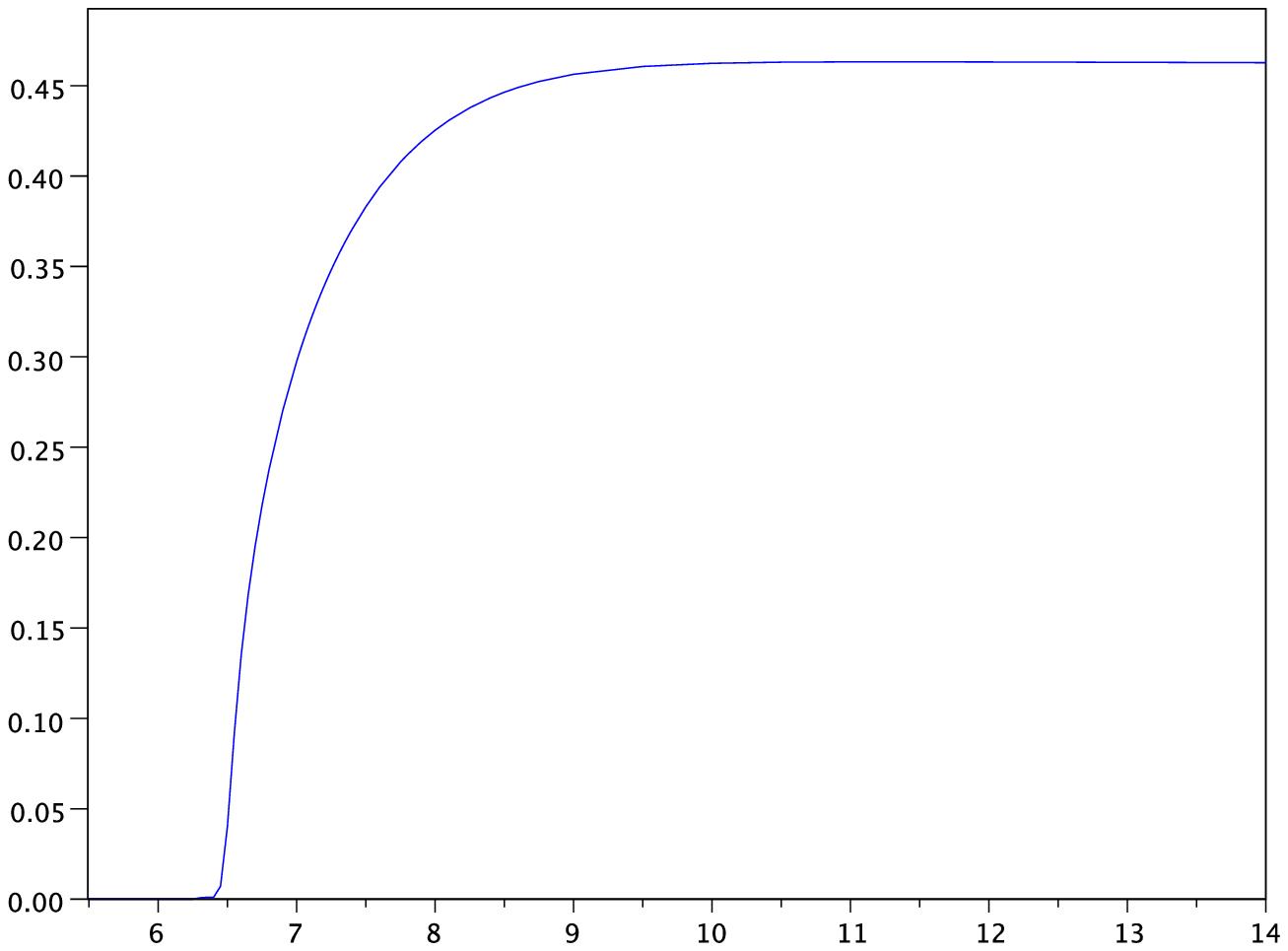}
\caption{Value of the ground state HF and HFB energies (top) and of the norm of the pairing matrix $A$ (bottom), as functions of the radius $r_{\rm max}$ in which the system is confined, for $N=16$, $N_b=200$ and $\ell_{\rm max}=1$.\label{fig:energies_rmax}}
\end{figure}

\subsubsection{Properties of the HFB ground state}

In Table~\ref{tab:results} below we give our results for $N_b=200$ and $N=6,10,16$ and $20$, for the optimal values of $r_{\rm max}$. With $\ell_{\rm max}=1$ we have observed that the shells are filled alternatively. In HF theory, the cases $N=10$ and $N=16$ correspond to closed shells, whereas for $N=6$ and $N=20$ the last shell is only partially filled. This is a simple explanation for the fact that the pairing matrix is much bigger in these cases.

In Table~\ref{tab:occ_nbs} we display the occupation numbers for the optimal HFB ground state in the closed shell case $N=16$ and in the open shell case $N=20$. 
Because of the spin, these are the eigenvalues of $G_0$ multiplied by 2 and that of $G_1$ multiplied by 6. Even in the closed shell case $N=16$, a rather important pairing effect is observed between the last filled orbital (the second $\ell=1$ eigenvalue) and the first unfilled one (the third $\ell=0$ eigenvalue). This results in a decrease of the last occupation number of the HF one-particle density matrix by approximately $0.228$.

\begin{table}[h!]
\centering
\begin{tabular}{|l|l|l|l|l|l|}
\hline
$N$ & $r_{\rm max}$ & HF gap & HF energy  & HFB energy & $\norm{A}$\\
\hline
6 & 15 & 0 & -1.7327688  & -1.9934252 & 1.0242134\\
10 & 11 & 1.023642 & -6.7911634  & -6.8148576 & 0.5871951\\
16 & 10 & 1.404396 & -19.232177 & -19.2403096 & 0.4623593\\
20 & 9 & 0 & -30.010574  & -30.174576 & 0.8235512\\
\hline
\end{tabular}

\caption{Results for $N_b=200$ and $\ell_{\rm max}=1$.\label{tab:results}}

\bigskip

\begin{tabular}{|c|c|}
\hline
\multicolumn{2}{|c|}{$N=16$}\\
\hline
$\ell=0$ & $\ell=1$\\
\hline
1.9999318 & 5.9997504\\
1.9980654 & 5.7710970\\
0.2281366 & 0.0026448\\
0.0002694 & 0.0002720\\
0.0000120 & 0.0000084\\
0.0000014 & 0.0000012\\
0.0000002 & 3.316D-07\\
6.896D-08 & 1.005D-07\\
$\vdots$ & $\vdots$\\
\hline
\end{tabular}
\qquad\qquad
\begin{tabular}{|c|c|}
\hline
\multicolumn{2}{|c|}{$N=20$}\\
\hline
$\ell=0$ & $\ell=1$\\
\hline
1.9999832 & 5.9999490\\
1.9997458 & 5.9983572\\
1.9875134 & 2.0084946\\
0.0045222 & 0.0012960\\
0.0000690 & 0.0000552\\
0.0000058 & 0.0000066\\
0.0000010 & 0.0000002\\
0.0000002 & 3.072D-07\\
$\vdots$ & $\vdots$\\
\hline
\end{tabular}

\caption{Occupation numbers of the HFB minimizer, for $N_b=200$ and $\ell_{\rm max}=1$.\label{tab:occ_nbs}}
\end{table}

\subsubsection{Quality of the approximation in terms of the number $N_b$ of points}

In Table~\ref{tab:N_b} we display the HF and HFB ground state energies for $N=16$, $\ell_{\rm max}=1$ for the the optimal value of $r_{\rm max}$, in terms of the number of discretization points $N_b$ of the regular grid. The convergence is not very fast, but we see that the difference between the HF and the HFB energy, as well as the norm of the pairing matrix are of the same order for small $N_b$ as they are for larger $N_b$'s. From this observation we can conclude that it is probably not necessary to take $N_b$ very large in order to decide whether pairing occurs or not.

\begin{table}[h!]
\centering
\begin{tabular}{|l|l|l|l|l|l|}
\hline
$N_b$ & $r_{\rm max}$ & HF energy  & HFB energy  & difference & $\norm{A}$\\
\hline
30 & 9 & -19.112314  &  -19.117948 & 0.005634 &  0.425604\\
50 & 9 & -19.189066 & -19.196012 & 0.006946& 0.445728\\
100 & 9 & -19.222300  & -19.229872 &0.007572& 0.454173\\
150 & 10 & -19.229494 & -19.237574 &0.008080& 0.461725\\
200 & 10 &  -19.232176 & -19.240308 &0.008132&  0.462363\\
250 & 10 & -19.233420 & -19.241576  &0.008156& 0.462659\\
300 & 10 &  -19.234094 &  -19.242264  &0.008170& 0.462821\\
400 & 11 & -19.234826 & -19.243068 &0.008242& 0.463905\\
500 & 11 &  -19.235206& -19.243456& 0.008250&0.463985\\
\hline
\end{tabular}

\bigskip

\caption{Value of the HF and HFB energies for $N=16$ and $\ell_{\rm max}=1$ and the (approximate) optimal $r_{\rm max}$.\label{tab:N_b}}
\end{table}

\subsection{A simplified model for protons and neutrons}\label{sec:Coulomb-perturbed}

In this section we report on our numerical results concerning a simple model inspired of nuclear physics. The interaction between protons and neutrons is not a fundamental law of nature because these are composite particles made of quarks, which interact through weak, strong and electrostatic forces. A common procedure used in nuclear physics is to use \emph{empirical forces}~\cite{RinSch-80} which involve a small number of parameters which are fitted to experiment or to the known behavior of the model in some limits. The most common forces used in practice are the so-called Skyrme~\cite{Skyrme-59} and Gogny~\cite{Gogny-73,Gogny-75,DecGog-80} forces and they depend nonlinearly on the state itself. Here we consider an effective force which is fixed and does not depend on the quantum state. We also take it spin-independent and isospin-independent. Our goal is to test some simple ideas and not to do a real nuclear physics calculation.

\subsubsection{Model}

The nucleon-nucleon potential has been observed to be repulsive at short distances and only attractive at medium distances. It decays very fast at infinity. A simple choice to describe this is to take
\begin{equation}
W(x)=\frac{\kappa}{|x|}-a_1\,e^{-b_1|x|^2}+a_2\,e^{-b_2|x|^2},
\label{eq:form_V_nuclear}
\end{equation}
with $a_2,a_1>0$, $b_1<b_2$. The constant $\kappa$ is 1 for the proton-proton interaction and $0$ for the proton-neutron and the neutron-neutron interaction. 
The other constants usually also depend on the isospin (the quantum variable which determines whether a nucleon is a neutron or a proton). For simplicity we work here with particles having a definite isospin. This means that we assume to have either only protons or only neutrons. In particular we want to ask for which strength of the effective force it becomes possible for the protons to overcome their Coulomb repulsion and form a bound state. In reality a nucleus is made of a certain number of protons and neutrons and one has to use a different HFB state for each species.

In our applications we have chosen for simplicity $b_1=1$, $b_2=4$, $a_1=a=2\,a_2/3$. This means that the effective force takes the form
\begin{equation}
W(x)=\frac{\kappa}{|x|}+a\left(\frac{3}{2}\,e^{-4|x|^2}-e^{-|x|^2}\right).
\label{eq:form_V_nuclear_final}
\end{equation}
When $\kappa=1$, this force is purely repulsive for $a\leq2.87$ and it becomes attractive at intermediate distances for larger $a$'s. The corresponding force is displayed in Figure~\ref{fig:Gogny_force} for $a=1$ and $\kappa=0$. 

\begin{figure}[h!]
\centering
\includegraphics[width=9cm]{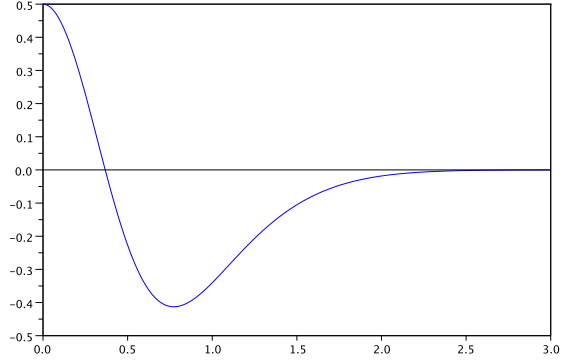}
\caption{The effective force $-\,e^{-|x|^2}+3\,e^{-4|x|^2}/2$ used in our calculation of Section~\ref{sec:Coulomb-perturbed}. The (repulsive) Coulomb potential must be added for protons. \label{fig:Gogny_force}}
\end{figure}

One can ask several questions concerning the model considered in this section:
\begin{enumerate}
\item For which value of $a$ does a system of $N$ identical nucleons bind in Hartree-Fock theory?
\item Is there always pairing when there is binding?
\item Can pairing effects allow for binding with a smaller $a$ than in HF theory?
\end{enumerate}
These questions are mostly of academic nature for the very simplified model considered in this section. But investigating the same problems with more realistic forces is  very important from a physical point of view. From a mathematical point of view, nothing seems to be known for simple models of the same form as in this section. It is not even known that binding always occurs for $a$ large enough with the previous interaction. We hope that our calculations will stimulate some further mathematical studies.

\subsubsection{Some computational details}

We always minimize over states having the spin, time-reversal and rotation symmetries. The Bach-Fröhlich-Jonsson Theorem~\ref{thm:variational_constrained} does not apply to the model of this section, hence we are making a further approximation here.

For such symmetric states we have shown in Section~\ref{sec:spherical-symm} that the energy can be expressed in terms of
$$(ij|mn)_{\ell,\ell'}= \int_0^\ii r^2\,dr\int_0^\ii s^2\,ds\;\chi_i(r)\,\chi_j(r)\,\chi_m(s)\,\chi_n(s)\, w_{\ell,\ell'}(r,s)$$
where, for the model considered in this section,
\begin{align*}
w_{\ell,\ell'}(r,s)&=\frac12\int_{-1}^1W\left(\sqrt{r^2+s^2-2rst}\right)\, P_\ell(t)\,P_{\ell'}(t)\,dt\\
&=\frac12\int_{-1}^1P_\ell(t)\,P_{\ell'}(t)\bigg(\frac{\kappa}{\sqrt{r^2+s^2-2rst}}\\
&\qquad\qquad\qquad -a_1\,e^{-b_1(r^2+s^2-2rst)}+a_2\,e^{-b_2(r^2+s^2-2rst)}\bigg)\,dt. 
\end{align*}
For $0\leq\ell,\ell'\leq\ell_{\rm max}$ with $\ell_{\rm max}$ not too large, the Gaussian integrals can be computed exactly and it is possible to find the exact expression of $w_{\ell,\ell'}(r,s)$. 

The computation of the integral $(ij|mn)_{\ell,\ell'}$ against hat functions is much more tedious, however. It is easy to find an exact expression for the Coulomb part, but not so simple for the Gaussian part. So we have performed a numerical calculation of these integrals. Since we have of the order of $(N_b)^4$ integrals, we could not take $N_b$ too large. The results of the previous section indicated that the existence of pairing effects does not depend very much on the size of the basis.

\subsubsection{Slow convergence and oscillations of Roothaan}

We have observed that the Roothaan algorithm \emph{almost always oscillates}, even in the HFB case (see some  examples in Figures~\ref{fig:RTHvsODA_a35},~\ref{fig:RTHvsODA_a20} and~\ref{fig:RTHvsODA_a20_kappa0}). This is in stark contrast with the results of the previous section where the Roothaan algorithm was almost always converging. Sometimes it very slowly converges in the HF case (see, e.g., Figure~\ref{fig:RTHvsODA_a20}). However we have always obtained convergence for the HF Roothaan algorithm when $a$ is small enough, that is, when it is expected that there is actually no binding. For the case displayed in Figure~\ref{fig:RTHvsODA_a20} we have $a=20$ but the critical $a$ is about $\simeq24$ (see the next section).

We conclude that using the ODA is very important for such attractive potentials. The same might be true with the more involved forces used in nuclear physics.

\begin{figure}[hp!]

\begin{tabular}{ll}
\hspace{-0.3cm}\includegraphics[width=7.2cm]{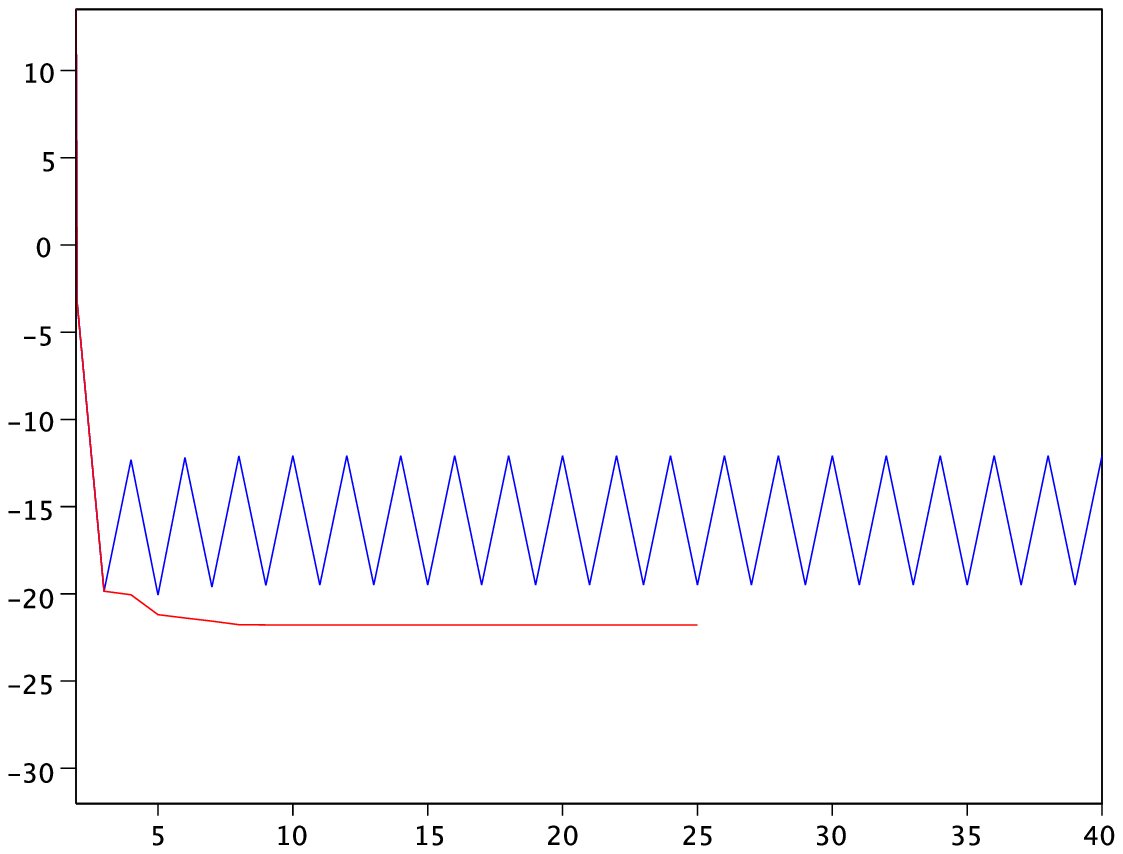}
&\includegraphics[width=7.2cm]{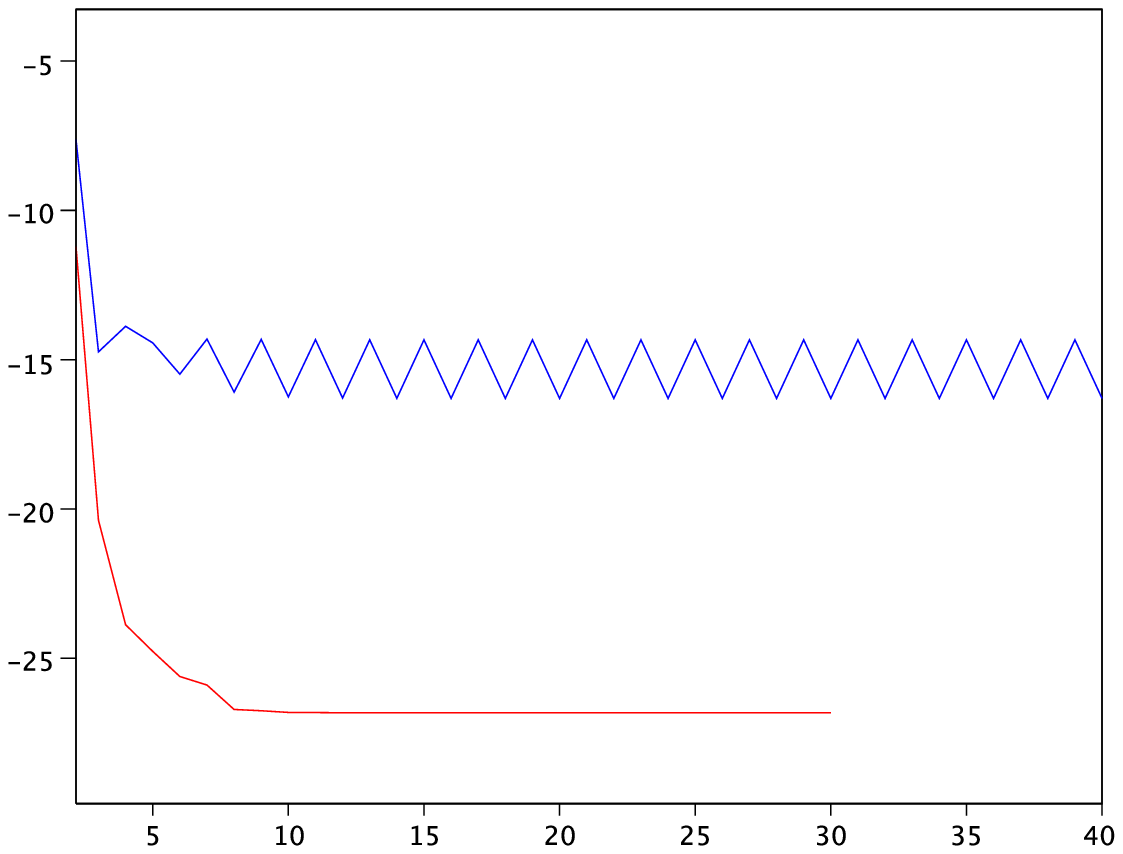}\\
\end{tabular}

\caption{Energy along the iterations in the HF (left) and HFB (right) cases, for the Roothaan (blue) and the ODA (red), with $N=4$, $N_b=20$, $\ell_{\rm max}=1$, $r_{\rm max}=3$, $a=35$ and $\kappa=1$ (proton-proton case).\label{fig:RTHvsODA_a35}}

\begin{tabular}{ll}
\hspace{-0.3cm}\includegraphics[width=7.2cm]{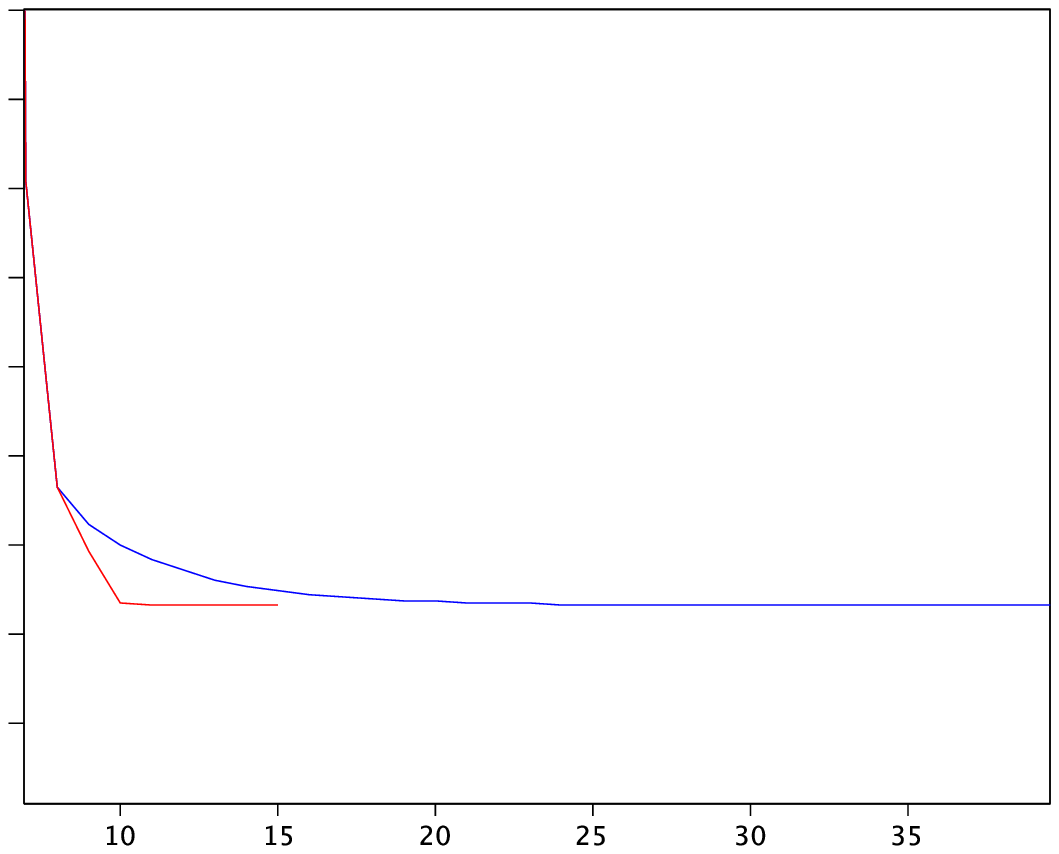}
&\includegraphics[width=7.2cm]{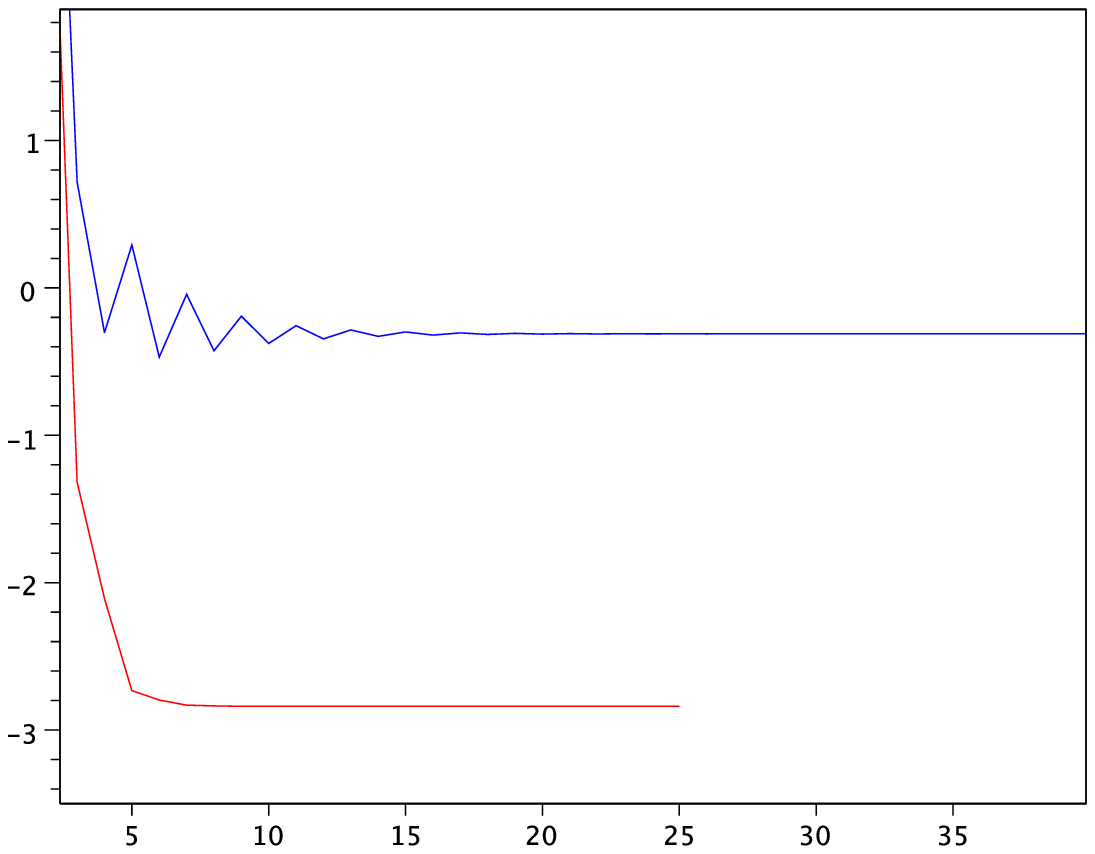}\\
\end{tabular}

\caption{Same calculation with $a=20$  and $\kappa=1$ (proton-proton case). The Roothaan algorithm slowly converges in the HF case and it oscillates in the HFB case but the two values are very close.\label{fig:RTHvsODA_a20}}

\begin{tabular}{ll}
\hspace{-0.3cm}\includegraphics[width=7.2cm]{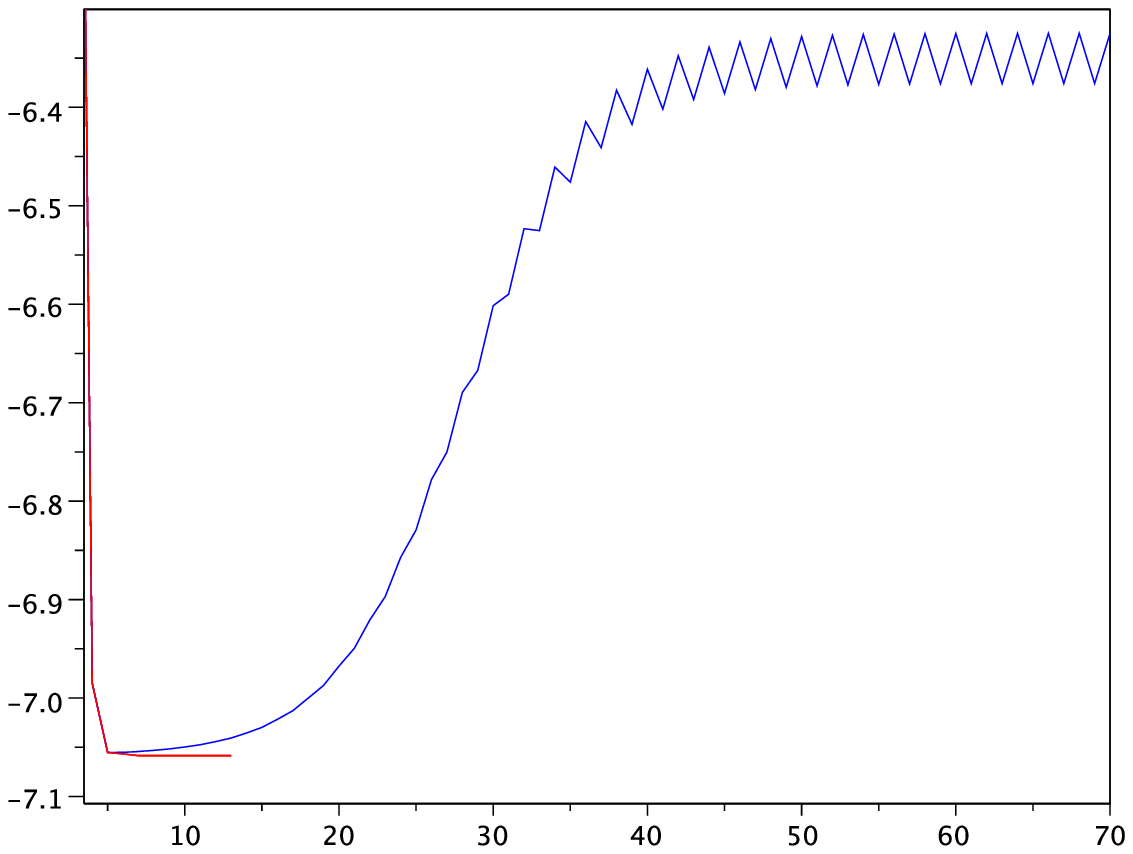}
&\includegraphics[width=7.2cm]{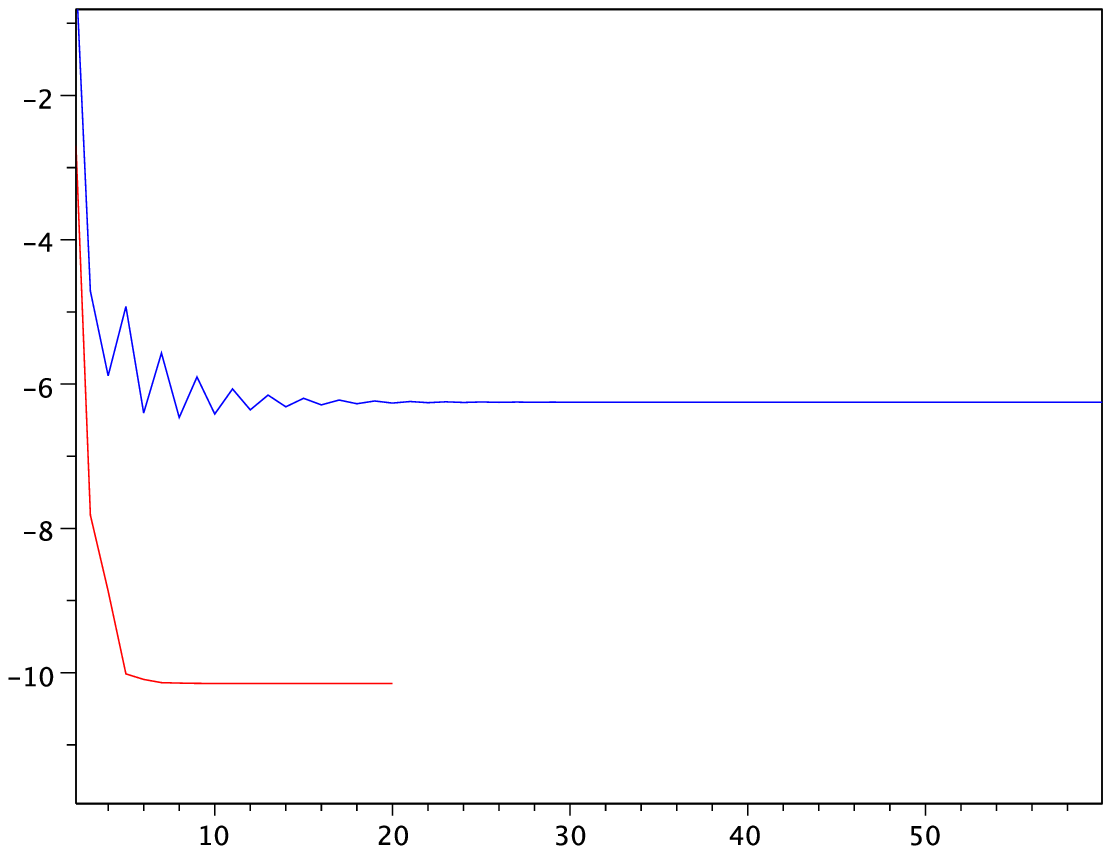}\\
\end{tabular}

\caption{Same calculation with $a=20$  and $\kappa=0$ (neutron-neutron case). The Roothaan algorithm oscillates in the HFB case, but the two values are very close.\label{fig:RTHvsODA_a20_kappa0}}
\end{figure}

\subsubsection{The critical strength}

In finite dimension there is always a minimizer. Saying differently, since the particles are trapped in a ball, they always bind. Furthermore we work with rotation-invariant states. So, for the true model in infinite dimension, the particles escaping to infinity cannot form a bound state of the same kind because they are too far from the (fixed) center of symmetry. In this special case they will spread out and have a vanishing energy.

In Hartree-Fock theory, we conclude that we can detect the loss of binding by looking at the last filled HF eigenvalue. When it crosses 0, this corresponds to the last particle becoming a scattering state. We can therefore choose as definition for the critical strength $a$, the value at which this eigenvalue is 0. In finite dimensional Hartree-Fock-Bogoliubov theory things are less clear and we will not discuss the problem of binding. In our simulations we have observed that the HFB ground state density was always rather close to the HF ground state density, which suggests that there is binding in HFB as well.

We have made some calculations for $N_b=50$ and $N=4$. We found that the critical strength is about $a_c\simeq 23.5$ in the proton-proton case and $a_c\simeq 17.5$ in the neutron-neutron case. In Figure~\ref{fig:vary_a} we display the HF and HFB energies as functions of the parameter $a$, for $N=4$ and $\kappa=1$ (proton-proton case). Figure~\ref{fig:vary_a_kappa0} is the equivalent result for $\kappa=0$ (neutron-neutron case). For these calculations we have chosen $r_{\rm max}=3$ which is the optimal choice for $a$ in a neighborhood of the critical value. Like in the previous section the results depend on the radius of the ball in which the system is confined. We see that there is always pairing, in the sense that the HFB curve is below the HF curve. This is even more manifest in the neutron-neutron case for which the potential is much more attractive than for protons, which repel with the Coulomb potential. Also, the norm of the pairing matrix $A$ does not vary too much with $a$, it stays between $0.80$ and $0.95$ for $a$ in the range $15\leq a\leq 30$, for both $\kappa=0$ and $\kappa=1$.

From these numerical results we can conclude that pairing seems to happen in this model, for any strength $a$ for which there is binding in Hartree-Fock theory. It is an interesting problem to actually prove that pairing always occurs, for instance for $a$ large enough. We are not aware of any result of this kind.

\begin{figure}[hp!]
\begin{tabular}{ll}
\hspace{-0.3cm}\includegraphics[width=7.2cm]{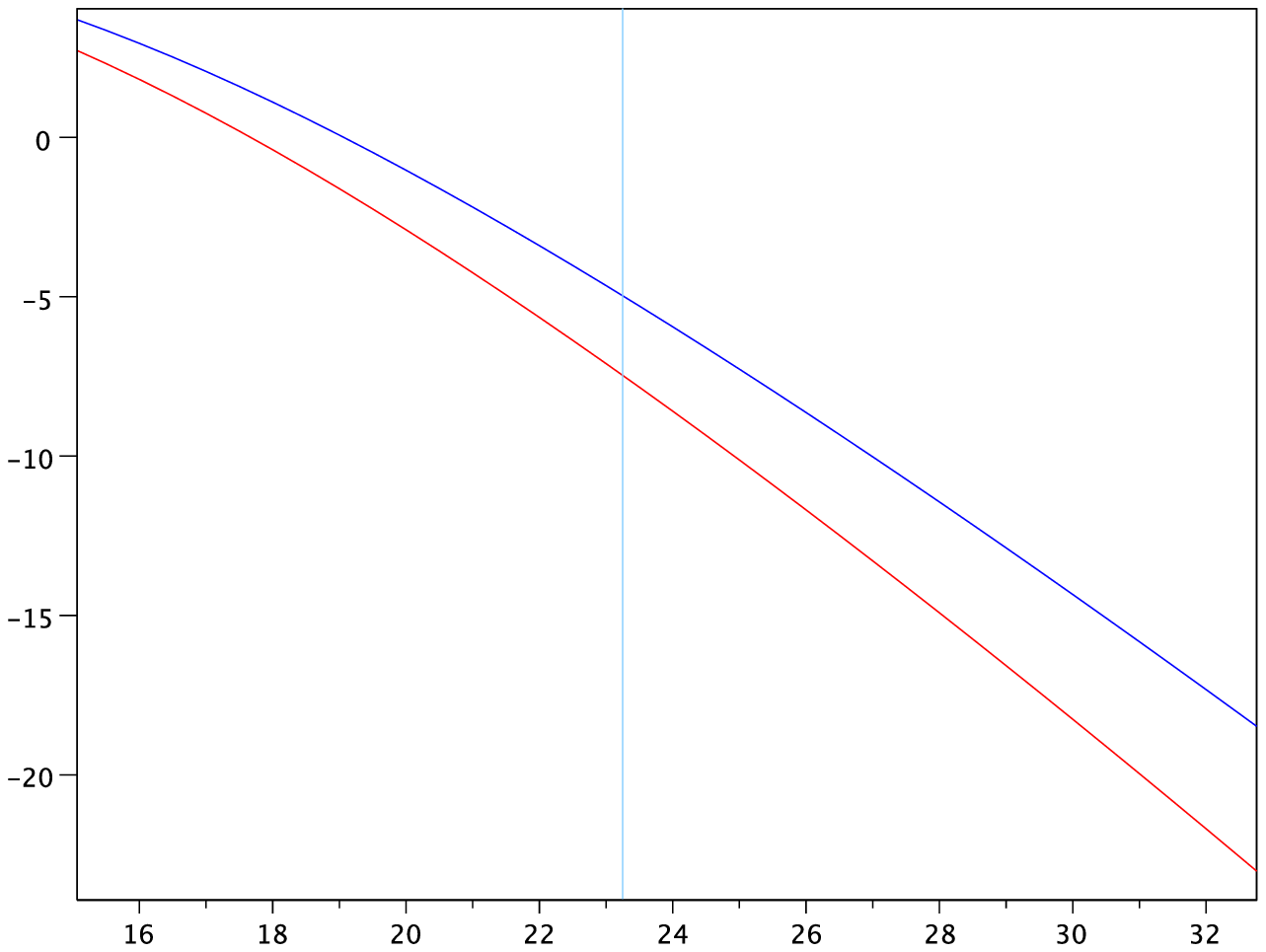}
&\includegraphics[width=7.2cm]{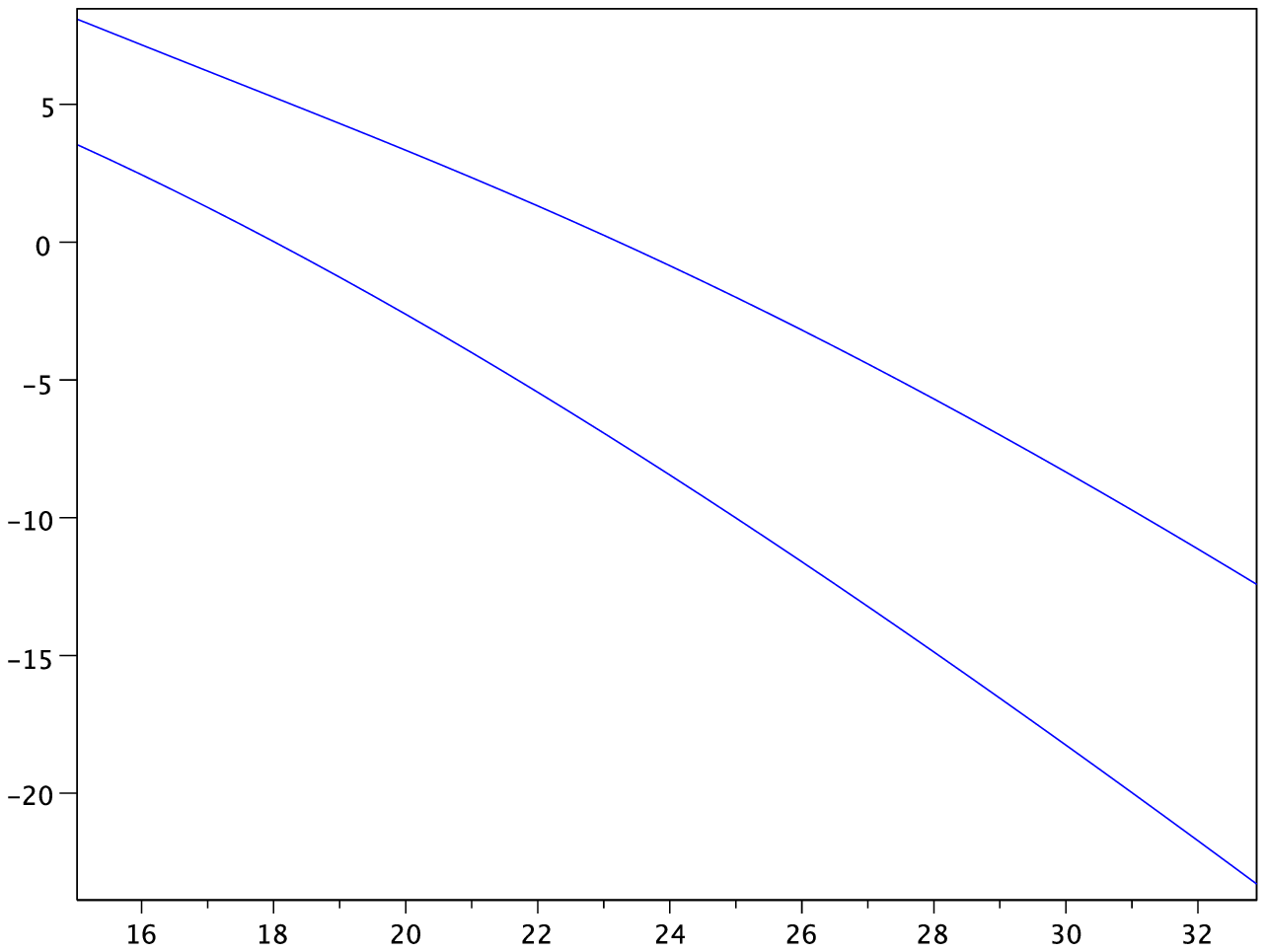}\\
\end{tabular}

\caption{Left: Values of the HF (blue) and HFB (red) ground state energies as functions of $a$, with $N=4$, $N_b=50$, $\ell_{\rm max}=1$, $r_{\rm max}=3$ and $\kappa=1$ (proton-proton case). The vertical line is the value of $a$ for which the last filled eigenvalue vanishes. Right: Values of the two filled HF eigenvalues for the same $a$.\label{fig:vary_a}}

\begin{tabular}{ll}
\hspace{-0.3cm}\includegraphics[width=7.2cm]{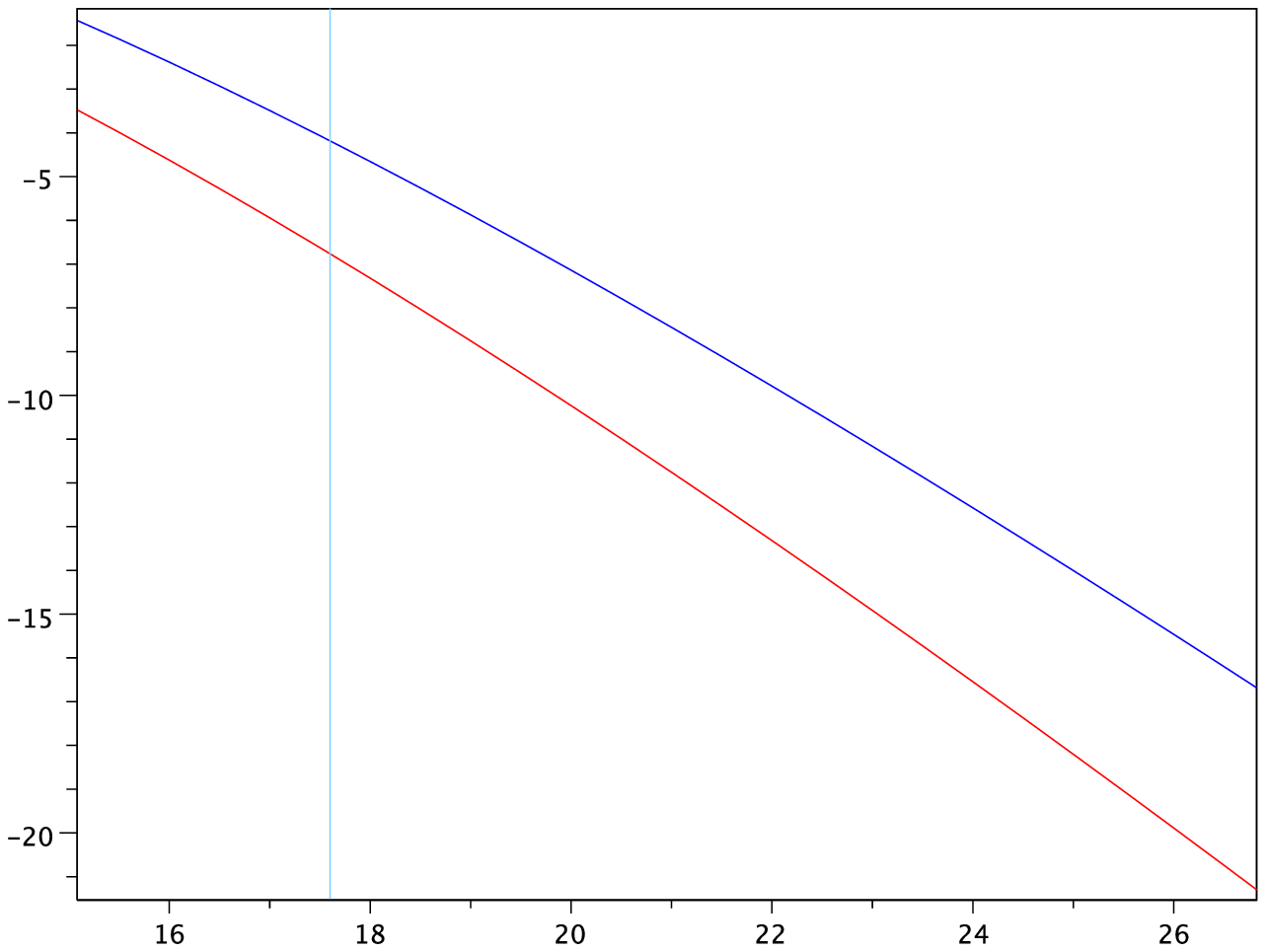}
&\includegraphics[width=7.2cm]{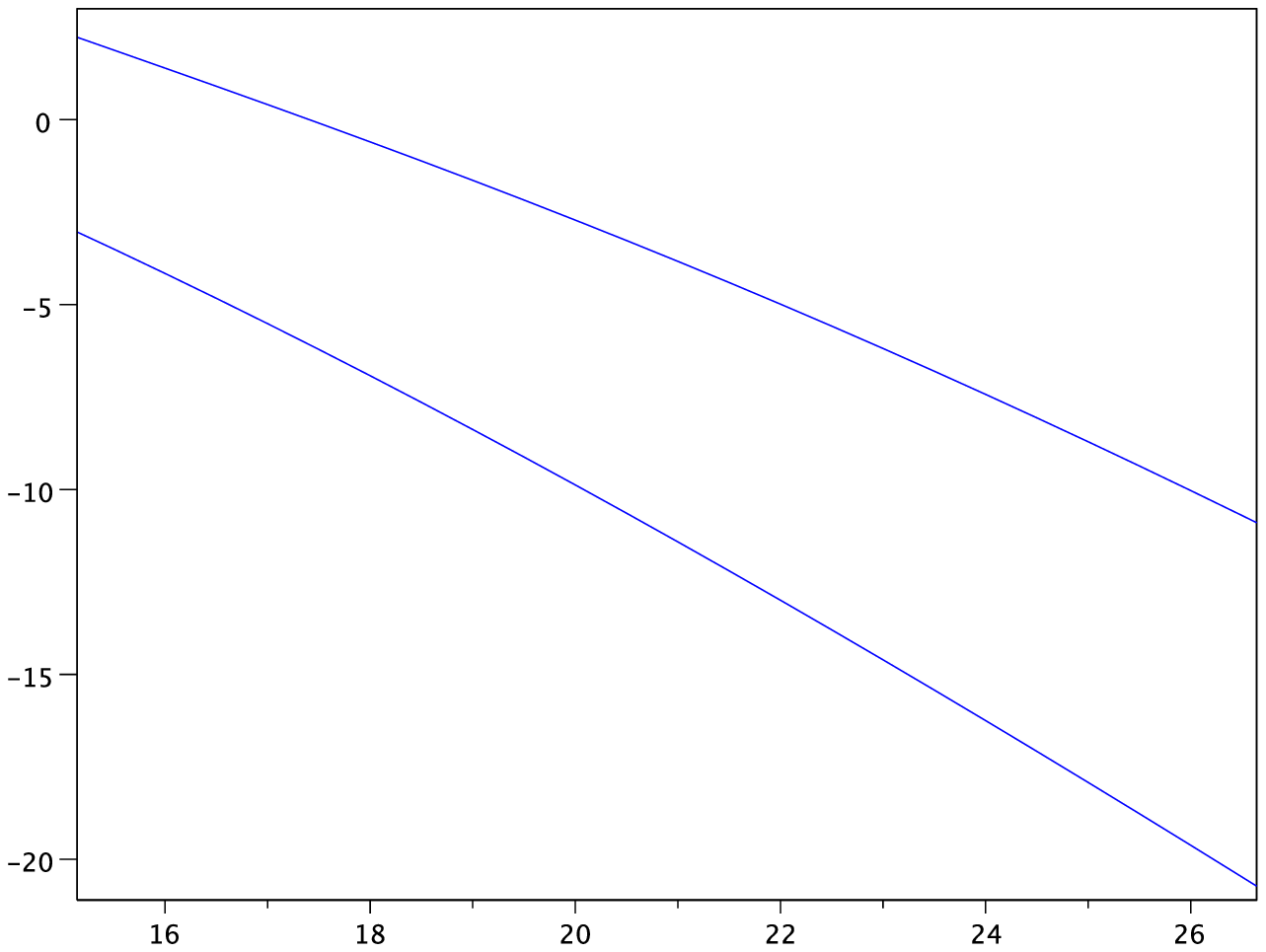}\\
\end{tabular}

\caption{Same calculations for $\kappa=0$ (neutron-neutron case).\label{fig:vary_a_kappa0}}
\end{figure}


\end{document}